\documentclass[11pt, letterpaper]{amsart}
\usepackage{amsmath}
\usepackage{amssymb}
\usepackage[left=1in,right=1in,bottom=0.8in,top=0.8in]{geometry}
\usepackage{amsfonts}
\usepackage{graphicx}
\usepackage[font=small,labelfont=bf]{caption}
\usepackage{amsthm}
\usepackage[T1]{fontenc}
\usepackage{mathrsfs}
\usepackage{mathtools}
\usepackage{stmaryrd}
\SetSymbolFont{stmry}{bold}{U}{stmry}{m}{n}
\usepackage{enumerate}
\usepackage[dvipsnames]{xcolor}
\usepackage{tikz}
\usepackage{tikz-cd}
\usetikzlibrary{decorations.pathmorphing}
\usepackage{rotating}
\usepackage{url}
\usepackage{hyperref}
\usepackage[
backend=biber,
style=alphabetic,
maxbibnames=99,
]{biblatex}

\makeatletter
\renewcommand{\tocsection}[3]{%
  \indentlabel{\@ifnotempty{#2}{\bfseries\ignorespaces#1 #2.\quad}}\bfseries#3}
  \renewcommand{\tocsubsection}[3]{%
  \indentlabel{\hspace{2.5em}\@ifnotempty{#2}{\ignorespaces#1 #2.\quad}}#3}
\makeatother

\newtheorem{theorem}{Theorem}[section]
\newtheorem{corollary}[theorem]{Corollary}
\newtheorem{proposition}[theorem]{Proposition}
\newtheorem{lemma}[theorem]{Lemma}

\newtheorem*{proposition*}{Proposition}

\theoremstyle{definition}
\newtheorem{definition}[theorem]{Definition}
\newtheorem{notation}[theorem]{Notation}

\theoremstyle{remark}
\newtheorem{remark}[theorem]{Remark}

\newtheorem{innercustomcon}{Conjecture}
\newenvironment{customcon}[1]
  {\renewcommand\theinnercustomcon{#1}\innercustomcon}
  {\endinnercustomcon}

\definecolor{refblue}{RGB}{30,90,180}
\definecolor{signal}{RGB}{203,23,206}
\hypersetup{
    colorlinks=true,
    linkcolor=refblue,
    citecolor=signal,
    urlcolor=signal,
    pdftitle={A dimension bound towards the Casas-Alvero conjecture},
    pdfauthor={Soham Ghosh},
    pdfsubject={A dimension bound towards the Casas-Alvero conjecture},
    pdfkeywords={Casas-Alvero, regular sequence, Hasse-Schmidt derivation}
}
\numberwithin{equation}{section}
\makeatletter
\@namedef{subjclassname@2020}{%
  \textup{2020} Mathematics Subject Classification}
\makeatother

\DeclareMathOperator{\Res}{\ensuremath{Res}}

\DeclareMathOperator{\rank}{\ensuremath{rank}}

\DeclareMathOperator{\spec}{\ensuremath{Spec}}

\DeclareMathOperator{\algk}{\ensuremath{Alg_\mathbb{K}}}
\DeclareMathOperator{\leff}{\ensuremath{less}}
\DeclareMathOperator{\lred}{\ensuremath{lred}}
\DeclareMathOperator{\lin}{\ensuremath{lin}}
\DeclareMathOperator{\cdim}{\ensuremath{codim}}
\DeclareMathOperator{\disc}{\ensuremath{Disc}}

\DeclareMathOperator{\Spec}{Spec}
\DeclareMathOperator{\Proj}{Proj}

\DeclareMathOperator{\Des}{Des}

\DeclareMathOperator{\htop}{ht}

\DeclareMathOperator{\length}{length}

\newcommand{\ch}{\mathfrak{C}}

\newcommand{\K}{\mathbb{K}}
\newcommand{\A}{\mathbb{A}}

\newcommand{\bfa}{\boldsymbol{a}}
\newcommand{\Gw}[2]{\Phi_{#1,#2}}

\author[Soham Ghosh]{Soham Ghosh}
\address{Department of Mathematics, University of Washington, Seattle, WA 98195, USA}
\email{soham13@uw.edu}
\title{A dimension bound towards the Casas-Alvero conjecture}
\keywords{Casas-Alvero conjecture, complete intersection, Hasse-Schmidt derivations, regular sequences}
\subjclass[2020]{12E05, 13C15, 13N15, 14G05, 14M10}
\date{October 12, 2026}

\begin{document}
\mathtoolsset{showonlyrefs}

\begin{abstract}
The Casas-Alvero conjecture predicts that every univariate polynomial over an algebraically closed field of characteristic zero sharing a common factor with each of its Hasse-Schmidt derivatives is a power of a linear polynomial. The conjecture for polynomials of a fixed degree $n$ is equivalent to the projective variety of such polynomials being one-dimensional. In this paper, we produce dimension bounds for this variety (in degree $n$) over any algebraically closed field of characteristic zero using an arithmetic invariant $B(n)$ of $n$. Consequently for infintely many degrees $n$, we show that the associated arithmetic Casas-Alvero scheme has finitely many rational points over any characteristic zero field. Along the way, we prove several rigidity results towards the conjecture and study arithmetic properties of $B(n)$.

\end{abstract}
\maketitle
\tableofcontents

\section{Introduction}\label{sec1}
\renewcommand*{\thetheorem}{\Alph{theorem}}
\renewcommand*{\thecorollary}{\Alph{theorem}}

\subsection{Aim and main results of the paper}\label{subsec1.1}
Let $\mathbb{K}$ be a field and $f(X)\in \mathbb{K}[X]$ be a monic polynomial of degree $n>1$. Let $f^{(i)}(X):=d^if(X)/dX^i$ be the $i^{th}$ formal derivative of $f(X)$ with respect to $X$ and let $f_i(X)$ be the $i^{th}$ Hasse--Schmidt derivative of $f(X)$. Over fields of characteristic $0$, the two derivatives are related via $f_i(X)=f^{(i)}(X)/i!$. This paper is concerned with the following question posed by E. Casas-Alvero in connection with his work \cite{CA} on higher-order polar germs:
 
\begin{customcon}{CA}[Casas-Alvero, 2001]\label{con1}
    Let $f(X)$ be a monic univariate polynomial of degree $n$ over a field $\mathbb{K}$ of characteristic $0$. Then $\gcd(f, f_{i})$ is non-trivial for each $i=1, \dots, n-1$ if and only if $f(X)=(X-\alpha)^n$ for some $\alpha\in \mathbb{K}$.
\end{customcon}

Conjecture~\ref{con1} is a statement about fields of characteristic $0$; in positive characteristic the corresponding statement is false in general (see Section~\ref{subsec1.2}). Nevertheless, reduction modulo primes is the main tool of \cite{GVB} and \cite{DJ}, and also of this paper, and we therefore use the following terminology over a field $\mathbb{K}$ of arbitrary characteristic. A monic polynomial $f(X)\in\mathbb{K}[X]$ of degree $n$ satisfies the \textit{Casas-Alvero condition} if $\gcd(f, f_i)$ is non-trivial for each $i=1,\dots, n-1$, and \textit{Conjecture~\ref{con1} holds in degree $n$ over $\mathbb{K}$} if every $f(X)\in\mathbb{K}[X]$ of degree $n$ satisfying the Casas-Alvero condition is of the form $(X-\alpha)^n$ with $\alpha\in\mathbb{K}$. For algebraically closed $\mathbb{K}$, the latter depends only on $n$ and on the characteristic $p$ of $\mathbb{K}$ (see Section~\ref{subsec1.2}), and we then say that Conjecture~\ref{con1} holds in degree $n$ in characteristic $p$.

In \cite{GVB}, the authors prove Conjecture~\ref{con1} over fields of characteristic $0$ in degrees $n=p^k$ and $2p^k$ for any prime $p$ and $k\in \mathbb{N}$ by reformulating the problem over any field $\mathbb{K}$ (irrespective of characteristic) as the absence of $\mathbb{K}$-rational points on a certain weighted projective $\mathbb{Z}$-scheme $X_n\subseteq \mathbb{P}_{\mathbb{Z}}(1,2,\dots, n-1)$. We refer to $X_n$ as the $n^{th}$ \textit{arithmetic Casas-Alvero scheme}. As a consequence of the methods of \cite{GVB}, it follows that, for algebraically closed fields and a fixed degree $n$, Conjecture~\ref{con1} depends only on the characteristic of the field $\mathbb{K}$, and if it is true over a certain characteristic, then it is true in all characteristics except finitely many primes. These results were improved in \cite{DJ}, where the authors used $p$-adic valuation techniques to also prove Conjecture~\ref{con1} (over characteristic $0$) in degrees $3 p^k$ and $4p^k$ for primes $p>3$, excepting degrees $n=4.5^k$ and $4.7^k$.  The case of $n=5p^k$ and corresponding bad primes were studied in \cite{SC}. Furthermore, the Conjecture in degrees $n=6p^k$, $7p^k$ (and $5p^k$ as well) has been verified in \cite{CLO} with the aid of a computer, barring the bad primes in each case, which were also completely identified. 

The aim of this paper is to bound the dimension of the locus of monic polynomials of degree $n$ satisfying the Casas-Alvero condition, and of the geometric fibres of the arithmetic Casas-Alvero scheme $X_n$, without assuming the conjecture.

We introduce the notion of (higher) D-subschemes of affine $\mathbb{K}$-schemes and various refinements (see Definitions~\ref{Def1}, \ref{Def4}, \ref{Def5}) to encode derivations, along with characteristic maps (see Definition~\ref{Def7}) for shift equivalence classes of monic polynomials $f(X)\in\mathbb{K}[X]$. These enable us to reformulate Conjecture~\ref{con1} (for $\mathbb{K}$ algebraically closed) as a complete intersection problem for (higher) discriminant hypersurfaces $\operatorname{Disc}^i_n\subseteq \mathbb{A}^n_{\mathbb{K}}$ ($1\leq i\leq n-1$). We relate these hypersurfaces to a union $X^i_n$ of certain involutions of D-subschemes of reduced principal monomial affine $\mathbb{K}$-schemes by the use of ``Vieta's map" $\nu_n: \mathbb{A}^n_{\mathbb{K}}\rightarrow \mathbb{A}^n_{\mathbb{K}}$.

The restriction of Vieta's map to normalized roots gives the projective coefficient locus of Section~\ref{subsec3.3}. Write
\[
 D_p(n):=\dim\bigcap_{i=1}^{n-1}X_n^i(\overline{\mathbb F}_p)
 \quad(p>0),\quad
 D_0(n):=\dim\bigcap_{i=1}^{n-1}X_n^i(\mathbb C),
\]
where the higher discriminant loci $X_n^i$ are constructed in Section~\ref{sec3}. These dimensions are unchanged by extending algebraically closed fields of the same characteristic. The normalized root locus has dimension $D_p(n)-1$, and the geometric fibre $(X_n)_{\mathbb{K}}:=X_n\times_{\Spec\mathbb{Z}}\Spec\mathbb{K}$ of the arithmetic Casas-Alvero scheme over an algebraically closed field $\mathbb{K}$ of characteristic $p$ satisfies $\dim (X_n)_{\mathbb{K}}=D_p(n)-2$, with $\dim\emptyset=-1$ (see Lemma~\ref{lem:dictionary}). Our first main result is the following.

\begin{theorem}[=Theorem~\ref{PropNew}]\label{mainthm2}
For $n\ge3$, the explicit integer $B(n)$ defined in \eqref{eq:certificate} satisfies
\[
 D_0(n)\le B(n)\le\max\{2,\lfloor(n-1)/2\rfloor\}.
\]
For every prime $p$, $D_0(n)\le D_p(n)$ and
$D_p(p^b m)=D_p(m)$ for $b\ge0$, $m\ge1$ and $p\nmid m$, with $D_p(1)=1$. Moreover, $D_0(n)\le\max\{2,n-p\}$, where $n/2<p\le n$ is a prime. Consequently, $D_0(n)\le n^{21/40}$ for $n\gg1$.
\end{theorem}

The arithmetic input is expressed directly in terms of the base-$p$ digits of the degree. For an integer $v=\sum_\nu v_\nu p^\nu$ expanded in base $p$, let $M_p(v)=|S_p(v)|=\prod_\nu(v_\nu+1)$. Then we prove
\[
 D_p(n)\le\lfloor M_p(n)/2\rfloor,\quad
 D_p(ap^e+s)\le\max\{2,(M_p(a)-1)(M_p(s)-1)\}
 \quad(0<s<p^e).
\]
The proof uses triangular coefficient elimination and evaluation-rank strata. A separate tangent-space argument, valid whenever nonempty sums of root multiplicities are nonzero in the field, supplies the characteristic-zero bound in Theorem~\ref{mainthm2}.

Our second main result describes exactly the degrees in which $B(n)$ is small, and what Theorem~\ref{mainthm2} gives in these degrees. In the following statement, $p$ denotes a prime, $b\geq0$ and $e>i\geq1$ denote integers, and $u$ denotes the prime-to-$p$ part $n/p^{v_p(n)}$ of $n$.

\begin{theorem}[=Proposition~\ref{maincor1}]\label{mainthm}
Let $X_n$ be the arithmetic Casas-Alvero scheme of degree $n\geq3$, and let $\mathbb{K}$ be a field of characteristic zero.
\begin{enumerate}[(i)]
\item $B(n)=1$ if and only if $n=p^b$, $n=2p^b$, or $n=3p^b$ with $p\geq5$. In these degrees $D_0(n)=1$ and $X_n(\mathbb{K})=\emptyset$. In particular, $B(n)=1$ for all $3\leq n\leq11$.
\item Let $n\geq7$. Then $B(n)\leq2$ if and only if $n=p^bu$ with $p\nmid u$ and
\[
 u\in\{1,2,3,4,5,6\}\cup\{p^e+1,\ p^e+2,\ 2p^e+1\}.
\]
If $B(n)\leq2$, then $X_n(\mathbb{K})$ is finite. Consequently, $X_n(\mathbb{K})$ is finite in the degrees
\[
 n=p^b,\ 2p^b,\ 3p^b,\ 4p^b,\ 5p^b,\ 6p^b,\ p^b(p^e+1),\ p^b(p^e+2),\ p^b(2p^e+1),
\]
and for $n\geq7$ these are the only degrees in which Theorem~\ref{mainthm2} yields finiteness.
\item Let $n\geq9$. Then $B(n)=3$ if and only if $n$ is not of the form described in (ii) and $n=p^bu$ with $p\nmid u$ and
\[
 u\in\{7,8\} \ \text{ with } u<p, \ \ \ \text{ or } \ \ \ u\in\{p^e+3,\ 3p^e+1\} \ \text{ with } p\geq5, \ \ \ \text{ or } \ \ \ u=p^e+p^i+1.
\]
If $B(n)=3$, then $\dim X_n(\overline{\mathbb{K}})\leq1$ for an algebraic closure $\overline{\mathbb{K}}$ of $\mathbb{K}$.
\end{enumerate}
\end{theorem}

Thus, over an algebraically closed characteristic-zero field, the degrees in (ii) admit only finitely many non-pure-power polynomials satisfying the Casas-Alvero condition, up to affine transformations $f(X)\mapsto a^{-n}f(aX+b)$. We note that among the degrees in (i), (ii) are the cases of Conjecture~\ref{con1} established in \cite{GVB} and \cite{DJ}, which Theorem~\ref{mainthm2} thus recovers, and the finiteness in degrees $p^b(p^e+1)$ and $p^b(p^e+2)$, which was also obtained by Massri~\cite{Massri}.

The invariant $B(n)$ is defined in \eqref{eq:certificate} as the minimum of $V(n)=\max\{2,\lfloor(n-1)/2\rfloor\}$ and of explicit integers $B_p(n)$, one for each prime $p\leq n$. Each $B_p(n)$ is computed from the base-$p$ digits of the prime-to-$p$ part $u$ of $n$, by applying the two displayed bounds in degree $u$, together with the bound $V(u)$ when $u<p$ (see Definition~\ref{def:Bp} for the precise list). Since $D_p(n)=D_p(u)\leq B_p(n)$ and $D_0(n)\leq D_p(n)$, the invariant $B(n)$ records, uniformly in $p$, the information that reduction modulo $p$ provides about the characteristic-zero locus, and it is readily computed from the factorization of $n$. It is therefore natural to ask in which degrees $B(n)$ is small, and how it is distributed. Theorem~\ref{mainthm} answers the first question for $B(n)\leq3$; for instance, $B(n)\leq2$ for all $3\leq n\leq100$ except $n=70$, $77$ and $88$, where $B(n)=3$. The second question is addressed by our third main result: the degrees in which $B(n)$ is bounded by a fixed constant are, up to a constant factor, as numerous as the primes, but no more. Let $\pi(N)$ be the number of primes $p\leq N$.

\begin{theorem}[=Proposition~\ref{prop:density}]\label{mainthmC}
For every integer $C\geq1$ there is a constant $c_C>0$ such that, for all $N\geq2$,
\[
 \pi(N)-1\leq\#\{3\leq n\leq N\mid \ B(n)\leq C\}\leq c_C\,\frac{N}{\log N}.
\]
In particular, the number of degrees $n\leq N$ with $B(n)\leq C$ lies between $N/(6\log N)-1$ and $c_CN/\log N$, the set of these degrees has natural density zero, and $B(n)$ is unbounded.
\end{theorem}

The lower bound is attained by the primes, where $B(p)=1$. For the upper bound we show that $B(n)\leq C$ forces $n=p^bu$ with $M_p(u)\leq(C+1)^2$ for some prime $p$, and count the integers $u\leq N$ with so few nonzero base-$p$ digits. Consequently, the finiteness statements obtained from Theorem~\ref{mainthm2} concern a set of degrees of density zero, and the bound $D_0(n)\leq B(n)$ is not uniform in $n$; the uniform estimate $D_0(n)\leq n^{21/40}$ of Theorem~\ref{mainthm2} rests instead on the prime gap bound of Baker, Harman and Pintz~\cite{BHP}.

Let $m=n-1$, and for each $1\leq d\leq m$ let $1\leq a_d\leq n$ and
\[
 S=\K[x_1,\ldots,x_m],\quad x_n=0,\quad
 \Gw{a_d}{d}=e_d\bigl((x_l-x_{a_d})_{l\ne a_d}\bigr),\quad
 I_{\bfa}=(\Gw{a_1}{1},\ldots,\Gw{a_m}{m}).
\]
Here $e_d$ is elementary symmetric of degree $d$ (see Notation~\ref{not:Phi}), and $w_r=|\{a_1,\ldots,a_r\}|$, $w=w_m$ count witness labels, with collisions allowed.

\begin{theorem}[=Theorems~\ref{thm:descents}, \ref{thm:prefix} and Corollary~\ref{cor:endblocks}]\label{mainthm5}
Assume $\operatorname{char}\K=0$.
\begin{enumerate}[(i)]
\item If $r+w_r\le n+1$, then $\Gw{a_1}{1},\ldots,\Gw{a_r}{r}$ is regular. In particular every word has a regular prefix of length $\lfloor(n+1)/2\rfloor$, and
$\dim S/I_{\bfa}\le\max\{0,w-2\}$.
\item For a nonconstant word with initial and terminal constant blocks of lengths $\ell$ and $t$ in increasing-degree order,
\[
 \dim S/I_{\bfa}\le
 \max\left\{1,\left\lfloor\frac{n-t-\ell-1}{2}\right\rfloor\right\}.
\]
In particular $\ell+t\ge n-4$ implies $\operatorname{ht}I_{\bfa}\ge n-2$.
\end{enumerate}
\end{theorem}

The coefficient formula underlying (i) is a signed descent-set enumeration. The finer estimate behind (ii) is the inequality $2\dim Z\le r-\kappa$ of Theorem~\ref{m:nr}(ii), valid for every irreducible closed cone $Z\subseteq V(I_{\bfa})$ of dimension at least $2$, where $r$ is the number of distinct root-coordinate functions on $Z$ and $\kappa$ is its first nonzero coefficient index, i.e., the least $d$ with $e_d\notin I(Z)$.

The regular prefix of Theorem~\ref{mainthm5}(i), combined with a degree estimate for the finite projective fibers (Theorem~\ref{thm:degree}), yields an explicit bound for the number of $\mathbb{K}$-rational points of $X_n$ whenever Theorem~\ref{mainthm2} gives finiteness, in particular in all degrees $n\geq4$ with $B(n)\leq2$, which are listed in Theorem~\ref{mainthm}(ii).

\begin{corollary}[=Corollary~\ref{corexplicit}]\label{maincorE}
If $D_0(n)\leq2$ and $n\geq4$, then for every field $\mathbb{K}$ of characteristic zero,
\[
 |X_n(\mathbb{K})|\leq\frac{n^{n-1}(n-1)!}{\lfloor(n+1)/2\rfloor+1}.
\]
\end{corollary}

\subsection{Existing results and methods}\label{subsec1.2}

Conjecture~\ref{con1} was originally posed over characteristic $0$ fields $\mathbb{K}$, and it is known that the conjecture holds over $\mathbb{C}$ if and only if it holds over all such fields $\mathbb{K}$. The corresponding statement does not hold in general over fields of positive characteristic: for instance, in characteristic $p$ the polynomial $f(X)=X^{p+1}-X^p=X^p(X-1)$ satisfies the Casas-Alvero condition, since $f_1=X^p$, $f_i=0$ for $2\leq i\leq p-1$ and $f_p=X-1$, but it is not of the form $(X-\alpha)^{p+1}$ (cf. \cite{GVB} for further counterexamples). Furthermore, they show that the truth of Conjecture~\ref{con1} over algebraically closed fields in a particular degree $n$ is independent of the choice of field, and only depends on its characteristic. The first progress over characteristic $0$ was made for degree $\leq 7$ polynomials via computational methods in \cite{TV}. Soon after, in \cite{GVB} the authors related the conjecture for a general polynomial $P(X)$ over any field $\mathbb{K}$ to the absence of $\mathbb{K}$-rational points of a weighted projective $\mathbb{Z}$-subscheme of $\mathbb{P}_{\mathbb{Z}}(1, 2, \dots, n-1)$ defined by vanishing of resultants $\Res_X(P, P_i)$ for all $1\leq i<\deg P$. Their methods utilizing reduction modulo prime arguments, successfully proved the conjecture over fields of characteristic $0$ for polynomials of degree $p^k$, $2p^k$ for any prime $p$. Furthermore, by \cite[Proposition 2.2]{GVB}, if Conjecture~\ref{con1} holds for all polynomials of degree $n$ over fields of characteristic $p$, for some $p=0$ or $p$ prime, then Conjecture~\ref{con1} holds true for degree $n$ polynomials over fields of any characteristic except finitely many primes. These methods were reformulated and extended in \cite{DJ} using $p$-adic valuations, where the authors proved Conjecture~\ref{con1} over characteristic $0$ for polynomials of degrees $3 p^k$ and $4p^k$ for primes $p$ greater than $3$ and $4$ respectively (except $p=5, 7$ when $n=4$). These results are particular instances of the following propositions.

\begin{proposition*}(Proposition~$2.6$, \cite{GVB}).
    For any positive integer $d$ and prime number $p$, if Conjecture~\ref{con1} holds in degree $d$ and characteristic $p$, then Conjecture~\ref{con1} also holds in degree $d p^k$ for any integer $k\geq 0$ in characteristics $p$ and $0$.
\end{proposition*}

\begin{proposition*}(Proposition~$9$, \cite{DJ}).
Let $\nu_p$ be an extension to $\mathbb{C}$ of the $p$-adic valuation on $\mathbb{Q}$ for some prime $p$. Consider the local ring $R:=\{z\in \mathbb{C}\mid \ \nu_p(z)\geq 0\}$ with maximal ideal $M=\{z\in \mathbb{C}\mid \ \nu_p(z)>0\}$ and let $\mathbb{K}_p=R/M$ be its residue field. Suppose that $n=n'p^e$ with $n'<p$ and that Conjecture~\ref{con1} holds for all polynomials in $\mathbb{K}_p[X]$ of degree $n'$. Then the conjecture holds for all polynomials of degree $n$ in $\mathbb{C}[X]$. 
\end{proposition*}

In this spirit, further results for certain polynomials of degree $5 p^k$ were shown in \cite{SC}. Recently, Spivakovsky and Schaub \cite{DM} have extended the results of \cite{GVB} by providing a non-exhaustive list of bad primes $p$ for each $n>1$, i.e., primes $p$ such that Conjecture~\ref{con1} does not hold in degree $n$ in characteristic $p$, and have subsequently produced further results \cite{DM25} towards Conjecture~\ref{con1}. There have also been several computational studies, cf. \cite{CLO}, where the authors verified the conjecture for polynomials of degrees $n=5p^k$, $6p^k$ and $7p^k$ barring the bad primes in each case, which were also completely classified. They also studied obstructions to hypothetical counterexamples and have verified the conjecture for degree $12$ (which is missed by the cases considered above). Computational approaches involving Gröbner bases, even though theoretically possible, get practically infeasible for large degrees due to the complexity of resultants. Alternate approaches to the conjecture have involved analytic tools via the Gauss--Lucas theorem. However, analytic approaches so far have been successful only in very low degrees as demonstrated in \cite{DJ}. We refer the readers to the references of the cited articles for further literature on the conjecture.

\subsection{Organization of the paper}\label{subsec1.3}
Sections~\ref{sec2} and~\ref{sec3} develop the D-subscheme and higher-discriminant constructions. The reduction to translated elementary symmetric systems is Proposition~\ref{Lem5}. Section~\ref{subsec:spec} proves the arithmetic and component estimates and Theorem~\ref{mainthm2}; the invariant $B(n)$ is introduced in Section~\ref{subsubsec:Bn}, where Theorems~\ref{mainthm} and~\ref{mainthmC} are proved. Section~\ref{subsec:splitting} establishes the reduction to the witness coordinates and the regular prefixes, and Section~\ref{subsec:height} proves the end-block bound and the degree estimate, which complete the proofs of Theorem~\ref{mainthm5} and Corollary~\ref{maincorE}.

\subsection*{Acknowledgement}
The author would like to thank Mark Spivakovsky and Daniel Schaub for several enlightening discussions, comments and suggestions at various stages of the work since 2024. He would also like to thank Max Lieblich, Farbod Shokrieh, Sándor Kovács, Utsav Choudhury, and Apoorva Khare for useful discussion and feedback at various stages. The author was partially supported by NSF CAREER DMS-2044564 and NSF FRG DMS-2151718 grants and a Simons Dissertation Fellowship SFI-MPS-SDF-00014771.

\subsection*{AI Statement}
The previous version of this work was concluded in January 2025, long before AI tools reached their current mathematical capabilities. Recently, the author was made aware of a subtle error in a technical lemma in Section 5 of that version, found by frontier large language models towards the end of July 2026, which left a gap in the proof of one of the main results. Some of the results of Section 5 of this current version have been since obtained by the author through a sustained interaction with several frontier large language models over the course of past few months. 

\section{Global notations and Definitions}\label{secnot}
We list a few notations and definitions that will be used throughout the paper.
\begin{enumerate}[(i)]
    \item For a univariate polynomial $f(x)=a_{n}x^{n}+a_{n-1}x^{n-1}+\dots+a_0\in \mathbb{K}[x]$ over any field $\mathbb{K}$, we will denote the $i^{th}$ Hasse--Schmidt derivative (introduced in \cite{HS}) of $f(x)$ by $f_i(x)$, which is defined as \[f_i(x)=\binom{n}{i}a_nx^{n-i}+\binom{n-1}{i}a_{n-1}x^{n-i-1}+\dots+\binom{i}{i}a_i.\]
    \item We will denote the $i^{th}$ multivariate Hasse--Schmidt derivation on $\mathbb{K}[x_1,\dots, x_k]$ by $HD^i_k$, which is defined as follows: for a monomial $x_1^{\alpha_1}\dots x_k^{\alpha_k}\in \mathbb{K}[x_1,\dots,x_k]$ define
    \begin{equation}\label{EqnHS1}
    HD_k^ix_1^{\alpha_1}\dots x_k^{\alpha_k}:=\sum_{j_1+\dots+j_k=i}\binom{\alpha_1}{j_1}\binom{\alpha_2}{j_2}\dots \binom{\alpha_k}{j_k}x_1^{\alpha_1-j_1}\dots x_k^{\alpha_k-j_k}.
    \end{equation}
    The derivation $HD^i_k:\mathbb{K}[x_1,\dots,x_k]\rightarrow \mathbb{K}[x_1,\dots, x_k]$ is defined by extending \eqref{EqnHS1} $\mathbb{K}$-linearly. 
    \item For any two univariate polynomials $f(x), g(x)\in \mathbb{K}[x]$, we will denote their classical resultant (see \cite[Chapter 12]{GKZ}) by $\operatorname{Res}(f, g)$.
\end{enumerate}

\renewcommand*{\thetheorem}{\arabic{section}.\arabic{theorem}}
\renewcommand*{\thecorollary}{\arabic{section}.\arabic{theorem}}

\section{D-subschemes and monomial affine \texorpdfstring{$\mathbb{K}$}{K}-schemes}\label{sec2}
\subsection{Preliminaries on D-subschemes}\label{subsec2.1} Let $\mathbb{K}$ be an algebraically closed field of characteristic $0$ unless otherwise mentioned and  $\algk$ be the category of finitely generated commutative unital $\mathbb{K}$-algebras. For a given polynomial ring $\mathbb{K}[x_1, \dots, x_k]$ over $\mathbb{K}$, let $D_k:\mathbb{K}[x_1, \dots, x_k]\rightarrow \mathbb{K}[x_1, \dots, x_k]$ be the $\mathbb{K}$-linear derivation $D_k=\sum_{i=1}^{k}\partial/\partial x_i$. Furthermore, for any integer $i\geq 1$, let \[D^{i}_k:=(D_k\circ D_k\circ\cdots\circ D_k):\mathbb{K}[x_1, \dots, x_k]\rightarrow \mathbb{K}[x_1, \dots, x_k]\] be the map obtained by composing $D_k$ with itself $i$ times.

\begin{definition}[D-ideal and D-subscheme]\label{Def1}
    Let $(X, \mathbb{A}^k_{\mathbb{K}})$ be a pair of an affine $\mathbb{K}$-scheme $X=\spec(A)$ and an affine space $\mathbb{A}^k_{\mathbb{K}}$ into which $X$ embeds as a closed subscheme. Any embedding $X\hookrightarrow \mathbb{A}^k_{\mathbb{K}}$ induces a surjective $\mathbb{K}$-algebra homomorphism $\mathbb{K}[x_1, \dots, x_k]\rightarrow A$, whereby $A=\mathbb{K}[x_1, \dots, x_k]/I$ for some ideal $I\subseteq \mathbb{K}[x_1, \dots, x_k]$.
    \begin{enumerate}[(i)]
        \item Define the \textit{D-ideal} of $I$ to be the ideal $D_k(I):=(\{D_k(f)\mid \ f\in I\})\subseteq \mathbb{K}[X_1, \dots, X_k]$ generated by the image of $I$ under the derivation $D_k$. 
        \item Define the \textit{D-subscheme} of $(X, \mathbb{A}^k_{\mathbb{K}})$ as the pair $(\mathscr{D}X, \mathbb{A}^k_{\mathbb{K}})$, where $\mathscr{D}X$ is the affine $\mathbb{K}$-scheme $\spec(\mathbb{K}[x_1, \dots, x_k]/D_k(I))$.
    \end{enumerate}
\end{definition}

\begin{remark}\label{Rem0}
\begin{enumerate}[(i)]
    \item[] 
    \item Note that for every ideal $I$ that $I\subseteq D_k(I)$, because if $f\in I$, then $f=D_k(x_1f)-x_1D_k(f)\in D_k(I)$. It follows that $\mathscr{D}X$ is indeed a closed subscheme of $X$.
    \item If an ideal $I\subseteq \mathbb{K}[X_1, \dots, X_k]$ is generated by $f_1, f_2, \dots, f_m$, then we claim that $D_k(I)=(f_1, \dots, f_m, D_k(f_1),\dots, D_k(f_m))$. Indeed, given any element $g=h_1f_1+h_2f_2+\cdots +h_mf_m\in I$, we have $D_k(g)=\sum_{i=1}^{m}(D_k(h_i)f_i+h_iD_k(f_i))\in (f_1, \dots, f_m, D_k(f_1),\dots, D_k(f_m))$ by the Leibniz rule for derivations. The reverse inclusion follows by definition of $D_k(I)$ and since $I\subseteq D_k(I)$.
\end{enumerate}
\end{remark} 

\subsubsection{Linear reduction of affine \texorpdfstring{$\mathbb{K}$}{K}-schemes}\label{subsubsec2.1.1} Let $X=\spec(A)$ be an affine $\mathbb{K}$-scheme, with $A=\mathbb{K}[x_1, \dots, x_k]/I$ where $I=(f_1, \dots, f_m)\subseteq \mathbb{K}[x_1, \dots, x_k]$. Assume the generator $f_1(x_1, \dots, x_k)=a_0+a_1x_1+a_2x_2+\cdots+a_kx_k$ is a linear polynomial, in which case $D_k(f_1)=a_1+a_2+\cdots+a_k$ is a scalar. We call such $f_1$ to be \textit{D-degenerate}, and classify them into the following two categories:
\begin{enumerate}[(i)]
\item \textit{Tame D-degeneracy}: This corresponds to the case $D_k(f_1)=a_1+\cdots+a_k=0$. Letting $I'=(f_2, f_3,\dots, f_m)\subseteq \mathbb{K}[X_1,\dots, X_k]$, we have $D_k(I)=(f_1, D_k(I'))$. Thus, $\mathscr{D}X$ is the scheme-theoretic intersection of the hypersurface $\spec(\mathbb{K}[x_1, \dots, x_k]/(f_1))$ and the D-subscheme $\mathscr{D}X'$ of the affine $\mathbb{K}$-scheme $X'=\spec(\mathbb{K}[x_1, \dots, x_k]/I')$ in $\mathbb{A}^k_{\mathbb{K}}$.
\item \textit{Wild D-degeneracy}: This corresponds to the generic case when $D_k(f_1)=a_1+\cdots+a_k\neq0$, in which case $D_k(I)=\mathbb{K}[x_1,\dots, x_k]$ and thus, $\mathscr{D}X$ collapses to the empty subscheme. 
\end{enumerate} 
Thus, generically, linear polynomials in the defining ideal of an affine $\mathbb{K}$-scheme $X$ collapse the D-subscheme $\mathscr{D}X$ (wild case), or are redundant (tame case). This situation can be rectified by the process of \textit{linear reduction} of $X$, which we now describe:

Since $f_1=a_0+a_1x_1+\cdots+a_kx_k$ is non-constant, assume, without loss of generality, that $a_1\neq 0$ and consider the $\mathbb{K}$-algebra homomorphism $\pi_{f_1}:\mathbb{K}[x_1, \dots, x_k]\rightarrow \mathbb{K}[x_2, \dots, x_k]$ given by 
\[x_1\mapsto -(a_0+a_2x_2+a_3x_3+\cdots+a_kx_k)/a_1 \quad \text{and } x_i\mapsto x_i \text{ for all } i\geq 2.\]
Geometrically, this map induces the inclusion of $\mathbb{A}^{k-1}_{\mathbb{K}}$ into $\mathbb{A}^k_{\mathbb{K}}$ by identifying it with the hyperplane $V(f_1)\subseteq \mathbb{A}^k_{\mathbb{K}}$. Furthermore, $\ker\pi_{f_1}=(f_1)$ and thus, we obtain the isomorphism of $\mathbb{K}$-algebras \[A=\mathbb{K}[x_1, \dots, x_k]/(f_1, \dots, f_m)\cong \mathbb{K}[x_2, \dots, x_k]/(\pi_{f_1}(f_2), \pi_{f_1}(f_3), \dots, \pi_{f_1}(f_m))=:A_1.\]
It follows that, although $X_1:=\spec(A_1)$ is isomorphic to $X:=\spec(A)$ as affine $\mathbb{K}$-schemes, $X_1$ comes with a natural embedding into $\mathbb{A}^{k-1}_{\mathbb{K}}$. We define the linear reduction of $(X, \mathbb{A}^k_{\mathbb{K}})$ with respect to $f_1\in I$ as the pair $(X_1, \mathbb{A}^{k-1}_{\mathbb{K}})$. Note that we can define the linear reduction of $X=\spec(\mathbb{K}[x_1, \dots, x_k]/I)$ with respect to any linear polynomial $f_1\in I$, by extending $\{f_1\}$ to a generating set for $I$. This process is well-defined, i.e., independent of the choice of generators of $I$, since the linear reduction of $X$ with respect to $f_1\in I$ is the affine scheme $\spec(\mathbb{K}[x_2, \dots, x_k]/\pi_{f_1}(I))$ with its embedding in $\mathbb{A}^{k-1}_{\mathbb{K}}$. 

The process of linear reduction can be iterated until one obtains an affine $\mathbb{K}$-scheme $X_{\lred}=\spec(\mathbb{K}[x_1, \dots, x_q]/I_{\lred})$ isomorphic to $X$, such that $I_{\lred}$ does not contain any non-constant linear polynomials. This process terminates after finitely many steps since after each linear reduction we obtain an affine scheme isomorphic to $X$ embedding into an affine space of one lower dimension. The following definitions formalize this sequence of linear reductions algebraically.

\begin{definition}[Linear sequences]\label{Def2}
\begin{enumerate}[(i)]
    \item A sequence of polynomials $f_1, \dots, f_r$ in $\mathbb{K}[x_1,\dots, x_k]$ is said to form a linear sequence if $f_1$ is a non-constant linear polynomial in $\mathbb{K}[x_1,\dots, x_k]$ and for each $2\leq i\leq r$, the residue of $f_i$ in $\mathbb{K}[x_1,\dots, x_k]/(f_1, \dots, f_{i-1})$ is a non-constant linear polynomial. 
    \item For an ideal $I\subseteq \mathbb{K}[x_1,\dots,x_k]$, a sequence of polynomials $f_1, f_2, \dots, f_r\in \mathbb{K}[x_1,\dots, x_k]$ is said to form an $I$-linear sequence if the $f_i$ form a linear sequence and each $f_i$ belongs in $I$.
\end{enumerate}  
\end{definition}

Although, for a linear sequence $f_1, \dots, f_r \in \mathbb{K}[x_1,\dots, x_k]$, we do not require $f_i$ to be linear for $2\leq i\leq r$, one can assume them to be linear without loss of generality,  by considering the residue of $f_i$ modulo $(f_1, \dots, f_{i-1})$ in $\mathbb{K}[x_1,\dots, x_k]$. Furthermore, it is clear that any non-constant linear polynomial in $\mathbb{K}[x_1,\dots, x_k]$ is a linear sequence of length $1$.

\begin{remark}\label{Rem1}
    \begin{enumerate}[(i)]
    \item[]
        \item For any linear sequence $f_1, \dots, f_r$ in $\mathbb{K}[x_1,\dots, x_k]$, one obtains the $\mathbb{K}$-algebra isomorphism $\mathbb{K}[x_1,\dots, x_k]/(f_1,\dots, f_r)\cong \mathbb{K}[y_1, y_2, \dots, y_{k-r}]$ by inducting on $r$. In particular, the quotient of $\mathbb{K}[x_1,\dots, x_k]$ by an ideal generated by a linear sequence is a global complete intersection. Furthermore, we also note that the length $r$ of any linear sequence $f_1, \dots, f_r$ in $\mathbb{K}[x_1,\dots, x_k]$ is at most $k$.
        \item For an embedded affine scheme $(X, \mathbb{A}^k_{\mathbb{K}})$ defined by $X=\spec(\mathbb{K}[x_1,\dots, x_k]/I)$, the process of $r$ iterations of linear reduction is equivalent to taking an $I$-linear sequence $f_1, \dots, f_r\in \mathbb{K}[x_1,\dots, x_k]$ and embedding $X$ in the closed subscheme $\spec(\mathbb{K}[x_1,\dots, x_k]/(f_1, \dots, f_r))\subseteq \mathbb{A}^k_{\mathbb{K}}$ as the closed subscheme $\displaystyle\spec\frac{\mathbb{K}[x_1,\dots, x_k]/(f_1,\dots, f_r)}{I/(f_1,\dots, f_r)}$.
        \item By the embedding of $X$ in $\spec(\mathbb{K}[x_1,\dots, x_k]/(f_1, \dots, f_r))$ from (ii) above, we must have $\dim X\leq \dim \mathbb{K}[x_1,\dots, x_k]/(f_1,\dots, f_r)= k-r$, where the last equality follows from (i) above. In particular, we note that for any proper ideal $I\subseteq \mathbb{K}[x_1,\dots, x_k]$, the length $r$ of any $I$-linear sequence is at most $\cdim_{\mathbb{K}[x_1,\dots, x_k]}(I)$.  
    \end{enumerate}
\end{remark}

We noted earlier that the process of linear reduction of an embedded affine scheme $(X, \mathbb{A}^k_{\mathbb{K}})$ terminates after finitely many steps (in fact, after at most $\cdim(X)$ steps by Remark~\ref{Rem1}(iii)). The terminal linearly reduced affine scheme obtained from $X=\spec(\mathbb{K}[x_1,\dots, x_k]/I)$, corresponds to a maximal $I$-linear sequence of $\mathbb{K}[x_1,\dots, x_k]$ by Remark~\ref{Rem1}(ii). \textit{A priori}, this process depends on the choice of the maximal $I$-linear sequence. The following proposition shows that the ideal generated by a maximal $I$-linear sequence, and hence the linear subvariety of $\mathbb{A}^k_{\mathbb{K}}$ in which the terminal scheme is embedded, is independent of this choice.  

\begin{proposition}\label{Prop1}
    Let $I\subsetneq \mathbb{K}[x_1,\dots, x_k]$ be a proper ideal, and $I_{\lin}\subseteq I$ be the sub-ideal generated by all the linear polynomials in $I$. If $f_1, \dots, f_r$ in $\mathbb{K}[x_1,\dots, x_k]$ is any maximal $I$-linear sequence, then $I_{\lin}=(f_1, \dots, f_r)$ as ideals of $\mathbb{K}[x_1,\dots, x_k]$.
\end{proposition}

\begin{proof}
    Without loss of generality, we may assume that $f_1, \dots, f_r$ are all linear polynomials in $I\subseteq \mathbb{K}[x_1,\dots, x_k]$. Clearly, we have $(f_1, \dots, f_r)\subseteq I_{\lin}$. Suppose the reverse inclusion fails to hold, i.e., there exists a linear polynomial $g\in I_{\lin}\setminus (f_1, \dots, f_r)$, whereby its residue modulo $(f_1, \dots, f_r)$ is also linear. If the residue is non-constant, then $f_1, f_2, \dots, f_r, g$ would be an $I$-linear sequence, contradicting the maximality of $f_1, f_2, \dots, f_r$. If the residue of $g$ modulo $(f_1, \dots, f_r)$ is a non-zero constant scalar, then since $(f_1,\dots, f_r)\subseteq I_{\lin}$, we would contradict the properness of $I$. Thus, the residue of $g$ modulo $(f_1,\dots, f_r)$ must be $0$, or equivalently, $g\in (f_1,\dots, f_r)$. This proves the reverse inclusion.   
\end{proof}

By Remark~\ref{Rem1}(ii) and Proposition~\ref{Prop1}, the terminal affine scheme obtained by a maximal sequence of iterations of linear reductions of $X=\spec(\mathbb{K}[x_1,\dots, x_k]/I)$ is, as a closed subscheme of the linear subvariety $V(I_{\lin})\subseteq\mathbb{A}^k_{\mathbb{K}}$, independent of the choice of the maximal $I$-linear sequence $f_1,\dots, f_r$ in $\mathbb{K}[x_1,\dots, x_k]$; in particular, the length $r$ of such a sequence and the number $q=k-r$ of remaining variables are independent of this choice, by Remark~\ref{Rem1}(i). However, the presentation of the terminal scheme as a closed subscheme $\spec(\mathbb{K}[x_1,\dots, x_q]/I_{\lred})$ of an affine space $\mathbb{A}^q_{\mathbb{K}}$, and hence the derivation $D_q$ of its coordinate ring, depend on the choice of the variables eliminated in each step, see Remark~\ref{Rem:choice} below. The following two definitions therefore retain this choice as part of the data.

\begin{definition}[Complete linear reduction]\label{Def3}
    Given an affine $\mathbb{K}$-scheme $X=\spec(A)$, with $A=\mathbb{K}[x_1, \dots, x_k]/I$, we define a \textit{complete linear reduction} of $X$ to be an embedded affine $\mathbb{K}$-scheme $(X_{\lred}, \mathbb{A}^q_{\mathbb{K}})$ with $X_{\lred}:=\spec(A_{\lred})$ isomorphic to $X$, where $A_{\lred}:=\mathbb{K}[x_1, \dots, x_q]/I_{\lred}$ is the $\mathbb{K}$-algebra obtained from $A$ by applying a maximal sequence of linear reductions and renaming the remaining variables $x_1,\dots, x_q$. The ideal $I_{\lred}$ contains no non-constant linear polynomial, and the choice of the eliminated variables is regarded as part of the data of the complete linear reduction. 
\end{definition}

\begin{definition}[Linearly essential D-subscheme]\label{Def4}
     Let $X=\spec(A)$ be an affine $\mathbb{K}$-scheme with $A=\mathbb{K}[x_1, \dots, x_k]/I$ for some ideal $I\subseteq \mathbb{K}[x_1, \dots, x_k]$, and let $(X_{\lred}, \mathbb{A}^q_{\mathbb{K}})$ be a complete linear reduction of $X$, with $X_{\lred}=\spec(A_{\lred})$ and $A_{\lred}=\mathbb{K}[x_1, \dots, x_q]/I_{\lred}$. Define the \textit{linearly essential D-subscheme} $\mathscr{D}_{\leff}X$ of $X$ with respect to this complete linear reduction to be the closed affine subscheme $\mathscr{D}X_{\lred}:=\spec(\mathbb{K}[x_1, \dots, x_q]/D_q(I_{\lred}))$ of $X_{\lred}$.
\end{definition}

\begin{remark}\label{Rem:choice}
The linearly essential D-subscheme depends, in general, on the chosen complete linear reduction, because the derivation $D_q$ is not invariant under the linear changes of coordinates relating two such reductions. For instance, let $I=(x_1-2x_2, (x_2-x_3)^2+x_3)\subseteq\mathbb{K}[x_1,x_2,x_3]$. Eliminating $x_1$ gives $I_{\lred}=(g)$ with $g=(x_2-x_3)^2+x_3\in\mathbb{K}[x_2,x_3]$ and $D_2(g)=1$, so that $\mathscr{D}X_{\lred}=\emptyset$; eliminating $x_2$ instead gives $I_{\lred}=(h)$ with $h=(x_1/2-x_3)^2+x_3\in\mathbb{K}[x_1,x_3]$ and $D_2(h)=1-(x_1/2-x_3)$, so that $\mathscr{D}X_{\lred}=\spec(\mathbb{K}[x_1,x_3]/(h, D_2(h)))$ is the reduced point $(x_1,x_3)=(0,-1)$. If $I$ contains no non-constant linear polynomial, then $(X,\mathbb{A}^k_{\mathbb{K}})$ is its own complete linear reduction and $\mathscr{D}_{\leff}X=\mathscr{D}X$ involves no choice; this is the only case used in this paper, see Remark~\ref{Rem2}.
\end{remark}

Since the linearly essential D-subscheme $\mathscr{D}_{\leff}X$ of an affine $\mathbb{K}$-scheme $X$ is also affine, we define higher order linearly essential D-subschemes of $X$ by defining the analogous subschemes of $\mathscr{D}_{\leff}X$.

\begin{definition}[Linearly essential $i^{th}$ D-subscheme]\label{Def5} Let $X=\spec(A)$ be an affine $\mathbb{K}$-scheme with $A=\mathbb{K}[x_1, \dots, x_k]/I$. We define the linearly essential $i^{th}$ D-subscheme of $X$ to be the affine $\mathbb{K}$-scheme $\mathscr{D}^i_{\leff}X:=\mathscr{D}_{\leff}(\mathscr{D}_{\leff}(\cdots (\mathscr{D}_{\leff}X)\cdots))$ obtained by iterating the construction of the linearly essential D-subscheme $i$ times for $i>0$, with respect to a choice of complete linear reduction at each step. Also define $\mathscr{D}_{\leff}^0X:=X$.     
\end{definition}

\begin{remark}\label{Rem2}
Let $X=\spec(A)$ be an affine $\mathbb{K}$-scheme for some $\mathbb{K}$-algebra $A=\mathbb{K}[x_1, \dots, x_k]/I$, where $I$ is a homogeneous ideal generated by homogeneous polynomials $f_1, \dots, f_m$ of degree $n>1$. Then $I$ does not contain any linear polynomials, so $X$ is \textit{completely linearly reduced}, i.e., $(X,\mathbb{A}^k_{\mathbb{K}})$ is its own complete linear reduction, and $\mathscr{D}_{\leff}X=\mathscr{D}X$ involves no choice. Note that the derivation $D_k$ sends monomials in $\mathbb{K}[x_1, \dots, x_k]$ of degree $n$ to homogeneous polynomials of degree $n-1$. Hence for each $j$, the homogeneous ideals $D^j_k(I):=D_k(D^{j-1}_k(I))$ are generated by the homogeneous polynomials $D^i_k(f_l)$ ($0\leq i\leq j$ and $1\leq l\leq m$) of degree at least $n-j$. Consequently, for all $1\leq j\leq n-2$, $D^j_k(I)$ do not contain linear polynomials implying that the affine $\mathbb{K}$-subschemes $\mathscr{D}_{\leff}^jX=\mathscr{D}^jX\subset X$ are completely linearly reduced. Thus, the linearly essential $j^{th}$ D-subscheme of $X$ is
\[\mathscr{D}^j_{\leff}X=\mathscr{D}^jX:=\spec(\mathbb{K}[x_1, \dots, x_k]/D^j_k(I)) , \ \forall 1\leq j\leq n-1.\]
\end{remark}

\subsubsection{D-subschemes of affine schemes over positive characteristic}\label{subsubsec2.1.2}
The $\mathbb{K}$-linear derivation $D_k:\mathbb{K}[x_1,\dots, x_k]\rightarrow \mathbb{K}[x_1,\dots, x_k]$ and its higher compositions can be packaged together into a \textit{Hasse--Schmidt} derivation $\exp(tD_k):\mathbb{K}[x_1,\dots, x_k]\rightarrow \mathbb{K}[x_1,\dots, x_k][\![t]\!]$ given by $\exp(D_k):=\sum_{i\geq 0}(D_k^i/i!)t^i$. In fact, the $i^{th}$ multivariate Hasse--Schmidt derivation obtained from $D_k$ (i.e., the coefficient of $t^i$ in $\exp(tD_k)$) can be defined alternatively by \eqref{EqnHS1}.

When $\mathbb{K}$ is an algebraically closed field of characteristic $p$, we will use the $i^{th}$ derivation on $\mathbb{K}[x_1,\dots, x_k]$ defined by \eqref{EqnHS1}. For an affine scheme $X=\spec(\mathbb{K}[x_1,\dots, x_k]/I)$ over a positive characteristic field, one can naively define the $j^{th}$ Hasse--Schmidt D-subscheme of $X$ as the subscheme $\mathcal{HD}^jX:=\spec(\mathbb{K}[x_1,\dots, x_k]/(HD_k(I), HD^2_k(I), \dots, HD^j_k(I)))$. For ``nice" affine schemes $X=\spec(\mathbb{K}[x_1,\dots, x_k]/I)$, i.e., those for which the ideal $(HD_k(I), HD^2_k(I),\dots, HD^j_k(I))\subseteq \mathbb{K}[x_1,\dots, x_k]$ is a proper ideal, the naive definition of $\mathcal{HD}^j_kX$ provides the desired construction. In general, one can construct reductions analogous to Definition~\ref{Def3} to tackle degeneracies.

In this paper, we will be concerned with (higher) $\delta$-subschemes (primarily $\delta=D_k$ and $HD_k$) of principal monomial schemes (over algebraically closed fields $\mathbb{K}$) of degree $n\geq2$. A \textit{principal monomial scheme} is defined to be a monomial scheme determined by a single monomial of total degree $n\geq2$ in the ring $\mathbb{K}[x_1, \dots, x_k]$. In particular, over characteristic $p$, the naive definition of $\mathcal{HD}^j_kX$ works for such $X$ for all $1\leq j<$ degree of the defining monomial.  

\subsection{D-subschemes of principal monomial affine \texorpdfstring{$\mathbb{K}$}{K}-schemes}\label{subsec2.2}
Let $\mathbb{K}$ be an algebraically closed field of characteristic $0$. For an integer $n\geq 2$, let $\mathbf{r}:=(r_1,\dots, r_k)$ corresponding to a fixed ordered partition $r_1+r_2+\cdots+r_k=n$ into $k\geq 1$ positive integers. Let $\mathbf{x}^{\mathbf{r}}\in \mathbb{K}[x_1, \dots, x_k]$ be the monomial $\mathbf{x}^{\mathbf{r}}=\prod_{i=1}^{k}x_i^{r_i}$. Define $\mathscr{S}_n(\mathbf{r})$ to be the affine monomial $\mathbb{K}$-scheme $\spec(\mathbb{K}[x_1, \dots, x_k]/I(\mathbf{r}))$, where $I(\mathbf{r})$ is the ideal generated by $\mathbf{x}^{\mathbf{r}}$ in $\mathbb{K}[x_1,\dots, x_k]$. We will identify the set of $\mathbb{K}$-rational points $\mathscr{S}_n(\mathbf{r})(\mathbb{K})$ with the corresponding affine algebraic subset $V(I(\mathbf{r}))=V(\mathbf{x}^{\mathbf{r}})\subseteq \mathbb{A}^k_{\mathbb{K}}$.

\begin{remark}\label{Remlem}
    The set of $\mathbb{K}$-rational points $\mathscr{S}_n(\mathbf{r})(\mathbb{K})$ is reducible, with irreducible components $V(x_i)$ for each $1\leq i\leq k$. With the canonical $\mathbb{K}$-affine scheme structure induced by $\mathscr{S}_n(\mathbf{r})$, each irreducible component $V(x_i)$ occurs with multiplicity $r_i$.
\end{remark}

Since $\mathbf{x}^{\mathbf{r}}\in \mathbb{K}[x_1, \dots, x_k]$ is a monomial of degree at least $2$, by Remark~\ref{Rem2}, we know that $\mathscr{D}^j\mathscr{S}_n(\mathbf{r}):=\spec(\mathbb{K}[x_1,\dots, x_k]/D^j_kI(\mathbf{r}))=\spec(\mathbb{K}[x_1, \dots, x_k]/(\mathbf{x}^{\mathbf{r}}, D_k\mathbf{x}^{\mathbf{r}}, \dots, D_k^j\mathbf{x}^{\mathbf{r}}))$ is completely linearly reduced for all $0\leq j\leq n-2$. Thus, the linearly essential $j^{th}$ D-subschemes of $\mathscr{S}_n(\mathbf{r})$ are: \[\mathscr{D}^j\mathscr{S}_n(\mathbf{r})=\spec(\mathbb{K}[x_1, \dots, x_k]/(\mathbf{x}^{\mathbf{r}}, D_k\mathbf{x}^{\mathbf{r}}, \dots, D_k^j\mathbf{x}^{\mathbf{r}})),\ \forall \ 1\leq j\leq n-1.\]
Henceforth we will refer to these closed subschemes of $\mathscr{S}_n(\mathbf{r})$ as the (higher) D-subschemes of $\mathscr{S}$. 

The next result characterizes the set of $\mathbb{K}$-rational points $\mathscr{D}^j\mathscr{S}_n(\mathbf{r})(\mathbb{K})$ for all $0\leq j\leq n-1$ in terms of the irreducible components $V(x_i)$ of $\mathscr{S}_n(\mathbf{r})(\mathbb{K})$. The set of $\mathbb{K}$-rational points $\mathscr{D}^j\mathscr{S}_n(\mathbf{r})(\mathbb{K})$ is equal to $V(\mathbf{x}^{\mathbf{r}}, D_k\mathbf{x}^{\mathbf{r}}, \dots, D_k^j\mathbf{x}^{\mathbf{r}})=\bigcap_{i=1}^{j}V(\mathbf{x}^{\mathbf{r}}, D_k^i\mathbf{x}^{\mathbf{r}})\subseteq \mathbb{A}^k_{\mathbb{K}}(\mathbb{K})$. The following proposition gives the set-theoretic description of $\mathscr{D}^j\mathscr{S}_n(\mathbf{r})(\mathbb{K})$.

\begin{proposition}\label{Prop2}
    For all $1\leq j\leq n$, the set $\mathscr{D}^{j-1}\mathscr{S}_n(\mathbf{r})(\mathbb{K})$ of $\mathbb{K}$-rational points of the $(j-1)^{th}$ D-subscheme of $\mathscr{S}_n(\mathbf{r})$ is the affine algebraic subset of $\mathbb{A}^k_{\mathbb{K}}(\mathbb{K})$ given by
    \begin{equation}\label{Eqn1}
        \mathscr{D}^{j-1}\mathscr{S}_n(\mathbf{r})(\mathbb{K})=\bigcup_{l=1}^{j}\bigcup_{\substack{i_1<i_2<\cdots<i_l \\ r_{i_1}+r_{i_2}+\cdots+r_{i_l}\geq j}}V(x_{i_1})\cap V(x_{i_2})\cap\cdots\cap V(x_{i_l}).
    \end{equation}
\end{proposition}

\begin{proof}
     We induct on $j$, with Remark~\ref{Remlem} corresponding to the base case $j=1$. Let us assume that the statement holds for some $j=m<n$. Since $$\mathscr{D}^{m-1}\mathscr{S}_n(\mathbf{r})(\mathbb{K})=V(\mathbf{x}^{\mathbf{r}}, D_k\mathbf{x}^{\mathbf{r}}, \dots, D_k^{m-1}\mathbf{x}^{\mathbf{r}})\subseteq \mathbb{A}^k_{\mathbb{K}}(\mathbb{K}),$$ by the induction hypothesis we have $$V(\mathbf{x}^{\mathbf{r}}, D_k\mathbf{x}^{\mathbf{r}}, \dots, D_k^{m-1}\mathbf{x}^{\mathbf{r}})=\bigcup_{l=1}^{m}\bigcup_{\substack{i_1<i_2<\cdots<i_l \\ r_{i_1}+r_{i_2}+\cdots+r_{i_l}\geq m}}V(x_{i_1})\cap V(x_{i_2})\cap\cdots\cap V(x_{i_l}).$$
    The above union can be stratified along the levels $r_{i_1}+r_{i_2}+\cdots + r_{i_l}= \mu$, where $\mu\geq m$. That is $V(\mathbf{x}^{\mathbf{r}}, D_k\mathbf{x}^{\mathbf{r}}, \dots, D_k^{m-1}\mathbf{x}^{\mathbf{r}})=\bigcup_{\mu=m}^{n}\mathscr{D}^{m-1}\mathscr{S}_n(\mathbf{r})[\mu](\mathbb{K})$, where $$\mathscr{D}^{m-1}\mathscr{S}_n(\mathbf{r})[\mu](\mathbb{K}):=\bigcup_{l=1}^{m}\bigcup_{\substack{i_1<i_2<\cdots<i_l \\ r_{i_1}+r_{i_2}+\cdots+r_{i_l}=\mu}}V(x_{i_1})\cap V(x_{i_2})\cap\cdots\cap V(x_{i_l}). $$
    By definition, any $\mathbb{K}$-rational point $P\in \mathscr{D}^{m-1}\mathscr{S}_n(\mathbf{r})(\mathbb{K})$ satisfies $\mathbf{x}^{\mathbf{r}}(P)=D_k\mathbf{x}^{\mathbf{r}}(P)=\cdots= D_k^{m-1}\mathbf{x}^{\mathbf{r}}(P)=0$, and thus $P\in \mathscr{D}^{m}\mathscr{S}_n(\mathbf{r})(\mathbb{K})$ if and only if $D_k^m\mathbf{x}^{\mathbf{r}}(P)=0$ as well. Recall that $\mathbf{x}^{\mathbf{r}}=\prod_{i=1}^{k}x_{i}^{r_i}$ and thus, for every $j\geq0$, one can see that 
    \begin{equation}\label{DEqn}
       D_k^j\mathbf{x}^{\mathbf{r}}=\sum_{\substack{\beta_1+\cdots+\beta_k=j\\ 0\leq \beta_i\leq r_i}}c_{\beta_1,\cdots,\beta_k}x_{1}^{r_1-\beta_1}\cdots x_{k}^{r_k-\beta_k} \ \ \text{(with appropriate coefficients }c_{\beta_1,\dots,\beta_k})  
    \end{equation} is a homogeneous degree $n-j$ polynomial in $x_{i}$ with integer coefficients. We first prove that $\mathscr{D}^{m-1}\mathscr{S}_n(\mathbf{r})[\mu](\mathbb{K})\subseteq \mathscr{D}^{m}\mathscr{S}_n(\mathbf{r})(\mathbb{K})$ for any $\mu\geq m+1$.  Let $P\in \mathscr{D}^{m-1}\mathscr{S}_n(\mathbf{r})[\mu](\mathbb{K})$ for any $\mu\geq m+1$, so that there exist $i_1<\cdots<i_l$ for some $1\leq l\leq m$, such that $P\in V(x_{i_1})\cap\cdots\cap V(x_{i_l})$ with multiplicities of $V(x_{i_j})$ equal to $r_{i_j}$ such that $r_{i_1}+\cdots +r_{i_l}=\mu\geq m+1$.  For this point $P$, we rewrite $D_k^m\mathbf{x}^{\mathbf{r}}(P)$ by grouping the terms $x_i(P)^{r_i-\beta_i}$ in each monomial $x_{1}(P)^{r_1-\beta_1}\cdots x_{k}(P)^{r_k-\beta_k}$ in $D_k^m\mathbf{x}^{\mathbf{r}}(P)$ according to whether the index $i$ of  $x_i(P)^{r_i-\beta_i}$ is equal to $i_j$ for some $j=1,\dots, l$ or not:
    \begin{equation}\label{EqnDm}
    D_k^m\mathbf{x}^{\mathbf{r}}(P)=\sum_{\substack{\beta_1+\cdots+\beta_k=m\\ 0\leq \beta_i\leq r_i}}c_{\beta_1,\dots,\beta_k}(\prod_{j=1}^{l}x_{i_j}(P)^{r_{i_j}-\beta_{i_j}})(\prod_{i\neq i_j}x_{i}(P)^{r_i-\beta_i}).   \end{equation}
    By choice of $P$, we have $r_{i_1}+\cdots+r_{i_l}=\mu\geq m+1$, while $\beta_{i_1}+\cdots+\beta_{i_l}\leq m$. Hence, for each monomial occurring in the sum in \eqref{EqnDm}, there is some $1\leq j\leq l$ such that $r_{i_j}-\beta_{i_j}\geq 1$. Since $P\in V(x_{i_1})\cap\cdots\cap V(x_{i_l})$, all monomials in $D_k^m\mathbf{x}^{\mathbf{r}}(P)$ vanish and so $D_k^m\mathbf{x}^{\mathbf{r}}(P)=0$. 
    
    Now, let $P\in \mathscr{D}^{m-1}\mathscr{S}_n(\mathbf{r})[m](\mathbb{K})$, so that $P\in V(x_{i_1})\cap\cdots\cap V(x_{i_l})$ for some $i_1<\cdots<i_l$ with $r_{i_1}+\cdots+r_{i_l}=m$. Any monomial in the sum in \eqref{EqnDm} with $\beta_i\geq1$ for some $i\neq i_1, \dots, i_l$ has a factor of $x_{i_j}$ for some $1\leq j\leq l$, so it vanishes at $P$. Thus, the only non-vanishing monomials can come from those for which $\beta_i=0$ for all $i\neq i_1, \dots, i_l$. Furthermore, since $0\leq \beta_{i_j}\leq r_{i_j}$ and $\sum_{i=1}^{k}\beta_i=\sum_{j=1}^{l}\beta_{i_j}=m=\sum_{j=1}^{l}r_{i_j}$, it follows that $\beta_{i_j}=r_{i_j}$ for all $1\leq j\leq l$. Hence the only possible non-zero monomial (up to a non-zero coefficient) is $\prod_{i\neq i_j}x_{i}(P)^{r_i}$ and it follows that $D_k^m\mathbf{x}^{\mathbf{r}}(P)=0$ if and only if there exists $i\neq i_j$ for which $x_{i}(P)=0$. Consequently, $D_k^m\mathbf{x}^{\mathbf{r}}(P)=0$ if and only if $P\in \bigcap_{j=1}^{l}V(x_{i_j})\cap V(x_{i})$ for some $i\neq i_j$, and thus $r_i+\sum_{j=1}^{l}r_{i_j}\geq m+1$. Thus, we have obtained that
     $$\mathscr{D}^{m}\mathscr{S}_n(\mathbf{r})(\mathbb{K})=\bigcup_{l=1}^{m+1}\bigcup_{\substack{i_1<i_2<\cdots<i_l \\ r_{i_1}+r_{i_2}+\cdots+r_{i_l}\geq m+1}}V(x_{i_1})\cap V(x_{i_2})\cap\cdots\cap V(x_{i_l}).$$
     This completes the inductive step and hence the proof.  
\end{proof}

Since $\mathscr{D}^{m}\mathscr{S}_n(\mathbf{r})(\mathbb{K})=\bigcap_{i=1}^{m}V(\mathbf{x}^{\mathbf{r}}, D^i_k\mathbf{x}^{\mathbf{r}})$, by Proposition~\ref{Prop2}, for any $1\leq m\leq n-1$, the affine algebraic set $V(\mathbf{x}^{\mathbf{r}}, D_k^m\mathbf{x}^{\mathbf{r}})$ contains an entire irreducible component $V(x_{j})$ of $\mathscr{S}_n(\mathbf{r})(\mathbb{K})$ if $V(x_j)$ has multiplicity $r_j\geq m+1$. The following lemma shows that the converse is true as well.

\begin{lemma}\label{Lem1}
    $V(x_{j})\subseteq V(\mathbf{x}^{\mathbf{r}}, D_k^m\mathbf{x}^{\mathbf{r}})$ if and only if $r_j\geq m+1$, for any $0\leq m\leq n-1$.
\end{lemma}
\begin{proof}
    By the general Leibniz rule for self-composition of derivations, we have
    \begin{align*}
        D_k^m\mathbf{x}^{\mathbf{r}}=D_k^m(\prod_{i=1}^{k}x_{i}^{r_i})=\sum_{l_1+\cdots +l_k=m}\binom{m}{l_1, \dots, l_k}\prod_{1\leq i\leq k}D_k^{l_i}(x_{i}^{r_i})\\= \sum_{0\leq l_j\leq m}\binom{m}{l_j}D_k^{l_j}(x_{j}^{r_j}) \sum_{l_1+\cdots+\hat{l_j}+\cdots+ l_k=m-l_j}\binom{m-l_j}{l_1, \dots, \hat{l_j}, \dots, l_k}\prod_{i\neq j}D_k^{l_i}(x_{i}^{r_i}),
    \end{align*}    
    where $\hat{l_j}$ in the second line indicates that $l_j$ is removed. Since $V(x_j)\subseteq V(\mathbf{x}^{\mathbf{r}})$, it suffices to show $V(x_j)\cap V(D^m_k\mathbf{x}^{\mathbf{r}})=V(x_j)$ if and only if $r_j\geq m+1$. For $P\in V(x_{j})\cap V(D_k^m\mathbf{x}^{\mathbf{r}})$, the terms with $l_j<r_j$ in the above expression for $D^m_k\mathbf{x}^{\mathbf{r}}$ vanish, when evaluated at $P$. Furthermore, $D^{l_j}_k(x_j^{r_j})=0$ for all $l_j>r_j$. Thus, the only remaining terms in $D^m_k\mathbf{x}^{\mathbf{r}}(P)$ are exactly those with $l_j=r_j$, so by the general Leibniz rule, $D_k^m\mathbf{x}^{\mathbf{r}}(P)=r_j!\binom{m}{r_j}D_k^{m-r_j}(\prod_{i\neq j}x_{i}^{r_i})(P)=0$. 
    
    Note that $V(x_{j})\subseteq V(D_k^m\mathbf{x}^{\mathbf{r}})$ if and only if the polynomial $r_j!\binom{m}{r_j}D_k^{m-r_j}(\prod_{i\neq j}x_{i}^{r_i})$ in the $k-1$ variables $\{x_{i}\mid \ i\neq j\}$, vanishes (identically) at all points of $V(x_j)\subseteq\mathbb{A}^k_{\mathbb{K}}(\mathbb{K})$, which is isomorphic to the affine $k-1$ space $\mathbb{A}_{\mathbb{K}}^{k-1}(\mathbb{K})$. It follows that the polynomial $r_j!\binom{m}{r_j}D_k^{m-r_j}(\prod_{i\neq j}x_{i}^{r_i})$ in $k-1$ variables must be identically $0$. This is true if and only if $r_j\geq m+1$, as $D_k^{m-r_j}(\prod_{i\neq j}x_{i}^{r_i})$ is a non-zero polynomial in $k-1$ variables, because $r_1+\cdots+\hat{r_j}+\cdots+r_k=n-r_j>m-r_j$.
\end{proof}

\section{Geometry of higher discriminants}\label{sec3}

Throughout this section $\mathbb{K}$ will denote an algebraically closed field of arbitrary characteristic. 

\subsection{Geometry of monic univariate polynomials of degree \texorpdfstring{$n$}{n}}\label{subsec3.1} 
For any positive integer $n$, let $\mathbb{K}[X]_n$ be the set of degree $n$ monic univariate polynomials over $\mathbb{K}$. We introduce certain constructions on $\mathbb{K}[X]$, and $\mathbb{K}[X]_n$ in particular, which will be useful in understanding Conjecture~\ref{con1}. 

\begin{definition}[Shift equivalence]\label{Def6}
    Define the equivalence relation $\sim$ on the monic polynomials of the polynomial ring $\mathbb{K}[X]$ as $f(X)\sim g(X)$ in $\mathbb{K}[X]$ if $g(X)=f(X-\beta)$ for some $\beta\in\mathbb{K}$. We say $f(X)$ and $g(X)$ are shift equivalent if $f(X)\sim g(X)$. We will denote the shift equivalence  class of $f(X)$ by $[f(X)]$ and the set of shift equivalence classes of monic polynomials by $[\mathbb{K}[X]]$.
\end{definition}

The shift equivalence class $[f(X)]$ of the monic polynomial $f(X)\in \mathbb{K}[X]$ consists of those monic polynomials $g(X)$, whose roots are exactly the roots of $f(X)$ up to translation -- in other words, the relative location of the roots of $g(X)$ are same as those of $f(X)$. In particular, shift equivalence restricts to a well-defined equivalence relation on $\mathbb{K}[X]_n$, whose equivalence classes will be denoted by $[\mathbb{K}[X]_n]$. Notice, if $f(X)$ satisfies the Casas-Alvero condition, then so does any $g(X)\in[f(X)]$. The following definition is crucial in lifting Conjecture~\ref{con1}, which is a problem in one variable, to a problem about linear algebraic subvarieties of higher dimensional affine spaces. 

\begin{definition}[Characteristic maps]\label{Def7}
For any monic $f(X)\in \mathbb{K}[X]$ of degree $n$ (i.e., $f(X)\in \mathbb{K}[X]_n$) along with a fixed labelling $\overline{\alpha}=(\alpha_1, \dots, \alpha_n)$ of roots, we define the $\overline{\alpha}$-characteristic map of $f(X)$ to be the $\mathbb{K}$-algebra map $\ch_{f,\overline{\alpha}}: \mathbb{K}[x_1, \dots, x_n]\rightarrow \mathbb{K}[X]$ defined by $x_i\mapsto X-\alpha_i$ for all $1\leq i\leq n$. Consequently, define the set of characteristic maps of $f(X)$ to be the set $\ch(f):=\{\ch_{f,\overline{\alpha}}\mid \ \overline{\alpha}\in \mathfrak{R}(f)\}$, where $\mathfrak{R}(f)$ is the set of all ordered tuples $\overline{\alpha}=(\alpha_1, \dots, \alpha_n)$ obtained by permuting the roots of $f$.
\end{definition}

Note that for each $1\leq i\leq n$ and $\overline{\alpha}\in \mathfrak{R}(f)$, $\ker\ch_{f,\overline{\alpha}}$ has a ``base $i$'' presentation, given by $\ker\ch_{f,\overline{\alpha}}=(\{x_j-x_i+\gamma_{ji}\mid \ 1\leq j\leq n\})\subseteq \mathbb{K}[x_1, \dots, x_n]$, where $\gamma_{ji}=\alpha_j-\alpha_i$. Geometrically the $\overline{\alpha}$-characteristic map $\ch_{f,\overline{\alpha}}$ induces a closed embedding $\ch_{f,\overline{\alpha}}^{\#}:\mathbb{A}^1_\mathbb{K}\hookrightarrow \mathbb{A}^n_{\mathbb{K}}$ as the closed subscheme $\mathscr{L}_{f,\overline{\alpha}}:=\spec(\mathbb{K}[x_1, \dots, x_n]/\ker\ch_{f,\overline{\alpha}})\subseteq \mathbb{A}^n_{\mathbb{K}}$. Considering the ``base $1$'' presentation of $\ker\ch_{f,\overline{\alpha}}$, the set of $\mathbb{K}$-rational points of $\mathscr{L}_{f,\overline{\alpha}}$ is $$\mathscr{L}_{f,\overline{\alpha}}(\mathbb{K})=V(\ker\ch_{f,\overline{\alpha}})=\{(\beta, \beta+\gamma_{12}, \beta+\gamma_{13}, \dots, \beta+\gamma_{1n})\mid \ \beta\in\mathbb{K}\}\subseteq \mathbb{A}^n_{\mathbb{K}}(\mathbb{K}).$$ Thus, the association $$(\beta, \beta+\gamma_{12}, \beta+\gamma_{13}, \dots, \beta+\gamma_{1n})\mapsto \prod_{j=1}^{n}(X+\beta+\gamma_{1j})= f(X+\beta+\alpha_1)$$ maps $\mathscr{L}_{f,\overline{\alpha}}(\mathbb{K})$ onto the shift equivalence class $[f(X)]$, i.e., the line $\mathscr{L}_{f,\overline{\alpha}}(\mathbb{K})$ parametrizes $[f(X)]$. This parametrization is bijective if $\operatorname{char}\mathbb{K}=0$, but not in general: for $f(X)=X^p-X$ over a field of characteristic $p$, one has $f(X+c)=f(X)$ for all $c\in\mathbb{F}_p$. 
Since the differences $\gamma_{1j}$ are unchanged when all the roots are translated by a common element of $\mathbb{K}$, the subscheme $\mathscr{L}_{f,\overline{\alpha}}\subseteq \mathbb{A}^n_{\mathbb{K}}$ depends only on the shift equivalence class $[f(X)]$ of $f(X)$. For each monic $f(X)\in \mathbb{K}[X]$ of degree $n$, we therefore obtain the set $\mathfrak{L}(f):=\{\mathscr{L}_{f,\overline{\alpha}}, \ \overline{\alpha}\in\mathfrak{R}(f)\}$ of affine lines in $\mathbb{A}^n_{\mathbb{K}}$, which only depends on the shift equivalence class $[f(X)]$. Furthermore, the action of the symmetric group $\mathfrak{S}_n$ on $\mathbb{A}^n_{\mathbb{K}}$ via permuting the variables of $\mathbb{K}[x_1,\dots, x_n]$, induces a transitive action on $\mathfrak{L}(f)$, so that the affine lines in $\mathfrak{L}(f)$ differ only by a permutation.

In what follows, let $f(X)=\prod_{i=1}^{n}(X-\alpha_i)\in \mathbb{K}[X]$ be a monic polynomial of degree $n\geq 2$ with a fixed ordering of roots $\overline{\alpha}=(\alpha_1, \dots, \alpha_n)\in \mathfrak{R}(f)$. Let $\mathbf{x}_n:=\prod_{i=1}^{n}x_i\in\mathbb{K}[x_1, \dots, x_n]$ be the unique elementary symmetric monomial in $n$ variables. The following lemma will allow us to transfer conditions on the derivatives of $f(X)$ in $\mathbb{K}[X]$ to appropriate conditions on transformations of (Hasse-Schmidt) D-subschemes of the monomial affine $\mathbb{K}$-scheme $\spec(\mathbb{K}[x_1,\dots, x_n]/(\mathbf{x}_n))$.

\begin{lemma}\label{Lem2}
Let $HD^i_n:\mathbb{K}[x_1, \dots, x_n]\rightarrow \mathbb{K}[x_1, \dots, x_n]$ be the $i^{th}$ multivariate Hasse-Schmidt derivation defined by \eqref{EqnHS1} and let $H_i:\mathbb{K}[X]\rightarrow \mathbb{K}[X]$ be the $i^{th}$ univariate Hasse-Schmidt derivation. Then, for all $i\geq 1$ and for any monic polynomial $f(X)=\prod_{i=1}^{n}(X-\alpha_i)\in\mathbb{K}[X]$ with a fixed ordering $\overline{\alpha}=(\alpha_1,\dots, \alpha_n)$ of roots, the following diagram commutes:
 \[
\begin{tikzcd}
 \mathbb{K}[x_1, \dots, x_n]\arrow[dd, "\ch_{f, \overline{\alpha}}"'] \arrow[r, "HD^i_n"] & \mathbb{K}[x_1,\dots, x_n] \arrow[dd, "\ch_{f, \overline{\alpha}}"] \\
                                  &                   \\
\mathbb{K}[X] \arrow[r, "H_i"']                 & \mathbb{K}[X]                
\end{tikzcd}
\]
\end{lemma}

\begin{proof}
It suffices to check that the diagram commutes for any monomial in $\mathbb{K}[x_1,\dots, x_n]$ as all the arrows are $\mathbb{K}$-linear. Let $x_1^{\beta_1}x_2^{\beta_2}\dots x_n^{\beta_n}\in \mathbb{K}[x_1,\dots, x_n]$, whereby using \eqref{EqnHS1} we have $$\ch_{f, \overline{\alpha}}(HD^i_n(x_1^{\beta_1}\dots x_n^{\beta_n}))=\sum_{j_1+\dots+j_n=i}\binom{\beta_1}{j_1}\dots \binom{\beta_n}{j_n}(X-\alpha_1)^{\beta_1-j_1}\dots (X-\alpha_n)^{\beta_n-j_n}.$$
Similarly, using Leibniz rule for the univariate Hasse-Schmidt derivations, we have
\begin{align*}
 H_i(\ch_{f, \overline{\alpha}}(x_1^{\beta_1}\dots x_n^{\beta_n})) =H_i((X-\alpha_1)^{\beta_1}\dots (X-\alpha_n)^{\beta_n})\\ = \sum_{j_1+\dots+j_n=i}H_{j_1}((X-\alpha_1)^{\beta_1})\dots H_{j_n}((X-\alpha_n)^{\beta_n}),
\end{align*}
which equals $\ch_{f, \overline{\alpha}}(HD^i_n(x_1^{\beta_1}\dots x_n^{\beta_n}))$ by using the fact that $H_i((X-\alpha)^n)=\binom{n}{i}(X-\alpha)^{n-i}$.
\end{proof}

The set $\mathbb{K}[X]_n$ can be naturally identified with $\mathbb{A}^n_{\mathbb{K}}$ via $a_0+a_1X+\dots+a_{n-1}X^{n-1}+X^{n}\mapsto (a_0, a_1,\dots, a_{n-1})$. This equips $\mathbb{K}[X]_n$ with a variety structure. The set $\mathbb{K}[X]_n$ can further be naturally identified with the quotient of $\mathbb{A}^n_{\mathbb{K}}$ by the symmetric group $\mathfrak{S}_n$, in every characteristic, since $\mathbb{K}[x_1,\dots, x_n]^{\mathfrak{S}_n}$ is the polynomial ring in the elementary symmetric polynomials over any field. For this, we introduce the following map.

\begin{definition}[Root map]\label{Def8}
    We define the degree $n$ root map to be the set map $\rho_n:\mathbb{A}^n_{\mathbb{K}}(\mathbb{K})\rightarrow \mathbb{K}[X]_n$ sending $(\alpha_1, \alpha_2, \dots, \alpha_n)$ to the degree $n$ polynomial $\prod_{i=1}^{n}(X+\alpha_i)$.
\end{definition}

\begin{remark}\label{Rem3}
\begin{enumerate}[(i)]
    \item[]
    \item Let $\mathfrak{S}_n$ act on $\mathbb{A}^n_{\mathbb{K}}(\mathbb{K})$ by permutation of coordinates. Then, by the definition of $\rho_n:\mathbb{A}^n_{\mathbb{K}}(\mathbb{K})\rightarrow \mathbb{K}[X]_n$, the fibers of $\rho_n$ are the $\mathfrak{S}_n$-orbits of $\mathbb{A}^n_{\mathbb{K}}(\mathbb{K})$. Thus, $\rho_n$ induces a set bijection $\overline{\rho_n}: \mathbb{A}^n_{\mathbb{K}}(\mathbb{K})\sslash\mathfrak{S}_n\xrightarrow{\sim} \mathbb{K}[X]_n$, where $\mathbb{A}^n_{\mathbb{K}}(\mathbb{K})\sslash\mathfrak{S}_n$ is the set of $\mathbb{K}$-rational points of the GIT quotient $\mathbb{A}^n_{\mathbb{K}}\sslash\mathfrak{S}_n:=\spec(\mathbb{K}[x_1,\dots, x_n]^{\mathfrak{S}_n})$. Furthermore, since $\rho_n$ can be identified with the quotient map $\mathbb{A}^n_{\mathbb{K}}(\mathbb{K})\rightarrow \mathbb{A}^n_{\mathbb{K}}(\mathbb{K})\sslash \mathfrak{S}_n$, it follows that $\rho_n$ is a finite morphism.
    \item Let $q_S: \mathbb{K}[X]_n\rightarrow [\mathbb{K}[X]_n]$ be the map sending a degree $n$ monic polynomial $f(X)\in\mathbb{K}[X]_n$ to its shift equivalence class $[f(X)]\in [\mathbb{K}[X]_n]$. The fibers of the composite map $\mathbb{A}^n_{\mathbb{K}}(\mathbb{K})\xrightarrow{\rho_n}\mathbb{K}[X]_n\xrightarrow{q_S} [\mathbb{K}[X]_n]$ are exactly $(q_S\circ \rho_n)^{-1}([f])=\bigcup_{\overline{\alpha}\in\mathfrak{R}(f)}\mathscr{L}_{f, \overline{\alpha}}(\mathbb{K})$. Thus, $\rho_n$ formalizes the parametrization of the shift equivalence class $[f(X)]$ by the points of $\mathscr{L}_{f, \overline{\alpha}}(\mathbb{K})$ (for any $\overline{\alpha}\in \mathfrak{R}(f)$) described previously.
\end{enumerate}   
\end{remark}

\subsection{Higher discriminant hypersurfaces}\label{subsec3.2} The goal of this subsection is to construct and describe higher discriminant hypersurfaces in $\mathbb{A}^n_{\mathbb{K}}$, motivating their relevance to Conjecture~\ref{con1}. 
 
 Let $f(X)=\prod_{i=1}^{n}(X-\alpha_i)\in \mathbb{K}[X]_n$ be a generic monic univariate polynomial of degree $n$ with an ordering of roots $\overline{\alpha}=(\alpha_1,\dots, \alpha_n)$ and let $f_i:=H_i(f)$ be its $i^{th}$ Hasse-Schmidt derivative. For each $1\leq i\leq n-1$, the hypothesis $\gcd(f, f_i)\neq 1$ is equivalent to the existence of $1\leq j_i\leq n$ such that we have the ideal containment $(f, f_i)\subseteq (X-\alpha_{j_i})\subseteq \mathbb{K}[X]$. By Lemma~\ref{Lem2}, since $\ch_{f, \overline{\alpha}}(\mathbf{x}_n)=f(X)$ and $\ch_{f, \overline{\alpha}}(HD_n^i\mathbf{x}_n)=f_i(X)$ for all $1\leq i\leq n-1$, by applying $\ch_{f, \overline{\alpha}}^{-1}$ to the above ideal containment in $\mathbb{K}[X]$, we obtain 
\begin{equation}\label{Eqn2}
(\mathbf{x}_n, HD^i_n\mathbf{x}_n, \ker\ch_{f,\overline{\alpha}})\subseteq (x_{j_i}, \ker\ch_{f,\overline{\alpha}})\subseteq \mathbb{K}[x_1,\dots, x_n].
\end{equation}
Now we note that $\mathbb{K}[x_1,\dots, x_n]/(x_i, \ker\ch_{f, \overline{\alpha}})\cong \mathbb{K}[X]/(X-\alpha_i)\cong \mathbb{K}$, whence $(x_i, \ker\ch_{f,\overline{\alpha}})\subseteq \mathbb{K}[x_1,\dots, x_n]$ are maximal ideals for all $1\leq i\leq n$. Thus, \eqref{Eqn2} is equivalent to the containment $V(x_{j_i})\cap \mathscr{L}_{f,\overline{\alpha}}(\mathbb{K})\subseteq V(\mathbf{x}_n, HD^i_n\mathbf{x}_n)\cap \mathscr{L}_{f,\overline{\alpha}}(\mathbb{K})\subseteq \mathbb{A}^n_{\mathbb{K}}(\mathbb{K})$ which is, furthermore, equivalent to the condition $V(x_{j_i})\cap V(\mathbf{x}_n, HD^i_n\mathbf{x}_n)\cap \mathscr{L}_{f,\overline{\alpha}}(\mathbb{K})\neq \emptyset$, since $V(x_{j_i})\cap \mathscr{L}_{f,\overline{\alpha}}(\mathbb{K})$ is a singleton. Since $\gcd(f, f_i)\neq 1$ is equivalent to the existence of an index $1\leq j_i\leq n$ for which the intersection $V(x_{j_i})\cap V(\mathbf{x}_n, HD^i_n\mathbf{x}_n)\cap \mathscr{L}_{f,\overline{\alpha}}(\mathbb{K})$ is non-empty, by noting that $V(\mathbf{x}_n)=\bigcup_{j=1}^{n}V(x_j)$, we obtain 
\begin{equation}\label{Eqn3}
    \gcd(f, f_i)\neq 1 \iff V(\mathbf{x}_n, HD_n^i\mathbf{x}_n)\cap \mathscr{L}_{f,\overline{\alpha}}(\mathbb{K})\neq \emptyset.
\end{equation}
Note that the non-emptiness of $V(\mathbf{x}_n, HD_n^i\mathbf{x}_n)\cap \mathscr{L}_{f,\overline{\alpha}}(\mathbb{K})$ is independent of the choice of  ordering of roots $\overline{\alpha}\in \mathfrak{R}(f)$. This is because $\mathfrak{S}_n$ acts transitively on $\mathfrak{L}(f)=\{\mathscr{L}_{f,\overline{\alpha}}, \ \overline{\alpha}\in\mathfrak{R}(f)\}$ and $\mathbf{x}_n, HD^i_n\mathbf{x}_n\in \mathbb{K}[x_1,\dots, x_n]^{\mathfrak{S}_n}$, making $V(\mathbf{x}_n, HD^i_n\mathbf{x}_n)$ $\mathfrak{S}_n$-invariant for all $1\leq i\leq n-1$. In fact, note that for all $0\leq i\leq n-1$, we have $HD^i_n\mathbf{x}_n=e_{n-i}(x_1,\dots, x_n)$, the degree $n-i$ elementary symmetric polynomial in $n$ variables.  

\begin{remark}\label{Rem4}
We describe the affine algebraic subsets $V(\mathbf{x}_n, HD_n^i\mathbf{x}_n)\subseteq \mathbb{A}^n_{\mathbb{K}}(\mathbb{K})$ appearing in \eqref{Eqn3}. Let $\mathbf{x}_{n-1}^j:=\prod_{l\neq j}x_l$ and note that $V(\mathbf{x}_n, HD_n^i\mathbf{x}_n)=\bigcup_{j=1}^{n}V(x_j, HD_n^i\mathbf{x}_n)$. Since $\mathbf{x}_n=x_j\cdot\mathbf{x}^j_{n-1}$, it follows by the Leibniz rule for Hasse-Schmidt derivations, that $$HD_n^i\mathbf{x}_n=x_jHD^i_n\mathbf{x}^j_{n-1}+HD_n^{i-1}\mathbf{x}^j_{n-1}= x_jHD^i_{n-1,j}\mathbf{x}^j_{n-1}+HD_{n-1,j}^{i-1}\mathbf{x}^j_{n-1},$$ where $HD^i_{n-1,j}:\mathbb{K}[x_1,\dots, \hat{x_j}, \dots, x_n]\rightarrow \mathbb{K}[x_1,\dots, \hat{x_j}, \dots, x_n]$ is the natural Hasse-Schmidt derivation obtained by restricting $HD^i_n$ on $\mathbb{K}[x_1,\dots, x_n]$ to $\mathbb{K}[x_1,\dots, \hat{x_j}, \dots, x_n]$. Consequently, we obtain $V(x_j, HD^i_n\mathbf{x}_n)=V(x_j, HD^{i-1}_{n-1, j}\mathbf{x}_{n-1}^j)\subseteq \mathbb{A}^n_{\mathbb{K}}(\mathbb{K})$. In other words, consider $V(HD^{i-1}_{n-1, j}\mathbf{x}_{n-1}^j)\subseteq Z^j_n(\mathbb{K})=\mathbb{A}^{n-1}_{\mathbb{K}}(\mathbb{K})$, where $Z^j_n=\spec(\mathbb{K}[x_1, \dots, x_j, \dots, x_n]/(x_j))$. Then $V(x_j, HD_n^i\mathbf{x}_n)$ is the image (henceforth denoted by $V^j_{n-1}(HD^{i-1}_{n-1}\mathbf{x}_{n-1})$) of $V(HD^{i-1}_{n-1, j}\mathbf{x}_{n-1}^j)\subseteq Z^j_n(\mathbb{K})$ in $\mathbb{A}^n_{\mathbb{K}}(\mathbb{K})$ under the natural inclusion $Z^j_n\hookrightarrow \mathbb{A}^n_{\mathbb{K}}$ as a closed subscheme. Thus,
\begin{equation}\label{Eqn4}
    V(\mathbf{x}_n, HD^i_n\mathbf{x}_n)=\bigcup_{j=1}^{n}V^j_{n-1}(HD^{i-1}_{n-1}\mathbf{x}_{n-1}).
\end{equation}
\end{remark}

\begin{definition}[Shift projection maps]\label{Def9}
  For all $1\leq j\leq n$, let $Z^j_n(\mathbb{K})$ be the coordinate hyperplane of $\mathbb{A}^n_{\mathbb{K}}(\mathbb{K})$ corresponding to $x_j=0$. Define the $j^{th}$ shift projection map to be the map $\mathfrak{p}^j_n: \mathbb{A}^n_{\mathbb{K}}(\mathbb{K})\rightarrow Z^j_n(\mathbb{K})$ which sends $(x_1, \dots, x_j, \dots, x_n)$ to $(x_1-x_j, \dots, x_j-x_j, \dots, x_n-x_j)$. 
\end{definition}

The utility of shift projection maps is in the following observation. By \eqref{Eqn3}, we are interested in the non-emptiness of $V(\mathbf{x}_n, HD_n^i\mathbf{x}_n)\cap \mathscr{L}_{f,\overline{\alpha}}(\mathbb{K})$ which, by \eqref{Eqn4}, is equivalent to the non-emptiness of $V^j_{n-1}(HD^{i-1}_{n-1}\mathbf{x}_{n-1})\cap \mathscr{L}_{f,\overline{\alpha}}(\mathbb{K})$ for some $1\leq j\leq n$. If $\mathbf{y}=(y_1, \dots, 0, \dots, y_n)\in V^j_{n-1}(HD^{i-1}_{n-1}\mathbf{x}_{n-1})\cap \mathscr{L}_{f,\overline{\alpha}}(\mathbb{K})$ (where $0$ is at the $j^{th}$ coordinate), then $$\mathscr{L}_{f,\overline{\alpha}}(\mathbb{K})= \{(\beta+y_1, \dots, \beta, \dots, \beta+y_n)\mid \ \beta\in \mathbb{K}\}=(\mathfrak{p}^j_n)^{-1}(\mathbf{y}).$$ Thus, \eqref{Eqn3} is equivalent to
\begin{equation}\label{Eqn5}
\gcd(f, f_i)\neq 1\quad \iff \quad \mathscr{L}_{f,\overline{\alpha}}(\mathbb{K})\subseteq X^i_n(\mathbb{K}):=\bigcup_{j=1}^{n}(\mathfrak{p}^j_n)^{-1}(V^j_{n-1}(HD^{i-1}_{n-1}\mathbf{x}_{n-1})).
\end{equation}

The equivalence \eqref{Eqn5} can be formalized using the root map $\rho_n:\mathbb{A}_{\mathbb{K}}^n(\mathbb{K})\rightarrow \mathbb{K}[X]_n$ as follows.

\begin{lemma}\label{Lem4}
    Let $\mathfrak{X}^i_n(\mathbb{K}):=\{f(X)\in \mathbb{K}[X]_n\mid \gcd(f, f_i)\neq 1\}\subseteq \mathbb{K}[X]_n$ and let $X^i_n(\mathbb{K})=\bigcup_{j=1}^{n}(\mathfrak{p}^j_n)^{-1}(V^j_{n-1}(HD^{i-1}_{n-1}\mathbf{x}_{n-1}))\subseteq \mathbb{A}^n_{\mathbb{K}}(\mathbb{K})$. Then $\rho_n^{-1}(\mathfrak{X}^i_n(\mathbb{K}))=X^i_n(\mathbb{K})$.
\end{lemma}

\begin{proof}
     Let $\overline{\beta}=(\beta_1,\dots, \beta_n)\in\mathbb{A}^n_{\mathbb{K}}(\mathbb{K})$, let $f=\rho_n(\overline{\beta})=\prod_{j=1}^{n}(X+\beta_j)$, and order the roots of $f$ as $\overline{\alpha}=(-\beta_1,\dots, -\beta_n)$. Then $\gamma_{1j}=\alpha_1-\alpha_j=\beta_j-\beta_1$, so $\overline{\beta}$ is the point of $\mathscr{L}_{f,\overline{\alpha}}(\mathbb{K})$ with parameter $\beta=\beta_1$, and $\mathscr{L}_{f,\overline{\alpha}}(\mathbb{K})=\overline{\beta}+\mathbb{K}\cdot(1,\dots, 1)$. Each set $(\mathfrak{p}^j_n)^{-1}(V^j_{n-1}(HD^{i-1}_{n-1}\mathbf{x}_{n-1}))$ is stable under translation by $\mathbb{K}\cdot(1,\dots, 1)$, since $\mathfrak{p}^j_n$ is constant on such lines. Hence $\overline{\beta}\in X^i_n(\mathbb{K})$ if and only if $\mathscr{L}_{f,\overline{\alpha}}(\mathbb{K})\subseteq X^i_n(\mathbb{K})$, which by \eqref{Eqn5} holds if and only if $f\in\mathfrak{X}^i_n(\mathbb{K})$.
\end{proof}

\begin{definition}
    In light of Lemma~\ref{Lem4}, for all $1\leq i\leq n-1$, we define the \textit{$i^{th}$ discriminant hypersurface of degree $n$} over an algebraically closed field $\mathbb{K}$ to be $X^i_n(\mathbb{K})\subseteq \mathbb{A}^n_{\mathbb{K}}(\mathbb{K})$.
\end{definition}

\begin{definition}
    For a generic degree $n$ monic univariate polynomial $f(X)=X^n+y_1X^{n-1}+y_2X^{n-2}+\dots+ y_{n-1}X+y_n$ and for all $1\leq i\leq n-1$, define the $i^{th}$ discriminant polynomial $\disc_n^i(y_1,\dots, y_n)$ to be the resultant $\Res(f, f_i)\in\mathbb{Z}[y_1, \dots, y_n]$ of $f$ and $f_i$, computed with respect to the formal degrees $n$ and $n-i$; it is weighted homogeneous of weight $n(n-i)$ when $y_j$ is given the weight $j$.
\end{definition}

Let $\kappa_n: \mathbb{A}^n_{\mathbb{K}}(\mathbb{K})\rightarrow \mathbb{K}[X]_n$ be the ``coefficient map" defined by $(y_1,\dots, y_n)\mapsto X^n+y_1X^{n-1}+\dots+y_{n-1}X+y_n$, which is also a bijection. Since $f$ is monic, $\Res(f,f_i)=\prod f_i(\alpha)$, the product being taken over the roots $\alpha$ of $f$ with multiplicity; thus $\disc^i_n$ vanishes at $\mathbf{y}\in\mathbb{A}^n_{\mathbb{K}}(\mathbb{K})$ if and only if $f=\kappa_n(\mathbf{y})$ and $f_i$ have a common root, also when the leading coefficient $\binom{n}{i}$ of $f_i$ vanishes in $\mathbb{K}$. Hence the zero locus $\Delta^i_n(\mathbb{K}):=V(\disc_n^i)\subseteq \mathbb{A}^n_{\mathbb{K}}(\mathbb{K})$ equals $\kappa_n^{-1}(\mathfrak{X}^i_n(\mathbb{K}))$. Thus, we have
\begin{equation}\label{Eqn:relations}
\mathfrak{X}^i_n(\mathbb{K})=\kappa_n(\Delta^i_n(\mathbb{K}))=\rho_n(X^i_n(\mathbb{K}))
\end{equation}
We note that, although $\Delta^i_n(\mathbb{K})$ and $\mathfrak{X}^i_n(\mathbb{K})$ in \eqref{Eqn:relations} above are defined by the vanishing of the discriminant polynomials $\disc^i_n$, the explicit form of these resultants is inconvenient for the dimension estimates of this paper. By contrast, by \eqref{Eqn5}, $X^i_n(\mathbb{K})$ has a very explicit description in terms of the shift projection maps and the zero sets of elementary symmetric polynomials. The following proposition will be useful in analyzing $X^i_n(\mathbb{K})$ further.

\begin{proposition}\label{Prop3}
        For each $1\leq i\leq n$ and $j\neq i$, let $\Phi_{ij}:Z^i_n(\mathbb{K})\rightarrow Z^i_n(\mathbb{K})$ be the regular map  $(x_1, \dots, x_{i-1}, 0, x_{i+1},\dots, x_n)\mapsto (y_1, \dots, y_{i-1}, 0,  y_{i+1}, \dots, y_n)$, where $y_i=0$, $y_j=-x_j$ and $y_l=x_l-x_j$ for all $l\notin\{i,j\}$. Then each $\Phi_{ij}$ is an involution on $Z^i_n(\mathbb{K})$ such that
        \begin{equation}\label{Eqn6}
            (\mathfrak{p}^{j}_n)^{-1}(V^{j}_{n-1}(HD^{k}_{n-1}\mathbf{x}_{n-1}))= (\mathfrak{p}^{i}_n)^{-1}(\Phi_{ij}(V^{i}_{n-1}(HD^{k}_{n-1}\mathbf{x}_{n-1}))).
        \end{equation}
    \end{proposition}

 \begin{proof}
     Let $(x_1,\dots, x_n)\in (\mathfrak{p}^{j}_n)^{-1}(V^{j}_{n-1}(HD^{k}_{n-1}\mathbf{x}_{n-1}))$, whereby equivalently $\mathfrak{p}^{j}_n(x_1,\dots, x_n)=(x_1-x_j, \dots, x_{j-1}-x_j, 0, x_{j+1}-x_j, \dots, x_n-x_j)\in V^{j}_{n-1}(HD^{k}_{n-1}\mathbf{x}_{n-1})\subseteq Z^j_n(\mathbb{K})$, i.e.,
     \begin{equation}\label{Eqprop1}
         HD^k_{n-1}\mathbf{x}_{n-1}(x_1-x_j, \dots, x_{j-1}-x_j, x_{j+1}-x_j,\dots, x_n-x_j)=0.
     \end{equation}
     Let $(y_1,\dots, y_{i-1}, 0, y_{i+1}, \dots, y_n):=\mathfrak{p}^{i}_n(x_1,\dots, x_n)=(x_1-x_i, \dots, x_{i-1}-x_i, 0, x_{i+1}-x_i, \dots, x_n-x_i)$ and then note that $y_l-y_j=x_l-x_j$ for all $l\neq i$ and $-y_j=x_i-x_j$. Substituting these for $x_l-x_j$ in \eqref{Eqprop1} for all $l\neq j$ (assume $i<j$ without loss of generality):
     \begin{align*}
         HD^k_{n-1}\mathbf{x}_{n-1}(y_1-y_j, \dots, y_{i-1}-y_j, -y_j, y_{i+1}-y_j,\dots, y_{j-1}-y_j, y_{j+1}-y_j,\dots, y_n-y_j)=0.
     \end{align*} Furthermore, since $HD^k_{n-1}\mathbf{x}_{n-1}(z_1,\dots, z_{n-1})$ is a symmetric polynomial in the $z_i$, we equivalently obtain the following by shifting $-y_j$ from the $i^{th}$ coordinate the $(j-1)^{th}$ coordinate
     \begin{equation}\label{Eqprop2}
         HD^k_{n-1}\mathbf{x}_{n-1}(y_1-y_j, \dots, y_{i-1}-y_j, y_{i+1}-y_j, \dots, y_{j-1}-y_j, -y_j, y_{j+1}-y_j,\dots, y_n-y_j)=0.
     \end{equation} By definition of the involution $\Phi_{ij}:Z^i_n(\mathbb{K})\rightarrow Z^i_n(\mathbb{K})$, we see that $\Phi_{ij}(y_1,\dots, y_{i-1}, 0, y_{i+1},\dots, y_n)=\Phi_{ij}(\mathfrak{p}^i_n(x_1,\dots, x_n)\in V^i_{n-1}(HD^k_{n-1}\mathbf{x}_{n-1})$. Since each of the above steps is reversible, we see that $(x_1,\dots, x_n)\in (\mathfrak{p}^{j}_n)^{-1}(V^{j}_{n-1}(HD^{k}_{n-1}\mathbf{x}_{n-1}))$ if and only if $\Phi_{ij}(\mathfrak{p}^i_n(x_1,\dots, x_n)\in V^i_{n-1}(HD^k_{n-1}\mathbf{x}_{n-1})$. Since $\Phi_{ij}$ is an involution, we have $(\mathfrak{p}^{j}_n)^{-1}(V^{j}_{n-1}(HD^{k}_{n-1}\mathbf{x}_{n-1}))=(\mathfrak{p}^{i}_n)^{-1}(\Phi_{ij}(V^{i}_{n-1}(HD^{k}_{n-1}\mathbf{x}_{n-1})))$.
 \end{proof} 

 The following lemma describes the composition of the involutions $\Phi_{ij}: Z^i_n(\mathbb{K})\rightarrow Z^i_n(\mathbb{K})$, with each other and with transpositions, for a fixed $1\leq i\leq n$. We skip the proof, which is a straightforward computation.

 \begin{lemma}\label{Lem3}
     Let $1\leq i\leq n$, and let $1\leq j_1, j_2\leq n$ be distinct indices, both different from $i$. Let $\tau_{j_1,j_2}:Z^i_n(\mathbb{K})\rightarrow Z^i_n(\mathbb{K})$ be the transposition of coordinates $x_{j_1}$ and $x_{j_2}$ for all $(x_1,\dots, x_{i-1}, 0, x_{i+1}, \dots, x_n)\in Z^i_n(\mathbb{K})$. Then we have $\Phi_{ij_1}\circ\Phi_{ij_2}=\tau_{j_1j_2}\circ \Phi_{ij_1}=\Phi_{ij_2}\circ \tau_{j_1j_2}$. Furthermore, for any $1\leq j, j_1, j_2\leq n$ all mutually distinct and not equal to $i$, we have $\Phi_{ij}\circ \tau_{j_1j_2}=\tau_{j_1j_2}\circ \Phi_{ij}$.
 \end{lemma}

 \begin{remark}\label{Rem5new}
     For brevity, let $\Phi_j:=\Phi_{nj}: Z^n_n(\mathbb{K})\rightarrow Z^n_n(\mathbb{K})$ be the involution described in Proposition~\ref{Prop3} for $i=n$, $j\neq i$ and let $\mathfrak{p}_n:=\mathfrak{p}^n_n:\mathbb{A}^n_{\mathbb{K}}(\mathbb{K})\rightarrow Z^n_n(\mathbb{K})$ be the $n^{th}$ shift projection map. Furthermore, define $\Phi_n:Z^n_n(\mathbb{K})\rightarrow Z^n_n(\mathbb{K})$ to be the identity map. By Proposition~\ref{Prop3}, we have $ (\mathfrak{p}^{j}_n)^{-1}(V^{j}_{n-1}(HD^{i-1}_{n-1}\mathbf{x}_{n-1}))= \mathfrak{p}_n^{-1}(\Phi_{j}(V^{n}_{n-1}(HD^{i-1}_{n-1}\mathbf{x}_{n-1})))$ for all $1\leq j\leq n$. The involutions $\Phi_j: Z^n_n(\mathbb{K})\rightarrow Z^n_n(\mathbb{K})$ are regular for all $1\leq j\leq n$: for $j\neq n$, $\Phi_j$ is induced by the corresponding algebra isomorphism
    \[
        \Phi^{\#}_j: \mathbb{K}[x_1,\dots, x_{n-1}]\rightarrow \mathbb{K}[x_1,\dots, x_{n-1}], \quad \Phi^{\#}_j(x_l)=\begin{cases}
            x_l-x_j, \quad l\neq j, \\
            -x_j, \quad l=j.
        \end{cases}
    \] Letting $\Phi^{\#}_n:\mathbb{K}[x_1,\dots, x_{n-1}]\rightarrow \mathbb{K}[x_1,\dots, x_{n-1}]$ be the identity map corresponding to $\Phi_n$, we have $\Phi_{j}(V^n_{n-1}(HD^{i-1}_{n-1}\mathbf{x}_{n-1}))=V^n_{n-1}(\Phi^{\#}_{j}(HD^{i-1}_{n-1}\mathbf{x}_{n-1}))\subseteq Z^n_n(\mathbb{K})$ for any $1\leq j\leq n$. Thus \eqref{Eqn5} yields:
     \begin{equation}\label{Eqn:shift}
     X^i_n(\mathbb{K})=\mathfrak{p}_n^{-1}(\bigcup_{j=1}^{n}V^n_{n-1}(\Phi^{\#}_j(HD^{i-1}_{n-1}\mathbf{x}_{n-1}))), \ \ \forall 1\leq i\leq n-1
     \end{equation}
     Note that $\bigcup_{j=1}^{n}V^n_{n-1}(\Phi^{\#}_j(HD^{i-1}_{n-1}\mathbf{x}_{n-1}))$ is an algebraic subset in $Z^n_n(\mathbb{K})\cong \mathbb{A}^{n-1}_{\mathbb{K}}(\mathbb{K})$, i.e., affine space of one lower dimension.
 \end{remark}

 \subsection{Arithmetic Casas-Alvero scheme and higher discriminants}\label{subsec3.3}
In this subsection, we provide an alternate description of the set $X_n(\mathbb{K})$ of $\mathbb{K}$-rational points of the arithmetic Casas-Alvero schemes $X_n$ considered in \cite{GVB}, for any algebraically closed field $\mathbb{K}$. This will be utilized in establishing Theorem~\ref{mainthm}.

We briefly recall the construction of the weighted projective schemes $X_n$. For each $n\geq 2$, let $\mathbb{Z}^w[y_1,y_2,\dots, y_{n-1}, y_n]$ be the graded algebra of $n$ variables $y_1,\dots, y_n$, where $y_i$ has weight $i$. Consider the \textit {reduced $i^{th}$ discriminant polynomials} $\disc^i_n(y_1,\dots, y_{n-1}, 0)$ (by setting $y_n=0$) for $1\leq i\leq n-1$, which are homogeneous polynomials in $\mathbb{Z}^w[y_1,y_2,\dots, y_{n-1}, y_n]$ of weighted degree $n(n-i)$. The $n^{th}$ arithmetic Casas-Alvero scheme is the weighted projective $\mathbb{Z}$-subscheme $X_n\subseteq \mathbb{P}_{\mathbb{Z}}(1,2,\dots, n-1)$ defined by the ideal $\langle \disc^i_n(y_1,\dots, y_{n-1},0), \ 1\leq i\leq n-1\rangle$.

For any algebraically closed field $\mathbb{K}$, it follows that the affine cone of $(X_n)_{\mathbb{K}}:=X_n\times_{\spec(\mathbb{Z})}\spec(\mathbb{K})$ is equal to $\bigcap_{i=1}^{n-1}\Delta^i_n(\mathbb{K})\cap V(y_n)\subseteq \mathbb{A}^n_{\mathbb{K}}(\mathbb{K})$, where $\Delta^i_n(\mathbb{K}):=V(\disc^i_n(y_1,\dots, y_n))\subseteq \mathbb{A}^n_{\mathbb{K}}(\mathbb{K})$ and $\mathbb{A}^n_{\mathbb{K}}=\spec(\mathbb{K}[y_1,\dots, y_n])$. Then we have (see \cite[Definition~3.1.9]{TH}):
\[ X_n(\mathbb{K})=(\widehat{X}_n(\mathbb{K})\setminus \{\mathbf{0}\})/ \mathbb{G}^w_m(\mathbb{K}); \ \ \widehat{X}_n(\mathbb{K}):=\bigcap_{i=1}^{n-1}\Delta^i_n(\mathbb{K})\cap V(y_n),\]
where $\mathbb{G}^w_m$ denotes the weighted action of $\mathbb{G}_m$ on $\mathbb{A}^n\setminus\{\mathbf{0}\}$ by $\lambda.(a_1,\dots, a_n)=(\lambda a_1, \lambda^2 a_2,\dots, \lambda^n a_n)$.

Note that there is a graded $\mathbb{Z}$-algebra map $\nu^{\#}_n: \mathbb{Z}^w[y_1,\dots, y_n]\rightarrow \mathbb{Z}[x_1,\dots, x_n]$, where $x_i$'s have weight $1$, given by $y_i\mapsto (-1)^{i}HD^{n-i}\mathbf{x}_n$, for all $1\leq i\leq n$. That this map is weighted graded follows from the fact that $HD^i_n\mathbf{x}_n=e_{n-i}(x_1,\dots, x_n)$, i.e., the degree $n-i$ elementary symmetric polynomial in $x_1,\dots, x_n$. Extending scalars, for any algebraically closed field $\mathbb{K}$, we obtain a weighted graded $\mathbb{K}$-algebra map $ \mathbb{K}^w[y_1,\dots, y_n]\rightarrow \mathbb{K}[x_1,\dots, x_n]$, which we also denote by $\nu^{\#}_n$. This induces a regular map $\nu_n:\mathbb{A}^n_{\mathbb{K}}\rightarrow \mathbb{A}^n_{\mathbb{K}}$, which we call \textit{Vieta's map}, since $\mathbf{a}:=(a_1,\dots, a_n)\mapsto (-e_1(\mathbf{a}), \dots, (-1)^ne_n(\mathbf{a}))$. By Vieta's formulae, $\kappa_n(\nu_n(\mathbf{a}))=\prod_{i=1}^{n}(X-a_i)=(-1)^n\rho_n(\mathbf{a})(-X)$. Since $H_i(f(-X))=(-1)^iH_i(f)(-X)$, the set $\mathfrak{X}^i_n(\mathbb{K})$ is stable under $f(X)\mapsto(-1)^nf(-X)$, and \eqref{Eqn:relations} gives $\nu_n(X^i_n(\mathbb{K}))= \Delta^i_n(\mathbb{K})$ for all $1\leq i\leq n-1$. Moreover, $\nu_n(\bigcap_{i=1}^{n-1}X^i_n(\mathbb{K}))= \bigcap_{i=1}^{n-1}\Delta^i_n(\mathbb{K})$, because the fibers of $\nu_n$ are the $\mathfrak{S}_n$-orbits and each $X^i_n(\mathbb{K})$ is $\mathfrak{S}_n$-stable, as noted below. Furthermore, since $\nu^{\#}_{n}(y_n)=(-1)^nHD^0\mathbf{x}_n=(-1)^nx_1x_2\dots x_n$, it follows that
\begin{align*}
\nu_n(\bigcap_{i=1}^{n-1}X^i_n(\mathbb{K}))\cap \nu_n(V(x_1x_2\dots x_n))= \bigcap_{i=1}^{n-1}\Delta^i_n(\mathbb{K})\cap V(y_n)\\ 
\implies \nu_n(\bigcup_{j=1}^{n}V(x_j)\cap(\bigcap_{i=1}^{n-1}X^i_n(\mathbb{K})))= \bigcap_{i=1}^{n-1}\Delta^i_n(\mathbb{K})\cap V(y_n),
\end{align*}
where we again use Vieta's relations to obtain the second equality. By Lemma~\ref{Lem3}, we observe that $X^i_n(\mathbb{K})$ is symmetric (i.e., invariant under $\mathfrak{S}_n$-action on $\mathbb{A}^n_{\mathbb{K}}$) for all $1\leq i\leq n-1$. In particular, if $\tau_{jn}$ is the permutation which swaps the coordinates $x_j$ and $x_n$, then $\tau_{jn}(V(x_j)\cap(\bigcap_{i=1}^{n-1}X^i_n(\mathbb{K})))=V(x_n)\cap(\bigcap_{i=1}^{n-1}X^i_n(\mathbb{K}))$. Since $\nu_n\circ \sigma=\nu_n$ for all $\sigma\in\mathfrak{S}_n$, we conclude
\begin{equation}\label{Eqn:Vieta}
    \nu_n(V(x_n)\cap(\bigcap_{i=1}^{n-1}X^i_n(\mathbb{K})))= \bigcap_{i=1}^{n-1}\Delta^i_n(\mathbb{K})\cap V(y_n)=\widehat{X}_n(\mathbb{K})
\end{equation}
Furthermore, $\nu_n:\mathbb{A}_{\mathbb{K}}^n\rightarrow \mathbb{A}_{\mathbb{K}}^n$ interchanges the weighted and unweighted actions of $\mathbb{G}_m$ on $\mathbb{A}_{\mathbb{K}}^n$, i.e., the following diagram commutes:
\[\begin{tikzcd}
\mathbb{G}_m\times \mathbb{A}_{\mathbb{K}}^n  \arrow[rr, "\operatorname{id}\times \nu_n"] \arrow[dd] &  & \mathbb{G}^w_m\times \mathbb{A}_{\mathbb{K}}^n \arrow[dd] \\
                             &  &              \\
\mathbb{A}_{\mathbb{K}}^n \arrow[rr, "\nu_n"]            &  & \mathbb{A}_{\mathbb{K}}^n           
\end{tikzcd}\]
The left vertical arrow in the diagram is the usual scaling action of $\mathbb{G}_m$ whereas the right arrow is the weighted scaling action of $\mathbb{G}_m$ on $\mathbb{A}^n_{\mathbb{K}}$. Since $\nu_n^{-1}(\{\mathbf{0}\})=\{\mathbf{0}\}$, the above commutative diagram yields a regular map $\overline{\nu_n}:\mathbb{P}^{n-1}_{\mathbb{K}}\rightarrow \mathbb{P}_{\mathbb{K}}(1,2,\dots, n)$. Utilizing the coefficient isomorphism $\kappa_n:\mathbb{A}^n_{\mathbb{K}}(\mathbb{K})\rightarrow \mathbb{K}[X]_n$, we observe that the fibers of the Vieta map $\nu_n$ are precisely the orbits of the $\mathfrak{S}_n$-action on $\mathbb{A}^n_{\mathbb{K}}(\mathbb{K})$. Moreover, $\nu_n$ is a finite morphism, since $\mathbb{K}[x_1,\dots, x_n]$ is integral over its subring $\mathbb{K}[e_1,\dots, e_n]$ generated by the elementary symmetric polynomials, each $x_i$ being a root of $\prod_{l=1}^{n}(T-x_l)=\sum_{h=0}^{n}(-1)^he_hT^{n-h}$. In particular, this implies that $\overline{\nu}_n: \mathbb{P}^{n-1}_{\mathbb{K}}\rightarrow \mathbb{P}_{\mathbb{K}}(1,2,\dots, n)$ is a quasi-finite morphism. Thus, we obtain the following description of $X_n(\mathbb{K})$.

\begin{proposition}\label{Prop:Vieta}
 Let $\widehat{\mathcal{V}}_n(\mathbb{K}):= V(x_n)\cap(\bigcap_{i=1}^{n-1}X^i_n(\mathbb{K}))\subseteq \mathbb{A}^n_{\mathbb{K}}(\mathbb{K})$ and $\mathcal{V}_n(\mathbb{K}):=(\widehat{\mathcal{V}}_n(\mathbb{K})\setminus\{\mathbf{0}\})/\mathbb{G}_m\subseteq \mathbb{P}^{n-1}_{\mathbb{K}}(\mathbb{K})$. Then $X_n(\mathbb{K})=\overline{\nu}_n(\mathcal{V}_n(\mathbb{K}))$, where $\overline{\nu}_n: \mathbb{P}^{n-1}_{\mathbb{K}}\rightarrow \mathbb{P}_{\mathbb{K}}(1,2,\dots, n)$ is the induced Vieta map.
\end{proposition}

\begin{proof}
    Follows immediately from \eqref{Eqn:Vieta}.
\end{proof}

\section{Proof of the main results}\label{sec4}

In this section, we provide various equivalent formulations of Conjecture~\ref{con1} using the higher discriminant hypersurfaces defined in Section~\ref{sec3}. We also prove Theorems~\ref{mainthm2}, \ref{mainthm}, \ref{mainthmC} and~\ref{mainthm5}, and Corollary~\ref{maincorE}. As usual, we let $\mathbb{K}$ be algebraically closed. Unless otherwise specified, we let $\mathbb{K}$ have arbitrary characteristic.

\begin{proposition}\label{Prop4}
    Conjecture~\ref{con1} holds in degree $n$ over $\mathbb{K}$ if and only if the intersection $\bigcap_{i=1}^{n-1}X^i_n(\mathbb{K})\subseteq \mathbb{A}^n_{\mathbb{K}}(\mathbb{K})$ has dimension $1$. 
\end{proposition}

\begin{proof}
    Note that a monic degree $n$ polynomial $f(X)\in \mathbb{K}[X]_n$ satisfies the Casas-Alvero condition if and only if $f(X)\in \bigcap_{i=1}^{n-1}\mathfrak{X}^i_n(\mathbb{K})$. Thus, Conjecture~\ref{con1} holds in degree $n$ over $\mathbb{K}$ if and only if $\bigcap_{i=1}^{n-1}\mathfrak{X}^i_n(\mathbb{K})=\{(X-\alpha)^n\mid \ \alpha\in \mathbb{K}\}$ which, by Lemma~\ref{Lem4}, is equivalent to $\bigcap_{i=1}^{n-1}X^i_n(\mathbb{K})=\Delta_n(\mathbb{K}):=\{(\alpha, \alpha, \dots, \alpha)\in \mathbb{A}^n_{\mathbb{K}}(\mathbb{K}) \mid \ \alpha\in \mathbb{K}\}$. Thus, it suffices to prove that $\dim \bigcap_{i=1}^{n-1}X^i_n(\mathbb{K}) =1$ is equivalent to $\bigcap_{i=1}^{n-1}X^i_n(\mathbb{K})= \Delta_n(\mathbb{K})$. We prove the forward implication as the reverse direction is obvious. Assume $\dim \bigcap_{i=1}^{n-1}X^i_n(\mathbb{K}) =1$, and let $(\alpha_1, \dots, \alpha_n)\in \bigcap_{i=1}^{n-1}X^i_n(\mathbb{K})\setminus \Delta_n(\mathbb{K})$ for the sake of contradiction. By the definition of $X^i_n(\mathbb{K})$, we have $f(X):=\prod_{i=1}^{n}(X-\alpha_i)\in \bigcap_{i=1}^{n-1}\mathfrak{X}^i_n(\mathbb{K})$. Furthermore, by \cite[Lemma 2]{DJ}, we see that $\prod_{i=1}^{n}(X-(\lambda\alpha_i+\alpha))\in\bigcap_{i=1}^{n-1}\mathfrak{X}^i_n(\mathbb{K})$ for all $\lambda, \alpha\in \mathbb{K}$, whence $$\lambda(\alpha_1, \alpha_2, \dots, \alpha_n)+\alpha(1,1,\dots, 1)\in \bigcap_{i=1}^{n-1}X^i_n(\mathbb{K}) \ \forall \lambda, \alpha\in \mathbb{K}.$$ This forms a $2$-dimensional linear subvariety of $\bigcap_{i=1}^{n-1}X^i_n(\mathbb{K})$ by choice of $(\alpha_1, \dots, \alpha_n)$, which contradicts the assumption $\dim \bigcap_{i=1}^{n-1}X^i_n(\mathbb{K}) =1$. This completes the proof.  
    \end{proof}

  The higher discriminant hypersurfaces in $\mathbb{A}^n_{\mathbb{K}}(\mathbb{K})$ are determined by $V(\Phi^{\#}_j(HD^{i-1}_{n-1}\mathbf{x}_{n-1}))\subseteq \mathbb{A}^{n-1}_{\mathbb{K}}(\mathbb{K})$ for all $1\leq j\leq n$ and $1\leq i\leq n-1$, by \eqref{Eqn:shift}. Thus, Proposition~\ref{Prop4} has a purely commutative algebraic analogue in terms of the polynomials $\Phi^{\#}_j(HD^{i-1}_{n-1}\mathbf{x}_{n-1})\in \mathbb{K}[x_1,\dots, x_{n-1}]$.   

  \begin{proposition}\label{Lem5}
    Conjecture~\ref{con1} holds in degree $n$ over $\mathbb{K}$ if and only if for all choices of $1\leq j_1, \dots, j_{n-1}\leq n$, the sequence $\Phi^{\#}_{j_1}(HD^{0}_{n-1}\mathbf{x}_{n-1}), \dots, \Phi^{\#}_{j_{n-1}}(HD^{n-2}_{n-1}\mathbf{x}_{n-1})$ forms a regular sequence of homogeneous polynomials in $\mathbb{K}[x_1,\dots, x_{n-1}]$. 
\end{proposition}
\begin{proof}
 By \eqref{Eqn:shift}, we can write
 \begin{equation}\label{Eqn8}
        \bigcap_{i=1}^{n-1}X^i_n(\mathbb{K})= \bigcup_{1\leq j_1,\dots, j_{n-1}\leq n}\mathfrak{p}_n^{-1}(\bigcap_{i=1}^{n-1}V^n_{n-1}(\Phi^{\#}_{j_i}(HD^{i-1}_{n-1}\mathbf{x}_{n-1}))).
    \end{equation} 
    We note that for any $1\leq j_i\leq n$, the polynomial $\Phi^{\#}_{j_i}(HD^{i-1}_{n-1}\mathbf{x}_{n-1})\in \mathbb{K}[x_1,\dots, x_{n-1}]$ is homogeneous of degree $n-i$ for all $1\leq i\leq n-1$. Proposition~\ref{Prop4} implies that Conjecture~\ref{con1} holds in degree $n$ over $\mathbb{K}$ if and only if $\bigcap_{i=1}^{n-1}V^n_{n-1}(\Phi^{\#}_{j_i}(HD^{i-1}_{n-1}\mathbf{x}_{n-1}))=\{(0,0,\dots, 0)\}\subseteq \mathbb{A}^{n-1}_{\mathbb{K}}(\mathbb{K})$ for any choice of $1\leq j_1, \dots, j_{n-1}\leq n$. The proposition follows from the fact that for any $r\geq 1$, a sequence of $r$ homogeneous polynomials in $r$ variables is regular if and only if the reduced affine variety defined by them is the origin in $\mathbb{A}^r_{\mathbb{K}}$.
\end{proof}

\begin{notation}\label{not:Phi}
For any polynomial $f\in \mathbb{K}[X]$, we will denote the coefficient of $X^j$ in $f$ by $[X^j]f$, and $H_i(f)=f_i$ will denote the $i^{th}$ Hasse-Schmidt derivative as in Section~\ref{subsec3.2}. For brevity, we will let  $m=n-1$, $S=\mathbb{K}[x_1,\ldots,x_m]$ and $x_n:=0$ in $S$. Let $e_d=e_d(x_1,\dots, x_m)=HD^{m-d}_m\mathbf{x}_m\in S$ denote the elementary symmetric polynomial of degree $d$ in $x_1,\dots, x_m$, so that $e_0=1$. For $1\leq j\leq n$ and $0\leq d\leq m$, we set
\begin{equation}\label{eq:Phidef}
 \Gw{j}{d}:=\Phi_{j}^{\#}(HD^{m-d}_m\mathbf{x}_m)=\Phi^{\#}_{j}(e_d)\in S,
\end{equation}
which is a homogeneous polynomial of degree $d$ (we suppress $m$ in $\Phi_{j, d}$, when it is clear from the context). For a tuple $\bfa=(a_1,\ldots,a_m)\in[n]^m$ of ``\textit{witness labels}" $a_d=j_{n-d}$, we set
\begin{equation}\label{eq:Gdef}
 I_{\bfa}=(\Gw{a_1}{1},\ldots,\Gw{a_m}{m})\subseteq S, \ \ \ \ \ A_{\bfa}=S/I_{\bfa}.
\end{equation}
Thus $\Gw{a_1}{1},\ldots,\Gw{a_m}{m}$ lists the polynomials $\Phi^{\#}_{j_{n-1}}(HD^{n-2}_{n-1}\mathbf{x}_{n-1}),\dots, \Phi^{\#}_{j_1}(HD^{0}_{n-1}\mathbf{x}_{n-1})$ of Proposition~\ref{Lem5} in increasing degree: $\Gw{a_d}{d}$ is the polynomial of degree $d$ of this sequence, and $a_d$ is its witness label. The notation \eqref{eq:Phidef}, \eqref{eq:Gdef} will be used throughout the remainder of this section.
\end{notation}
 For any algebraically closed field $\mathbb{K}$ of characteristic $p\geq 0$, let 
\begin{equation}\label{Eq:Dp}
     D_p(n):=\dim\bigcap_{i=1}^{n-1}X_n^i(\mathbb{K}),
\end{equation}
which is independent of the choice of $\mathbb{K}$ (as long as same characteristic), since $\Gw{j}{d}\in \mathbb{Z}[x_1,\dots, x_m]$ for all $1\leq j\leq n$ and $0\leq d\leq m$. The following lemma relates $D_p(n)$ to the algebras $A_{\bfa}$ and to the geometric fibre $(X_n)_{\mathbb{K}}=X_n\times_{\spec(\mathbb{Z})}\spec(\mathbb{K})$ of the arithmetic Casas-Alvero scheme introduced in Section~\ref{subsec3.3}. Since $\mathbb{K}$ is algebraically closed, $X_n(\mathbb{K})$ is the set of closed points of $(X_n)_{\mathbb{K}}$, and its dimension as an algebraic set equals $\dim (X_n)_{\mathbb{K}}$.

\begin{lemma}\label{lem:dictionary}
With the convention $\dim\emptyset=-1$, we have
\begin{equation}\label{eq:dimCn}
 D_p(n)=1+\max_{\bfa}\dim A_{\bfa},\quad
 \dim (X_n)_{\K}=\dim X_n(\K)=D_p(n)-2.
\end{equation}
\end{lemma}
\begin{proof}
By \eqref{Eqn8} and \eqref{eq:Gdef}, $\bigcap_{i=1}^{n-1}X^i_n(\mathbb{K})=\bigcup_{\bfa\in[n]^m}\mathfrak{p}_n^{-1}(V(I_{\bfa}))$, where $V(I_{\bfa})\subseteq Z^n_n(\mathbb{K})$. The map $\mathfrak{p}_n$ restricts to the identity on $Z^n_n(\mathbb{K})$, and its fibers are the lines parallel to $(1,\dots, 1)$, so $\mathfrak{p}_n^{-1}(W)\cong W\times\mathbb{A}^1_{\mathbb{K}}(\mathbb{K})$ and $\mathfrak{p}_n^{-1}(W)\cap V(x_n)=W$ for every subset $W\subseteq Z^n_n(\mathbb{K})$. Since $0\in V(I_{\bfa})$ for every word, this gives $D_p(n)=1+\max_{\bfa}\dim V(I_{\bfa})=1+\max_{\bfa}\dim A_{\bfa}$, the last equality by the Nullstellensatz, and it shows that $\widehat{\mathcal{V}}_n(\mathbb{K})=\bigcup_{\bfa}V(I_{\bfa})$ has dimension $D_p(n)-1$. If $D_p(n)=1$, then $\widehat{\mathcal{V}}_n(\mathbb{K})=\{0\}$ and $X_n(\mathbb{K})=\emptyset$ by Proposition~\ref{Prop:Vieta}. If $D_p(n)\geq2$, then $\mathcal{V}_n(\mathbb{K})$ has dimension $D_p(n)-2$, and so does its image $X_n(\mathbb{K})$ under the quasi-finite morphism $\overline{\nu}_n$.
\end{proof}

\subsection{Coefficient bounds and the proofs of Theorems~\ref{mainthm2}, \ref{mainthm} and~\ref{mainthmC}}\label{subsec:spec}\label{subsec4.1}
Set $D_p(1)=1$ for all primes $p$. By construction, the structure morphism $\phi_n:X_n\rightarrow\Spec\mathbb Z$ is proper, whereby Chevalley's upper semicontinuity theorem \cite[Corollaire~13.1.5]{EGAIV} yields
\begin{equation}\label{eq:specialize}
 D_0(n)\le D_p(n)\quad(\text{for all prime } p).
\end{equation}

\begin{definition}
Let $p$ be a prime. For an integer $v=\sum_\nu v_\nu p^\nu$ expanded in base $p$, let
\[
S_p(v)=\left\{\sum_\nu u_\nu p^\nu:0\le u_\nu\le v_\nu\right\},
 \quad M_p(v)=|S_p(v)|=\prod_\nu(v_\nu+1).
\]
\end{definition}

\begin{lemma}\label{m:support}
Let $\mathbb K$ be an algebraically closed field of characteristic
$p>0$, and $f(X)=\sum_{j=0}^{n}c_jX^j\in\mathbb K[X]$
be monic of degree $n$. Suppose that $\gcd(f,H_i(f))\ne 1$ for every $1\le i\le n-1$. Then $c_j=0$ for every $j\notin S_p(n)$. Furthermore, if $n=p^e u$ such that $p\nmid u$, then $D_p(n)=D_p(u)$.
\end{lemma}
\begin{proof}
Assume for contradiction that there exists $1\leq j\leq n-1$ such that $j\notin S_p(n)$ and $c_j\neq0$, and let $j_0$ be the largest such index. Thus, every $l>j_0$ with $c_l\ne0$ lies in $S_p(n)$, so by Lucas's theorem, $\binom l{j_0}=0$ in $\K$. Hence $H_{j_0}(f)=c_{j_0}\neq 0$, which contradicts the choice of $f(X)$. Now let $q=p^e$, so that $n=qu$, with $p\nmid u$. Define the closed immersion
\[
 \iota_q:\mathbb A^u_{\K}\rightarrow\mathbb A^n_{\K},
 \quad
 \mathbf{z}:=(z_1,\ldots,z_u)\mapsto(y_1,\ldots,y_n),
 \quad
 y_j=
 \begin{cases}
  z_{j/q},&q\mid j,\\
  0,&q\nmid j.
 \end{cases}
\]
By the definition of the coefficient maps in \eqref{Eqn:relations}, $\kappa_n\bigl(\iota_q(\mathbf{z})\bigr)(X)=\kappa_u(\mathbf{z})(X^q)$. Since $S_p(n)=qS_p(u)$, the first part of the lemma implies that $\bigcap_{i=1}^{n-1}\Delta_n^i(\K)\subseteq \iota_q(\mathbb{A}^u_\mathbb{K})$. For $\mathbf{z}\in\mathbb A^u_{\K}(\K)$, let $g(X):=\kappa_u(\mathbf{z})$ and $f(X)=g(X^q)$. Comparing the coefficients of $T^i$ in the identity
\[
 f(X+T)=g(X^q+T^q)
 =g(X^q)+\sum_{k=1}^{u}H_k(g)(X^q)T^{qk}
\]
yields $H_i(f)=0$ if $q\nmid i$, and $H_{qk}(f)=H_k(g)(X^q)$ if $1\leq k<u$. Thus, for every $1\leq k<u$,
\[
\begin{aligned}
 \gcd(f,H_{qk}(f))\neq1
 &\Longleftrightarrow
 \exists\alpha\in\K:\quad
 g(\alpha^q)=H_k(g)(\alpha^q)=0\\
 &\Longleftrightarrow
 \exists\beta\in\K:\quad g(\beta)=H_k(g)(\beta)=0\\
 &\Longleftrightarrow \gcd(g,H_k(g))\neq1,
\end{aligned}
\]
where the second equivalence uses the surjectivity of
$\alpha\mapsto\alpha^q$ on the algebraically closed field $\K$.
For $q\nmid i$, the condition $\gcd(f,H_i(f))\neq1$ is
automatic, since $H_i(f)=0$ and $f$ is nonconstant. Together
with the preceding containment and \eqref{Eqn:relations}, this proves that
$\iota_q$ restricts to an isomorphism of affine algebraic sets
\begin{equation}\label{Eq:isoDelta}
 \iota_q:
 \bigcap_{k=1}^{u-1}\Delta_u^k(\K)
 \xrightarrow{\ \sim\ }
 \bigcap_{i=1}^{n-1}\Delta_n^i(\K).
\end{equation}
Indeed, the inverse is the restriction of the natural
projection map $\mathbf{A}^n_{\mathbb{K}}\rightarrow \mathbb{A}^u_{\mathbb{K}}$ given by $(y_1,\ldots,y_n)\mapsto(y_q,y_{2q},\ldots,y_{uq})$. By Section~\ref{subsec3.3} the Vieta map $\nu_n:\mathbb{A}^n_\mathbb{K}\rightarrow\mathbb{A}^n_\mathbb{K}$ is finite and $\nu_n(\bigcap_{i=1}^{n-1}X^i_n(\mathbb{K}))= \bigcap_{i=1}^{n-1}\Delta^i_n(\mathbb{K})$, thereby yielding
\[
 D_p(n):=\dim\bigcap_{i=1}^{n-1}X_n^i(\mathbb{K})=\dim\bigcap_{i=1}^{n-1}\Delta_n^i(\K)
 \quad(n\geq2).
\]
Thus, \eqref{Eq:isoDelta} implies $D_p(n)=D_p(u)$, which completes the proof.

\end{proof}

Before proving the bounds on $D_p(n)$, we will require the following linear-algebraic lemma.

\begin{lemma}\label{m:rank}
Let $\mathbb{K}$ be an algebraically closed field and let $W\subseteq\mathbb{K}[X]$ be an $M$-dimensional $\mathbb{K}$-linear subspace containing $1$. Fix integers $r\geq1$ and $0\leq h<M$. For
$\mathbf{a}=(a_1,\dots,a_r)\in\mathbb{K}^r$, let $\operatorname{ev}_{\mathbf{a}}:W\rightarrow\mathbb{K}^r$ be the evaluation map defined by $w\mapsto\bigl(w(a_1),\dots,w(a_r)\bigr)$. Let $E_{r,h}(W)\subseteq\mathbb{A}^r_{\mathbb{K}}$ be the reduced
locally closed subscheme whose $\mathbb{K}$-rational points are
\[
 E_{r,h}(W)(\mathbb{K}) := \left\{ (a_1,\dots,a_r)\in\mathbb{K}^r:
 a_i\neq a_j\text{ for }i\neq j,\quad \rank(\operatorname{ev}_{\mathbf{a}})=h\right\}.
\]
Then $\dim E_{r,h}(W)\leq h$.
\end{lemma}
\begin{proof}
Let $E:=E_{r,h}(W)$ and assume $E\neq\emptyset$. Consider the incidence subscheme
\[
 \mathcal{I} := \left\{ (\mathbf{a},[w])\in E\times\mathbb{P}(W):
 w(a_i)=0\text{ for }1\leq i\leq r \right\},
\]
where $\mathbb{P}(W)$ parametrizes one-dimensional subspaces of $W$. By definition of $E$, the evaluation morphism  $W\otimes_{\mathbb{K}}\mathcal{O}_E\rightarrow\mathcal{O}_E^{\oplus r}$ has constant rank $h$, so its kernel is locally free of rank $M-h$. Consequently, the projection $\mathcal{I}\to E$ is a projective bundle with fibers $\mathbb{P}^{M-h-1}_{\mathbb{K}}$, and hence
\[
 \dim\mathcal{I}=\dim E+M-h-1.
\]

On the other hand, the fiber of the second projection $\mathcal{I}\to\mathbb{P}(W)$ over $[w]$ is contained in $V(w)^r$. Since $w$ is a non-zero univariate polynomial, $V(w)$ is finite. Thus this projection has finite fibers, and the fiber-dimension inequality
\cite[Lemma~33.20.4]{stacks-project} yields $\dim\mathcal{I}\leq\dim\mathbb{P}(W)=M-1$. Combining the two dimension bounds yields $\dim E\leq h$.
\end{proof}

\subsubsection{Bounds on $D_p(n)$ based on arithmetic of $n$}\label{subsubsec:arithmeticboundn}
The following proposition provides the first arithmetic bound on $D_p(n)$.

\begin{proposition}\label{m:halfsupport}
For every $n\geq1$ and prime $p$, $D_p(n)\leq\left\lfloor\frac{M_p(n)}2\right\rfloor$.
\end{proposition}
\begin{proof}
Let $\K$ be algebraically closed of characteristic $p$.
For $n=1$, the assertion follows from $D_p(1)=1$ and $M_p(1)=2$.
For $n\geq2$, we have
\[
 D_p(n)=\dim\bigcap_{i=1}^{n-1}\Delta_n^i(\K),
\]
as established in the proof of Lemma~\ref{m:support}.
Furthermore, if $M_p(n)=2$, Lemma~\ref{m:support} shows that every polynomial parametrized
by this intersection has the form $X^n+c_0$, proving the bound. Thus, assume $M_p(n)\geq3$, and set $M=\lfloor M_p(n)/2\rfloor$.
Let $B$ consist of the $M$ smallest elements of $S_p(n)$, and set
\[
 J=S_p(n)\setminus(B\cup\{n\}),\quad
 k=|J|,\quad W=\langle X^l:l\in B\rangle_{\K}\subseteq \K[X].
\]
Then $\dim W=M$, $1\in W$, and
$1\leq k=M_p(n)-M-1\leq M$. Moreover, $l<j$ whenever
$l\in B$ and $j\in J$.

For $\mathbf y\in\K^n$, write $f_{\mathbf y}=\kappa_n(\mathbf y)$.
Let $\mathcal I\subseteq\A^n_{\K}\times\A^k_{\K}$ be the reduced
closed subscheme defined by
\[\mathcal{I}:=\left\{(\mathbf y, \mathbf b)\in\A^n_{\K}\times\A^k_{\K}\mid  [X^l]f_{\mathbf y}=0 \text{ if } l\notin S_p(n) \text{ and }  f_{\mathbf y}(b_j)=H_j(f_{\mathbf y})(b_j)=0 \text{ if } j\in J\right\},
\]
and let $\theta$ and $\pi$ be its projections to $\A^n_{\K}$ and
$\A^k_{\K}$, respectively. By Lemma~\ref{m:support} and
\eqref{Eqn:relations}, every
$\mathbf y\in\bigcap_{i=1}^{n-1}\Delta_n^i(\K)\subseteq \A^n_\K$ lifts to
$\mathcal I$ under $\theta$ by choosing a common root $b_j$ of
$f_{\mathbf y}$ and $H_j(f_{\mathbf y})$ for each $j\in J$. Consequently,
\[
 D_p(n)\leq\dim\overline{\theta(\mathcal I)}\leq\dim\mathcal I.
\]

We bound $\dim\mathcal I$ using the fibers of $\pi:\mathcal{I}\rightarrow \A^k_\K$.
Fix $\mathbf b=(b_j)_{j\in J}\in\A^k_\K(\K)$. By definition, every monic polynomial supported on $S_p(n)$ (i.e., with only coefficients indexed by $S_p(n)$ being non-zero) has a unique expression
\[
 f=X^n+\sum_{j\in J}c_jX^j+v,\quad v\in W.
\]
Since $H_j(v)=0$ for $j\in J$, the derivative equations yield
\begin{equation}\label{Eq:recursion}
 H_j(f)(b_j)=0
 \quad\Longleftrightarrow\quad
 c_j=-\sum_{\substack{l\in J\cup\{n\}\\l>j}}
          \binom lj c_l b_j^{\,l-j},
 \quad c_n=1.
\end{equation}
They determine the $c_j$ uniquely, in decreasing order of $j$,
independently of $v$. Let $P_\mathbf{b}(X):=X^n+\sum_{j\in J}c_jX^j$ denote the resulting polynomial, whereby the remaining equations determining $\mathcal{I}$, identify the fiber $\pi^{-1}(\mathbf b)$ with
\[
 \left\{v\in W:
 \operatorname{ev}_{\mathbf b}(v)
       =-\bigl(P_{\mathbf b}(b_j)\bigr)_{j\in J}\right\},
 \quad
 \operatorname{ev}_{\mathbf b}(v)=\bigl(v(b_j)\bigr)_{j\in J}.
\]
Thus, if $\rank(\operatorname{ev}_{\mathbf b})=h$, this fiber is
either empty or an affine space of dimension $M-h$. For $0\leq h\leq M$, let $T_h\subseteq\A^k_{\K}$ be the reduced
locally closed locus where $\operatorname{ev}_{\mathbf b}$ has
rank $h$. For $h<M$, this locus is covered by the images of
$E_{r,h}(W)$ (as defined in Lemma~\ref{m:rank}) under the coordinate-repetition maps
\[
 \delta_\sigma:\A^r_{\K}\rightarrow\A^k_{\K},\quad
 (u_1,\ldots,u_r)\mapsto
       \bigl(u_{\sigma(j)}\bigr)_{j\in J},
 \quad \sigma:J\twoheadrightarrow\{1,\ldots,r\},
\]
where $1\leq r\leq k$ is the number of distinct elements among the coordinates of $\mathbf{b}=(b_j)_{j\in J}\in \A^k_\K(\K)$. Indeed, deleting repeated entries from
$\mathbf b$ does not change its evaluation rank. There are only
finitely many such maps, so Lemma~\ref{m:rank} gives
$\dim T_h\leq h$. For $h=M$, we use the trivial bound $\dim T_M\leq k$.
The fiber-dimension inequality
\cite[Lemma~33.20.4]{stacks-project} now yields, for each
nonempty $T_h$,
\[
 \dim\pi^{-1}(T_h)
 \leq\dim T_h+M-h
 \leq
 \begin{cases}
 M,&h<M,\\
 k\leq M,&h=M.
 \end{cases}
\]
Since the finitely many loci $\pi^{-1}(T_h)$ cover $\mathcal I$,
we conclude that $D_p(n)\leq\dim\mathcal I\leq M$.
\end{proof}

Our next bound on $D_p(n)$ utilizes the quotient and residue of $n$ modulo $p^e$ for any $e\geq 1$.

\begin{theorem}\label{m:blocks}
Let $q=p^e$, where $p$ is prime and $e\geq1$, and write $n=aq+s$ with $a\geq1$ and $0<s<q$. Then $D_p(n)\leq\max\{2,(M_p(a)-1)(M_p(s)-1)\}$.
\end{theorem}
\begin{proof}
Let $\mathbb{K}$ be algebraically closed of characteristic $p$. For brevity, set $M_1=M_p(a)$, $M_2=M_p(s)$, $M=(M_1-1)(M_2-1)$ and $k=M_1+M_2-2$, and let $S_a^-=S_p(a)\setminus\{a\}$ and $S_s^-=S_p(s)\setminus\{s\}$, so that $|S_a^-|=M_1-1$ and $|S_s^-|=M_2-1$. Since $a,s\geq1$, we have $M_1,M_2\geq2$. Since $0<s<q=p^e$, the nonzero base-$p$ digits of $aq$ occur in positions at least $e$, whereas those of $s$ occur in positions below $e$. Consequently,
\begin{equation}\label{Eq:supportfactor}
S_p(n)=\{vq+j\mid \ v\in S_p(a),\ j\in S_p(s)\},
\end{equation}
where each element of $S_p(n)$ has a unique such representation, as $0\leq j<q$.

As in the proof of Proposition~\ref{m:halfsupport}, write $f_{\mathbf{y}}=\kappa_n(\mathbf{y})$ for $\mathbf{y}\in\mathbb{A}^n_{\mathbb{K}}(\mathbb{K})$, and let $\mathcal{I}\subseteq\mathbb{A}^n_{\mathbb{K}}\times\mathbb{A}^{M_2-1}_{\mathbb{K}}\times\mathbb{A}^{M_1-1}_{\mathbb{K}}$ be the reduced closed subscheme whose $\mathbb{K}$-rational points are the tuples $(\mathbf{y},\mathbf{b},\mathbf{w})$, with $\mathbf{b}=(b_j)_{j\in S_s^-}$ and $\mathbf{w}=(w_v)_{v\in S_a^-}$, satisfying the following three conditions:
\begin{enumerate}[(i)]
\item $[X^l]f_{\mathbf{y}}=0$ for all $0\leq l<n$ with $l\notin S_p(n)$,
\item $f_{\mathbf{y}}(b_j)=H_{aq+j}(f_{\mathbf{y}})(b_j)=0$ for all $j\in S_s^-$,
\item $f_{\mathbf{y}}(w_v)=H_{vq+s}(f_{\mathbf{y}})(w_v)=0$ for all $v\in S_a^-$.
\end{enumerate}
Note that the derivative orders $aq+j$ and $vq+s$ occurring in (ii) and (iii) all lie in $\{1,\dots,n-1\}$. Let $\theta$ and $\pi$ be the projections of $\mathcal{I}$ to $\mathbb{A}^n_{\mathbb{K}}$ and to $\mathbb{A}^{M_2-1}_{\mathbb{K}}\times\mathbb{A}^{M_1-1}_{\mathbb{K}}$, respectively. If $\mathbf{y}\in\bigcap_{i=1}^{n-1}\Delta^i_n(\mathbb{K})$, then $f_{\mathbf{y}}$ satisfies (i) by Lemma~\ref{m:support}, and by \eqref{Eqn:relations} it has a common root with each of the derivatives $H_{aq+j}(f_{\mathbf{y}})$ and $H_{vq+s}(f_{\mathbf{y}})$; choosing such common roots as the $b_j$ and $w_v$ lifts $\mathbf{y}$ to $\mathcal{I}$ under $\theta$. Hence, exactly as in the proof of Proposition~\ref{m:halfsupport},
\[
D_p(n)=\dim\bigcap_{i=1}^{n-1}\Delta^i_n(\mathbb{K})\leq\dim\overline{\theta(\mathcal{I})}\leq\dim\mathcal{I}.
\]

We bound $\dim\mathcal{I}$ using the fibers of $\pi$. Fix $(\mathbf{b},\mathbf{w})\in\mathbb{A}^{M_2-1}_{\mathbb{K}}(\mathbb{K})\times\mathbb{A}^{M_1-1}_{\mathbb{K}}(\mathbb{K})$. By \eqref{Eq:supportfactor}, a monic polynomial $f$ of degree $n$ satisfies (i) if and only if it has a (necessarily unique) expression
\[
f(X)=\sum_{v\in S_p(a)}\sum_{j\in S_p(s)}c_{v,j}X^{vq+j} \ \text{ with } c_{a,s}=1.
\]
By Lucas's theorem, $\binom{uq+l}{vq+j}=\binom{u}{v}\binom{l}{j}$ in $\mathbb{K}$ for all $u,v\geq0$ and $0\leq l,j<q$. Since $\binom{u}{a}=0$ for $u\in S_a^-$ and $\binom{l}{s}=0$ for $l\in S_s^-$, for such $f$ we obtain
\[
H_{aq+j}(f)=\sum_{\substack{l\in S_p(s)\\ l\geq j}}\binom{l}{j}c_{a,l}X^{l-j} \ \text{ and } \ H_{vq+s}(f)=\sum_{\substack{u\in S_p(a)\\ u\geq v}}\binom{u}{v}c_{u,s}X^{(u-v)q}
\]
for all $j\in S_s^-$ and $v\in S_a^-$. Thus the derivative conditions $H_{aq+j}(f)(b_j)=0$ and $H_{vq+s}(f)(w_v)=0$ in (ii) and (iii) are equivalent to
\begin{equation}\label{eq:blocks-recursion}
c_{a,j}=-\sum_{\substack{l\in S_p(s)\\ l>j}}\binom{l}{j}c_{a,l}b_j^{\,l-j} \ \ (j\in S_s^-) \ \text{ and } \ c_{v,s}=-\sum_{\substack{u\in S_p(a)\\ u>v}}\binom{u}{v}c_{u,s}w_v^{\,(u-v)q} \ \ (v\in S_a^-).
\end{equation}
Starting from $c_{a,s}=1$, the relations \eqref{eq:blocks-recursion} determine the coefficients $c_{a,j}$ for $j\in S_s^-$ and $c_{v,s}$ for $v\in S_a^-$ uniquely, in decreasing order of $j$ and of $v$ respectively, as polynomial functions of $\mathbf{b}$ and of $\mathbf{w}$ respectively. Let $P_{\mathbf{b},\mathbf{w}}(X)$ denote the sum of the resulting terms $c_{v,j}X^{vq+j}$ over all $(v,j)\in S_p(a)\times S_p(s)$ with $v=a$ or $j=s$. Then $f$ satisfies (i) together with the derivative conditions in (ii) and (iii) if and only if $f=P_{\mathbf{b},\mathbf{w}}+\eta$ for some $\eta$ in
\[
W:=\langle X^{vq+j}\mid \ v\in S_a^-,\ j\in S_s^-\rangle_{\mathbb{K}}\subseteq\mathbb{K}[X],
\]
which is an $M$-dimensional subspace containing $1$. Consequently, the $\mathbb{K}$-rational points of the fiber $\pi^{-1}(\mathbf{b},\mathbf{w})$ are in bijection with the set of $\eta\in W$ satisfying the remaining root conditions $\eta(b_j)=-P_{\mathbf{b},\mathbf{w}}(b_j)$ for all $j\in S_s^-$ and $\eta(w_v)=-P_{\mathbf{b},\mathbf{w}}(w_v)$ for all $v\in S_a^-$. This is a system of $k$ inhomogeneous linear equations in $\eta$, whose associated linear map is
\[
\operatorname{ev}_{\mathbf{b},\mathbf{w}}:W\rightarrow\mathbb{K}^k, \ \ \eta\mapsto\bigl((\eta(b_j))_{j\in S_s^-},(\eta(w_v))_{v\in S_a^-}\bigr).
\]
Thus, if $\rank\operatorname{ev}_{\mathbf{b},\mathbf{w}}=h$, then $\pi^{-1}(\mathbf{b},\mathbf{w})$ is either empty or an affine space of dimension $M-h$.

For $0\leq h\leq\min\{M,k\}$, let $T_h\subseteq\mathbb{A}^{M_2-1}_{\mathbb{K}}\times\mathbb{A}^{M_1-1}_{\mathbb{K}}=\mathbb{A}^k_{\mathbb{K}}$ be the reduced locally closed subscheme of tuples $(\mathbf{b},\mathbf{w})$ with $\rank\operatorname{ev}_{\mathbf{b},\mathbf{w}}=h$. These finitely many loci cover $\mathbb{A}^k_{\mathbb{K}}$, whereby $\dim\mathcal{I}=\max_h\dim\pi^{-1}(T_h)$. Since $1\in W$, the coordinate-repetition argument in the proof of Proposition~\ref{m:halfsupport} applies verbatim to the $k$ coordinates of $(\mathbf{b},\mathbf{w})$, and Lemma~\ref{m:rank} gives $\dim T_h\leq h$ for all $h<M$. As the fibers of $\pi$ over $T_h$ have dimension at most $M-h$, the fiber-dimension inequality \cite[Lemma~33.20.4]{stacks-project} yields $\dim\pi^{-1}(T_h)\leq\dim T_h+M-h\leq M$ for all $h<M$. The fibers of $\pi$ over $T_M$ are empty or single points, so that $\dim\pi^{-1}(T_M)\leq\dim T_M\leq k$. Therefore
\[
D_p(n)\leq\dim\mathcal{I}\leq\max\{M,k\}.
\]
If $M_1,M_2\geq3$, then $M-k=(M_1-2)(M_2-2)-1\geq0$, and if $M_1=M_2=2$, then $M=1$ and $k=2$. In both cases $\max\{M,k\}=\max\{2,M\}$, which proves the theorem in these cases. It remains to consider the case where exactly one of $M_1$ and $M_2$ equals $2$. Then $M\geq2$ and $k=M+1$, and by the above it suffices to show that $\dim\mathcal{I}_M\leq M$, where $\mathcal{I}_M:=\pi^{-1}(T_M)$.

Write $\mathbb{A}^k_{\mathbb{K}}=\mathbb{A}^M_{\mathbb{K}}\times\mathbb{A}^1_{\mathbb{K}}$, where the first factor parametrizes the longer of the two tuples $\mathbf{b}$ and $\mathbf{w}$, namely $\mathbf{b}$ if $M_1=2$ and $\mathbf{w}$ if $M_2=2$, and let $\rho:\mathcal{I}_M\rightarrow\mathbb{A}^M_{\mathbb{K}}$ be the composition of $\pi|_{\mathcal{I}_M}$ with the projection to the first factor. We claim that there is a nonempty open subset $U\subseteq\mathbb{A}^M_{\mathbb{K}}$ such that $\rho$ has finite fibers over all $\mathbb{K}$-rational points of $U$. Granting the claim, the fiber-dimension inequality gives $\dim\rho^{-1}(U)\leq\dim U=M$. On the other hand, $C:=\mathbb{A}^M_{\mathbb{K}}\setminus U$ is a proper closed subset, so $\dim(C\times\mathbb{A}^1_{\mathbb{K}})\leq M$, and the closed subset $\mathcal{I}_M\setminus\rho^{-1}(U)=\mathcal{I}_M\cap\pi^{-1}(C\times\mathbb{A}^1_{\mathbb{K}})$ maps to $C\times\mathbb{A}^1_{\mathbb{K}}$ under $\pi$ with fibers that are empty or single points, whereby $\dim(\mathcal{I}_M\setminus\rho^{-1}(U))\leq M$ as well. Since every irreducible component of $\mathcal{I}_M$ either meets the open subset $\rho^{-1}(U)$ or is contained in its complement, it follows that $\dim\mathcal{I}_M\leq M$, as required. We now prove the claim, treating the two cases separately.

First suppose $M_1=2<M_2$. Then $S_a^-=\{0\}$, $M=M_2-1\geq2$ and $W=\langle X^j\mid \ j\in S_s^-\rangle_{\mathbb{K}}$. Writing $w=w_0$, the second relation in \eqref{eq:blocks-recursion} reads $c_{0,s}=-w^{aq}$, so that
\[
P_{\mathbf{b},w}(X)=X^{aq}g_{\mathbf{b}}(X)-w^{aq}X^s, \ \text{ where } g_{\mathbf{b}}(X):=X^s+\sum_{j\in S_s^-}c_{a,j}X^j
\]
has coefficients depending only on $\mathbf{b}$. Fix $t\in S_s^-\setminus\{0\}$, which exists since $M\geq2$, and let $U\subseteq\mathbb{A}^M_{\mathbb{K}}$ be the locus of those $\mathbf{b}$ for which the linear map $\operatorname{ev}_{\mathbf{b}}:W\rightarrow\mathbb{K}^M$, $\eta\mapsto(\eta(b_j))_{j\in S_s^-}$, is an isomorphism and $g_{\mathbf{b}}(b_t)\neq0$. The matrix of $\operatorname{ev}_{\mathbf{b}}$ with respect to the monomial basis of $W$ is $(b_j^{\,l})_{j,l\in S_s^-}$, whose determinant is a nonzero polynomial in $\mathbf{b}$, being a signed sum of the $M!$ pairwise distinct monomials $\prod_{j\in S_s^-}b_j^{\,\sigma(j)}$, where $\sigma$ ranges over the permutations of $S_s^-$. The polynomial function $\mathbf{b}\mapsto g_{\mathbf{b}}(b_t)$ is nonzero as well: for $b_0=1$ and $b_j=0$ for all $j\in S_s^-\setminus\{0\}$, the relations \eqref{eq:blocks-recursion} give $c_{a,j}=0$ for all $j\neq0$ and $c_{a,0}=-1$, so that $g_{\mathbf{b}}(X)=X^s-1$ and $g_{\mathbf{b}}(b_t)=-1$. Thus $U$ is the complement of the zero locus of the product of two nonzero polynomial functions on $\mathbb{A}^M_{\mathbb{K}}$, and hence a nonempty open subset.

Now fix $\mathbf{b}\in U(\mathbb{K})$. Since $\operatorname{ev}_{\mathbf{b}}$ is an isomorphism, there are unique $R_{\mathbf{b}},S_{\mathbf{b}}\in W$ such that $R_{\mathbf{b}}(b_j)=g_{\mathbf{b}}(b_j)$ and $S_{\mathbf{b}}(b_j)=b_j^{aq}g_{\mathbf{b}}(b_j)$ for all $j\in S_s^-$, and $R_{\mathbf{b}}\neq0$ since $R_{\mathbf{b}}(b_t)=g_{\mathbf{b}}(b_t)\neq0$. For any $w\in\mathbb{K}$, the polynomial
\[
f_w(X):=(X^{aq}-w^{aq})g_{\mathbf{b}}(X)+w^{aq}R_{\mathbf{b}}(X)-S_{\mathbf{b}}(X)
\]
is of the form $P_{\mathbf{b},w}+\eta$ with $\eta\in W$, because $g_{\mathbf{b}}-X^s\in W$, and it vanishes at $b_j$ for all $j\in S_s^-$. By the injectivity of $\operatorname{ev}_{\mathbf{b}}$, it is the only polynomial of the form $P_{\mathbf{b},w}+\eta$ with $\eta\in W$ having the latter property. Hence, if $(\mathbf{y},\mathbf{b},w)\in\mathcal{I}(\mathbb{K})$, then $f_{\mathbf{y}}=f_w$, and the condition $f_{\mathbf{y}}(w)=0$ in (iii) gives
\[
0=f_w(w)=w^{aq}R_{\mathbf{b}}(w)-S_{\mathbf{b}}(w).
\]
Every exponent occurring in $X^{aq}R_{\mathbf{b}}(X)$ is at least $aq$, whereas every exponent occurring in $S_{\mathbf{b}}$ lies in $S_s^-$ and is thus less than $s<aq$. Since $R_{\mathbf{b}}\neq0$, the polynomial $X^{aq}R_{\mathbf{b}}(X)-S_{\mathbf{b}}(X)$ is therefore nonzero, so $w$ is one of its finitely many roots, and each such $w$ determines $f_{\mathbf{y}}=f_w$ and hence $\mathbf{y}$. Thus $\rho^{-1}(\mathbf{b})$ is finite, which proves the claim in this case.

Now suppose $M_2=2<M_1$. Then $S_s^-=\{0\}$, $M=M_1-1\geq2$ and $W=\langle X^{vq}\mid \ v\in S_a^-\rangle_{\mathbb{K}}$. Writing $b=b_0$, the first relation in \eqref{eq:blocks-recursion} reads $c_{a,0}=-b^s$, so that
\[
P_{b,\mathbf{w}}(X)=X^sF_{\mathbf{w}}(X^q)-b^sX^{aq}, \ \text{ where } F_{\mathbf{w}}(Y):=Y^a+\sum_{v\in S_a^-}c_{v,s}Y^v
\]
has coefficients depending only on $\mathbf{w}$. Fix $t\in S_a^-\setminus\{0\}$, and let $U\subseteq\mathbb{A}^M_{\mathbb{K}}$ be the locus of those $\mathbf{w}$ for which $\operatorname{ev}_{\mathbf{w}}:W\rightarrow\mathbb{K}^M$, $\eta\mapsto(\eta(w_v))_{v\in S_a^-}$, is an isomorphism and $F_{\mathbf{w}}(w_t^q)\neq0$. As in the previous case, $U$ is a nonempty open subset: the matrix $(w_v^{\,uq})_{v,u\in S_a^-}$ of $\operatorname{ev}_{\mathbf{w}}$ has nonzero determinant as a polynomial in $\mathbf{w}$, and for $w_0=1$ and $w_v=0$ for all $v\in S_a^-\setminus\{0\}$, the relations \eqref{eq:blocks-recursion} give $F_{\mathbf{w}}(Y)=Y^a-1$ and $F_{\mathbf{w}}(w_t^q)=-1$. For $\mathbf{w}\in U(\mathbb{K})$, let $R_{\mathbf{w}},S_{\mathbf{w}}\in W$ be the unique polynomials with $R_{\mathbf{w}}(w_v)=F_{\mathbf{w}}(w_v^q)$ and $S_{\mathbf{w}}(w_v)=w_v^sF_{\mathbf{w}}(w_v^q)$ for all $v\in S_a^-$, so that $R_{\mathbf{w}}\neq0$. Arguing as in the previous case (now using $F_{\mathbf{w}}(X^q)-X^{aq}\in W$), if $(\mathbf{y},b,\mathbf{w})\in\mathcal{I}(\mathbb{K})$, then
\[
f_{\mathbf{y}}(X)=(X^s-b^s)F_{\mathbf{w}}(X^q)+b^sR_{\mathbf{w}}(X)-S_{\mathbf{w}}(X) \ \text{ and } \ b^sR_{\mathbf{w}}(b)-S_{\mathbf{w}}(b)=0.
\]
Every exponent occurring in $X^sR_{\mathbf{w}}(X)$ is congruent to $s$ modulo $q$, whereas every exponent occurring in $S_{\mathbf{w}}$ is divisible by $q$. As $0<s<q$ and $R_{\mathbf{w}}\neq0$, the polynomial $X^sR_{\mathbf{w}}(X)-S_{\mathbf{w}}(X)$ is nonzero, and the finiteness of $\rho^{-1}(\mathbf{w})$ follows as before. This proves the claim, and hence the theorem.
\end{proof}

\subsubsection{A component estimate from coefficient collisions}\label{subsubsec:component}
Our last bound on $D_p(n)$ is obtained from the Zariski tangent spaces of the affine algebraic sets $V(I_{\bfa})\subseteq\mathbb{A}^m_{\mathbb{K}}(\mathbb{K})$ of \eqref{eq:Gdef}. We keep the notation $e_d=HD^{m-d}_m\mathbf{x}_m$, $e_0=1$, $\Gw{j}{d}=\Phi^{\#}_{j}(e_d)$ and the convention $x_n=0$ of Notation~\ref{not:Phi}. We first express the polynomials $\Gw{j}{d}$ in terms of $e_0,\dots, e_d$.

\begin{lemma}\label{lem:shift}
For all $1\leq j\leq n$ and $0\leq d\leq m$, we have in $S$
\begin{equation}\label{eq:shift-identity}
 \Gw{j}{d}=\sum_{h=0}^{d}\binom{n-h}{d-h}(-x_j)^{d-h}e_h.
\end{equation}
In particular, $\Gw{n}{d}=e_d$ for all $d$, and $\Gw{j}{m}=\prod_{l\in[n]\setminus\{j\}}(x_l-x_j)$ for all $1\leq j\leq n$.
\end{lemma}
\begin{proof}
Let $v$ be an indeterminate over $S$. Expanding the product gives the identity $\prod_{l=1}^{m}(1+x_lv)=\sum_{h=0}^{m}e_hv^h$ in $S[v]$. Apply to it the $\mathbb{K}[v]$-algebra automorphism of $S[v]$ induced by $\Phi^{\#}_j$. If $j\leq m$, then $\Phi^{\#}_j(1+x_jv)=1-x_jv$ and $\Phi^{\#}_j(1+x_lv)=1+(x_l-x_j)v$ for $l\neq j$, by the definition of $\Phi^{\#}_j$ in Remark~\ref{Rem5new}; if $j=n$, then $\Phi^{\#}_n$ is the identity. In both cases, since $x_n=0$, we obtain
\[
 \sum_{h=0}^{m}\Gw{j}{h}v^h=\prod_{l\in[n]\setminus\{j\}}\bigl(1+(x_l-x_j)v\bigr)=\prod_{l=1}^{n}\bigl((1-x_jv)+x_lv\bigr)=\sum_{h=0}^{n}e_h(x_1,\dots, x_n)v^h(1-x_jv)^{n-h}.
\]
Here the second equality holds because the factor with index $l=j$ of the middle product equals $1$, and the third one is obtained by choosing, in each of the $n$ factors, either the summand $1-x_jv$ or the summand $x_lv$. As $x_n=0$, we have $e_h(x_1,\dots, x_n)=e_h$ for $h\leq m$ and $e_n(x_1,\dots, x_n)=0$. Expanding $(1-x_jv)^{n-h}$ by the binomial theorem and comparing the coefficients of $v^d$ gives \eqref{eq:shift-identity}. The first special case follows from \eqref{eq:shift-identity}, since $x_n=0$, and the second one is the coefficient of $v^m$ in the first product above.
\end{proof}

For $\sigma\in\mathfrak{S}_n$ and a word $\bfa\in[n]^m$, put $\sigma\bfa=(\sigma(a_1),\dots, \sigma(a_m))$.

\begin{lemma}\label{lem:relabel}
For every $\sigma\in\mathfrak{S}_n$ there is a graded $\mathbb{K}$-algebra automorphism $\sigma^{\#}$ of $S$ such that $\sigma^{\#}(\Gw{j}{d})=\Gw{\sigma(j)}{d}$ for all $1\leq j\leq n$ and $0\leq d\leq m$. In particular, $\sigma^{\#}(I_{\bfa})=I_{\sigma\bfa}$ for every word $\bfa\in[n]^m$, so that $\sigma^{\#}$ induces a graded isomorphism $A_{\bfa}\cong A_{\sigma\bfa}$, and $\Gw{a_1}{1},\dots, \Gw{a_r}{r}$ is a regular sequence in $S$ if and only if $\Gw{\sigma(a_1)}{1},\dots, \Gw{\sigma(a_r)}{r}$ is, for every $1\leq r\leq m$.
\end{lemma}
\begin{proof}
Since the transpositions generate $\mathfrak{S}_n$ and the required property is preserved under composition, it suffices to construct $\tau^{\#}$ for a transposition $\tau$. If $\tau=\tau_{jk}$ with $1\leq j<k\leq m$, let $\tau^{\#}$ be the automorphism of $S$ interchanging $x_j$ and $x_k$. It fixes each $e_d$, and since the maps induced on $S$ compose in the opposite order, Lemma~\ref{Lem3} (with $i=n$) gives $\tau^{\#}\circ\Phi^{\#}_l=\Phi^{\#}_{\tau(l)}\circ\tau^{\#}$ for all $1\leq l\leq n$, trivially so for $l=n$; hence $\tau^{\#}(\Gw{l}{d})=\Gw{\tau(l)}{d}$. If $\tau=\tau_{jn}$ with $j\leq m$, put $\tau^{\#}=\Phi^{\#}_j$. By Proposition~\ref{Prop3}, $\Phi^{\#}_j$ is an involution, so $\Phi^{\#}_j(\Gw{j}{d})=e_d=\Gw{n}{d}$ and $\Phi^{\#}_j(\Gw{n}{d})=\Gw{j}{d}$, while for $l\notin\{j,n\}$ Lemma~\ref{Lem3} gives $\Phi^{\#}_j\circ\Phi^{\#}_l=\Phi^{\#}_l\circ\tau^{\#}_{jl}$, whence $\Phi^{\#}_j(\Gw{l}{d})=\Gw{l}{d}$. All these automorphisms are linear substitutions of the variables, hence graded, and the remaining assertions follow from \eqref{eq:Gdef}.
\end{proof}

As $\mathfrak{S}_n$ acts transitively on $[n]$, Lemma~\ref{lem:relabel} allows us to assume $a_1=n$ whenever convenient, without affecting $\dim A_{\bfa}$ and hence, by Lemma~\ref{lem:dictionary}, without affecting $D_p(n)$. Note that $\Gw{a_1}{1}=e_1$ if $a_1=n$, by Lemma~\ref{lem:shift}.

By an \textit{irreducible closed cone} we mean a nonempty irreducible closed subset $Z\subseteq\mathbb{A}^m_{\mathbb{K}}(\mathbb{K})$ whose defining ideal $I(Z)\subseteq S$ is a homogeneous prime ideal. Since $I_{\bfa}$ is homogeneous, every irreducible component of $V(I_{\bfa})$ is an irreducible closed cone. For an irreducible closed cone $Z$, define an equivalence relation $\sim_Z$ on $[n]$ by $l\sim_Z l'$ if and only if $x_l-x_{l'}\in I(Z)$, where $x_n=0$. Let $C_1,\dots, C_r$ be its equivalence classes, numbered so that $n\in C_r$, and put $\mu_i=|C_i|$, so that $\mu_1+\dots+\mu_r=n$. We call $r$ the \textit{number of coordinate classes} of $Z$ and $\mu_1,\dots, \mu_r$ the \textit{multiplicities} of $Z$. Let $\lambda:\mathbb{A}^{r-1}_{\mathbb{K}}(\mathbb{K})\rightarrow\mathbb{A}^m_{\mathbb{K}}(\mathbb{K})$ be the injective linear map sending $u=(u_1,\dots, u_{r-1})$ to the point $x$ with $x_l=u_i$ for $l\in C_i$, $1\leq i<r$, and $x_l=0$ for $l\in C_r$. Its image is the linear subspace $\Lambda=V(x_l-x_{l'}\mid \ l\sim_Z l')$, which contains $Z$, and $\lambda$ induces an isomorphism of $\mathbb{A}^{r-1}_{\mathbb{K}}(\mathbb{K})$ onto $\Lambda$.

\begin{notation}\label{not:bar}
For $g\in S=\mathbb{K}[x_1,\dots, x_m]$, we write $\overline{g}=g\circ\lambda\in\mathbb{K}[u_1,\dots, u_{r-1}]$. Thus $g\mapsto\overline{g}$ is the surjective $\mathbb{K}$-algebra homomorphism $S\rightarrow\mathbb{K}[u_1,\dots, u_{r-1}]$ with $\overline{x_l}=u_i$ for $l\in C_i$, $i<r$, and $\overline{x_l}=0$ for $l\in C_r$; its kernel is the ideal $I(\Lambda)\subseteq I(Z)$ of $\Lambda$, so that $\overline{g}$ is the restriction of $g$ to $\Lambda$, written in the coordinates $u_1,\dots, u_{r-1}$ of $\Lambda$. Since $\lambda$ maps $Z':=\lambda^{-1}(Z)$ onto $Z$, we have $\overline{g}\in I(Z')$ if and only if $g\in I(Z)$.
\end{notation}
Applying $g\mapsto\overline{g}$ to the identity $\prod_{l=1}^{m}(1+x_lv)=\sum_{h=0}^{m}e_hv^h$ of the proof of Lemma~\ref{lem:shift} gives
\begin{equation}\label{eq:pullback-gen}
 \prod_{i=1}^{r-1}(1+u_iv)^{\mu_i}=\sum_{h=0}^{m}\overline{e_h}\,v^h \ \ \text{ in } \mathbb{K}[u_1,\dots, u_{r-1}][v].
\end{equation}
In particular, $\overline{e_1}=\sum_{i<r}\mu_iu_i$, and if $u\in\mathbb{A}^{r-1}_{\mathbb{K}}(\mathbb{K})$ satisfies $\overline{e_h}(u)=0$ for all $1\leq h\leq m$, then $\prod_{i<r}(1+u_iv)^{\mu_i}=1$ in $\mathbb{K}[v]$, so each $1+u_iv$ is a unit of $\mathbb{K}[v]$ and $u=0$.

For a point $u\in\mathbb{A}^{r-1}_{\mathbb{K}}(\mathbb{K})$ and $g\in\mathbb{K}[u_1,\dots, u_{r-1}]$, let $d_ug=\sum_{i=1}^{r-1}\frac{\partial g}{\partial u_i}(u)\,du_i$ be the differential of $g$ at $u$, a linear form on $\mathbb{K}^{r-1}$. For a closed subset $Y\subseteq\mathbb{A}^{r-1}_{\mathbb{K}}(\mathbb{K})$ with ideal $I(Y)$ and $u\in Y$, the Zariski tangent space $T_uY=\{\xi\in\mathbb{K}^{r-1}\mid \ d_ug(\xi)=0 \text{ for all } g\in I(Y)\}$ is dual to $\mathfrak{m}_u/\mathfrak{m}_u^2$, where $\mathfrak{m}_u$ is the maximal ideal of the local ring $\mathcal{O}_{Y,u}$. If $Y$ is irreducible, then $\dim Y=\dim\mathcal{O}_{Y,u}$, and $\dim\mathcal{O}_{Y,u}\leq\dim_{\mathbb{K}}\mathfrak{m}_u/\mathfrak{m}_u^2$ by Krull's height theorem and Nakayama's lemma; since $T_uY$ is contained in the common kernel of the linear forms $d_ug$, $g\in I(Y)$, we obtain
\begin{equation}\label{eq:tangent0}
 \dim Y\leq\dim_{\mathbb{K}}T_uY\leq(r-1)-\rank\{d_ug_1,\dots, d_ug_N\}
\end{equation}
for every irreducible closed subset $Y\subseteq\mathbb{A}^{r-1}_{\mathbb{K}}(\mathbb{K})$, every $u\in Y$ and all $g_1,\dots, g_N\in I(Y)$.

\begin{lemma}\label{lem:tangent}
Let $\bfa\in[n]^m$ be a word with $a_1=n$, and let $Z\subseteq V(I_{\bfa})$ be an irreducible closed cone with $r$ coordinate classes, whose multiplicities satisfy
\begin{equation}\label{eq:multiplicities}
 \sum_{i\in J}\mu_i\neq0 \ \text{ in } \mathbb{K} \ \text{ for every nonempty } J\subseteq\{1,\dots,r\}.
\end{equation}
Let $u\in\mathbb{A}^{r-1}_{\mathbb{K}}(\mathbb{K})$ be a nonzero point with $\overline{\Gw{a_d}{d}}(u)=0$ for all $1\leq d\leq m$, and let $k$ be the smallest index with $\overline{e_k}(u)\neq0$. Then $2\leq k\leq m$, and
\begin{enumerate}[(i)]
\item the linear forms $d_u\overline{\Gw{a_1}{1}},\dots, d_u\overline{\Gw{a_{k-1}}{k-1}}$ are linearly independent;
\item if moreover $u_z=0$ for some $z<r$, then $d_u\overline{\Gw{a_1}{1}},\dots, d_u\overline{\Gw{a_{k}}{k}}$ are linearly independent.
\end{enumerate}
\end{lemma}
\begin{proof}
The index $k$ exists since $u\neq0$, and $k\geq2$ since $\overline{e_1}=\overline{\Gw{a_1}{1}}$ vanishes at $u$. We write $b_d=x_{a_d}$, so that $b_1=0$ and $\overline{b_d}\in\{u_1,\dots, u_{r-1},0\}$ for all $d$, and by Lemma~\ref{lem:shift},
\begin{equation}\label{eq:Gbar}
 \overline{\Gw{a_d}{d}}=\sum_{h=0}^{d}\binom{n-h}{d-h}(-\overline{b_d})^{d-h}\,\overline{e_h} \ \ \ (1\leq d\leq m).
\end{equation}
We first isolate two facts about the point $u$ and the differentials of $\overline{e_1},\dots, \overline{e_k}$ at $u$, which are needed for the proofs of (i) and (ii).

\noindent\textbf{Step 1: the point $u$ has at least $k$ distinct nonzero coordinates.} For $j\geq1$ let $p_j=\sum_{l=1}^{m}x_l^{\,j}\in S$, so that $\overline{p_j}=\sum_{i<r}\mu_iu_i^{\,j}$. Newton's identities
\[
 p_j=\sum_{h=1}^{j-1}(-1)^{h-1}e_hp_{j-h}+(-1)^{j-1}je_j \ \ \ (j\geq1), \ \ \text{ with } e_h=0 \text{ for } h>m,
\]
hold in $\mathbb{Z}[x_1,\dots, x_m]$. Applying $g\mapsto\overline{g}$ and evaluating at $u$, where $\overline{e_h}(u)=0$ for $1\leq h<k$, we get $\overline{p_j}(u)=0$ for $1\leq j<k$. Let $\gamma_1,\dots, \gamma_t$ be the distinct nonzero elements among $u_1,\dots, u_{r-1}$, and put $\eta_l=\sum_{i:\,u_i=\gamma_l}\mu_i$, which is nonzero in $\mathbb{K}$ by \eqref{eq:multiplicities}. Then $\sum_{l=1}^{t}\eta_l\gamma_l^{\,j}=\overline{p_j}(u)=0$ for $1\leq j<k$. If we had $t<k$, the first $t$ of these equations would form a linear system in $\eta_1,\dots, \eta_t$ with matrix $(\gamma_l^{\,j})_{1\leq j,l\leq t}$, whose determinant $\prod_{l}\gamma_l\prod_{l<l'}(\gamma_{l'}-\gamma_l)$ is nonzero, forcing $\eta_1=\dots=\eta_t=0$. Hence $t\geq k$.

\noindent\textbf{Step 2: the differentials $d_u\overline{e_1},\dots, d_u\overline{e_{k-1}}$ are linearly independent.} The following computation is a variant of the classical evaluation of the Jacobian matrix of the elementary symmetric polynomials, cf.\ \cite[p.~170]{LP} and \cite[Chapter~I, (2.2)]{Mac}: the identity $\partial e_h/\partial x_l=\sum_{s=0}^{h-1}(-x_l)^se_{h-1-s}$ exhibits $(\partial e_h/\partial x_l)_{h,l}$ as the product of a unitriangular matrix with entries in $\mathbb{Z}[e_1,\dots, e_m]$ and the Vandermonde matrix $((-x_l)^{h-1})_{h,l}$, so that its determinant is $\prod_{l<l'}(x_l-x_{l'})$. We include the short derivation, since we need it for the polynomials $\overline{e_h}$, in which the multiplicities $\mu_i$ enter, and in arbitrary characteristic. For $1\leq l\leq m$ and $h\geq0$, let $e^{(l)}_h=e_h(x_1,\dots, \widehat{x_l},\dots, x_m)$ and $e^{(l)}_{-1}=0$. Separating the monomials of $e_h$ containing $x_l$ from those not containing it gives $e_h=e^{(l)}_h+x_le^{(l)}_{h-1}$, whence $\partial e_h/\partial x_l=e^{(l)}_{h-1}$ and, by induction on $h$, $e^{(l)}_h=\sum_{s=0}^{h}(-x_l)^se_{h-s}$. By the chain rule applied to the substitution $g\mapsto\overline{g}$, for $i<r$ and $h\geq1$,
\[
 \frac{\partial\overline{e_h}}{\partial u_i}=\sum_{l\in C_i}\overline{\Bigl(\frac{\partial e_h}{\partial x_l}\Bigr)}=\mu_i\sum_{s=0}^{h-1}(-u_i)^s\,\overline{e_{h-1-s}}.
\]
Evaluating at $u$ and using $\overline{e_0}=1$ and $\overline{e_h}(u)=0$ for $1\leq h<k$, we obtain
\begin{equation}\label{eq:dc}
 d_u\overline{e_h}=\sum_{i=1}^{r-1}\mu_i(-u_i)^{h-1}du_i \ \ \ (1\leq h\leq k).
\end{equation}
By Step 1 we may choose $k-1$ indices $i<r$ with pairwise distinct values $u_i$, whereby the corresponding $(k-1)\times(k-1)$ minor of the coefficient matrix $(\mu_i(-u_i)^{h-1})_{1\leq h<k,\ i<r}$ of $d_u\overline{e_1},\dots, d_u\overline{e_{k-1}}$ is the product of the nonzero scalars $\mu_i$ with a Vandermonde determinant in the distinct values $-u_i$, hence nonzero. Therefore $d_u\overline{e_1},\dots, d_u\overline{e_{k-1}}$ are linearly independent.

\noindent\textbf{Proof of \textup{(i)}.} Let $1\leq d<k$. Differentiating \eqref{eq:Gbar} at $u$, the terms involving $d_u\overline{b_d}$ are
\[
 \binom{n-h}{d-h}(d-h)(-\overline{b_d}(u))^{d-h-1}\overline{e_h}(u)\,d_u(-\overline{b_d}) \ \ \ (0\leq h<d),
\]
and they vanish for $1\leq h<d$ because $\overline{e_h}(u)=0$. The term with $h=0$ vanishes as well: for $d=1$ because $\overline{b_1}=0$, and for $d\geq2$ because \eqref{eq:Gbar} evaluated at $u$ reads $\binom{n}{d}(-\overline{b_d}(u))^{d}=0$, so that $\binom{n}{d}=0$ in $\mathbb{K}$ or $\overline{b_d}(u)=0$, and in both cases $\binom{n}{d}(-\overline{b_d}(u))^{d-1}=0$. Consequently
\begin{equation}\label{eq:dG}
 d_u\overline{\Gw{a_d}{d}}=\sum_{h=1}^{d}\binom{n-h}{d-h}(-\overline{b_d}(u))^{d-h}\,d_u\overline{e_h} \ \ \ (1\leq d<k).
\end{equation}
Thus $d_u\overline{\Gw{a_1}{1}},\dots, d_u\overline{\Gw{a_{k-1}}{k-1}}$ are obtained from $d_u\overline{e_1},\dots, d_u\overline{e_{k-1}}$ by a unitriangular matrix, and (i) follows from Step 2.

\noindent\textbf{Proof of \textup{(ii)}.} Suppose that $u_z=0$ for some $z<r$. By \eqref{eq:Gbar} evaluated at $u$, $0=\overline{\Gw{a_k}{k}}(u)=\binom{n}{k}(-\overline{b_k}(u))^{k}+\overline{e_k}(u)$, so $\binom{n}{k}\neq0$ in $\mathbb{K}$ and $\beta:=\overline{b_k}(u)\neq0$. Hence $\overline{b_k}=u_v$ for some $v<r$, and $v\neq z$. Differentiating \eqref{eq:Gbar} for $d=k$ at $u$ as in the proof of (i), now with $\overline{e_k}(u)\neq0$, gives
\[
 d_u\overline{\Gw{a_k}{k}}=\sum_{h=1}^{k}\binom{n-h}{k-h}(-\beta)^{k-h}\,d_u\overline{e_h}-k\binom{n}{k}(-\beta)^{k-1}du_v.
\]
By (i) and \eqref{eq:dG}, the span of $d_u\overline{\Gw{a_1}{1}},\dots, d_u\overline{\Gw{a_{k-1}}{k-1}}$ equals the span $L$ of $d_u\overline{e_1},\dots, d_u\overline{e_{k-1}}$, so it suffices to show $d_u\overline{\Gw{a_k}{k}}\notin L$. Restrict all linear forms to the hyperplane $H=\{\xi\in\mathbb{K}^{r-1}\mid \ \xi_v=0\}$. The restriction of $du_v$ is zero, so $d_u\overline{\Gw{a_k}{k}}|_H$ is a linear combination of $d_u\overline{e_1}|_H,\dots, d_u\overline{e_k}|_H$ in which $d_u\overline{e_k}|_H$ has coefficient $1$. By \eqref{eq:dc}, the coefficient matrix of $d_u\overline{e_1}|_H,\dots, d_u\overline{e_k}|_H$ is $(\mu_i(-u_i)^{h-1})_{1\leq h\leq k,\ i<r,\ i\neq v}$. The values $u_i$ with $i\neq v$ include $u_z=0$ and every $\gamma_l\neq\beta$, since such a $\gamma_l$ is attained at some index $i$ with $u_i\neq\beta=u_v$; these are at least $t\geq k$ distinct values, so as in Step 2 the matrix has a nonzero $k\times k$ minor and $d_u\overline{e_1}|_H,\dots, d_u\overline{e_k}|_H$ are linearly independent. Hence $d_u\overline{\Gw{a_k}{k}}|_H\notin L|_H$, so $d_u\overline{\Gw{a_k}{k}}\notin L$, which proves (ii).
\end{proof}

We can now prove the next bound on $D_p(n)$.

\begin{theorem}\label{m:nr}\label{thm:component}
Let $\bfa\in[n]^m$ be a word with $a_1=n$, let $p$ be the characteristic of $\mathbb{K}$, and let $Z\subseteq V(I_{\bfa})$ be an irreducible closed cone with $r$ coordinate classes, whose multiplicities satisfy \eqref{eq:multiplicities}; this is automatic if $p=0$ or $p>n$.
\begin{enumerate}[(i)]
\item If $r\leq2$, then $Z=\{0\}$.
\item If $D=\dim Z\geq2$, and $\kappa$ is the smallest index for which $e_\kappa\notin I(Z)$, then
\begin{equation}\label{eq:component}
 2D\leq r-\kappa.
\end{equation}
\item If $p=0$ or $p>n\geq3$, then $\dim A_{\bfa}\leq\max\{1,\lfloor(n-1)/2\rfloor-1\}$, and consequently $D_p(n)\leq\max\{2,\lfloor(n-1)/2\rfloor\}$.
\end{enumerate}
\end{theorem}
\begin{proof}
We use Notation~\ref{not:bar} and write $b_d=x_{a_d}$. The set $Z'=\lambda^{-1}(Z)\subseteq\mathbb{A}^{r-1}_{\mathbb{K}}(\mathbb{K})$ is an irreducible closed cone with $\dim Z'=\dim Z$, since $\lambda$ is an isomorphism onto $\Lambda\supseteq Z$, and $\overline{g}\in I(Z')$ if and only if $g\in I(Z)$; in particular $\overline{\Gw{a_d}{d}}\in I(Z')$ for all $d$.

(i) If $r=1$, then $x_l\in I(Z)$ for all $l$, so $Z=\{0\}$. If $r=2$, then $\overline{\Gw{a_1}{1}}=\overline{e_1}=\mu_1u_1\in I(Z')$ by \eqref{eq:pullback-gen}, and $\mu_1\neq0$ in $\mathbb{K}$ by \eqref{eq:multiplicities}, so $u_1\in I(Z')$, whence $Z'=\{0\}$ and $Z=\{0\}$.

(ii) Since $D\geq2$, the cone $Z'$ contains a nonzero point, so by the remark following \eqref{eq:pullback-gen} not all $\overline{e_h}$ lie in $I(Z')$. Hence the index $\kappa$ exists, and $\kappa\geq2$ because $e_1=\Gw{a_1}{1}\in I(Z)$. We bound $D$ in two steps.

First, choose $u\in Z'$ with $\overline{e_\kappa}(u)\neq0$. Then $u\neq0$, $\overline{\Gw{a_d}{d}}(u)=0$ for all $d$, and the index $k$ of Lemma~\ref{lem:tangent} at $u$ equals $\kappa$, since $\overline{e_h}\in I(Z')$ for $h<\kappa$. By \eqref{eq:tangent0} and Lemma~\ref{lem:tangent}(i),
\[
 D=\dim Z'\leq(r-1)-(\kappa-1)=r-\kappa.
\]

Second, we have $\kappa+D-2\leq r-2\leq m-1$ by the first step. The closed subset $Z'\cap V(\overline{e_\kappa},\dots, \overline{e_{\kappa+D-2}})$ of $Z'$ contains $0$, so it is nonempty, and by Krull's height theorem each of its irreducible components has dimension at least $D-(D-1)=1$. Hence it contains a point $u'\neq0$. At $u'$ we have $\overline{e_h}(u')=0$ for all $1\leq h\leq\kappa+D-2$, so the index $k$ of Lemma~\ref{lem:tangent} at $u'$ satisfies $k\geq\kappa+D-1$. Moreover, by \eqref{eq:Gbar} with $d=\kappa$ and $\overline{e_h}\in I(Z')$ for $1\leq h<\kappa$, we have
\[
 \binom{n}{\kappa}(-\overline{b_\kappa})^{\kappa}+\overline{e_\kappa}\in I(Z').
\]
Since $\overline{e_\kappa}\notin I(Z')$ and $I(Z')$ is prime, this forces $\binom{n}{\kappa}\neq0$ in $\mathbb{K}$ and $\overline{b_\kappa}\notin I(Z')$; as $\overline{b_\kappa}\in\{u_1,\dots, u_{r-1},0\}$, we get $\overline{b_\kappa}=u_z$ for some $z<r$. Evaluating at $u'$ gives $\binom{n}{\kappa}(-u_z(u'))^{\kappa}=-\overline{e_\kappa}(u')=0$, so $u_z(u')=0$. Therefore Lemma~\ref{lem:tangent}(ii) applies at $u'$, and with \eqref{eq:tangent0} it gives
\[
 D=\dim Z'\leq(r-1)-k\leq(r-1)-(\kappa+D-1)=r-\kappa-D,
\]
which proves \eqref{eq:component}.

(iii) Let $p=0$ or $p>n\geq3$, and let $Z$ be an irreducible component of $V(I_{\bfa})$, with $r$ coordinate classes. Every nonempty partial sum of the multiplicities of $Z$ lies in $\{1,\dots, n\}$, so \eqref{eq:multiplicities} holds. By Lemma~\ref{lem:shift}, $\Gw{a_m}{m}=\prod_{l\in[n]\setminus\{a_m\}}(x_l-x_{a_m})$ lies in the prime ideal $I(Z)$, so $x_l-x_{a_m}\in I(Z)$ for some $l\neq a_m$, i.e., $l\sim_Z a_m$; hence $r\leq n-1$. If $\dim Z\geq2$, then \eqref{eq:component} and $\kappa\geq2$ give $2\dim Z\leq n-3$. Therefore every irreducible component of $V(I_{\bfa})$ has dimension at most $\max\{1,\lfloor(n-3)/2\rfloor\}$, and $\dim A_{\bfa}=\dim V(I_{\bfa})\leq\max\{1,\lfloor(n-1)/2\rfloor-1\}$ by the Nullstellensatz. As every word may be normalized so that $a_1=n$ by Lemma~\ref{lem:relabel}, Lemma~\ref{lem:dictionary} gives $D_p(n)\leq\max\{2,\lfloor(n-1)/2\rfloor\}$.
\end{proof}
\subsubsection{The invariant \texorpdfstring{$B(n)$}{B(n)} and finiteness}\label{subsubsec:Bn}
We now combine the bounds of Sections~\ref{subsubsec:arithmeticboundn} and~\ref{subsubsec:component} into a single explicit invariant of the degree $n$. For an integer $v\geq1$ put
\[
 V(v)=\max\{2,\lfloor(v-1)/2\rfloor\},
\]
so that $D_0(v)\leq V(v)$ for all $v\geq3$ by Theorem~\ref{m:nr}(iii).

\begin{definition}\label{def:Bp}
Let $n\geq3$, let $p$ be a prime, and let $u=n/p^{v_p(n)}$ be the prime-to-$p$ part of $n$, so that $D_p(n)=D_p(u)$ by Lemma~\ref{m:support}. For each $e\geq1$ with $p^e<u$, write $u=a_ep^e+s_e$ with $a_e\geq1$ and $0\leq s_e<p^e$; here $s_e\neq0$, since $p\nmid u$. We define $B_p(n)$ to be the minimum of the following integers, each of which is an upper bound for $D_p(u)=D_p(n)$ by the result indicated:
\begin{enumerate}[(a)]
\item $\lfloor M_p(u)/2\rfloor$, by Proposition~\ref{m:halfsupport} applied in degree $u$;
\item $\max\{2,(M_p(a_e)-1)(M_p(s_e)-1)\}$ for each $e\geq1$ with $p^e<u$, by Theorem~\ref{m:blocks} applied in degree $u$ with $q=p^e$;
\item $V(u)$ if $3\leq u<p$, by Theorem~\ref{m:nr}(iii) applied in degree $u$;
\item $1$ if $u=3<p$, by Theorem~\ref{m:nr}(i), as explained below.
\end{enumerate}
The entry (a) always applies, so $B_p(n)$ is well defined, and $D_p(n)\leq B_p(n)$.
\end{definition}
To justify the entry (d), let $u=3<p$, so that $m=2$ in degree $u$. For every word $\bfa\in[3]^2$ with $a_1=3$ and every irreducible component $Z$ of $V(I_{\bfa})$, the polynomial $\Gw{a_2}{2}=\prod_{l\in[3]\setminus\{a_2\}}(x_l-x_{a_2})$ of Lemma~\ref{lem:shift} lies in the prime ideal $I(Z)$, so $Z$ has at most two coordinate classes, and $Z=\{0\}$ by Theorem~\ref{m:nr}(i), whose hypothesis \eqref{eq:multiplicities} holds because $p>3$. Hence $D_p(3)=1$ by Lemmas~\ref{lem:relabel} and~\ref{lem:dictionary}. Finally, we define
\begin{equation}\label{eq:certificate}
 B(n)=\min\bigl(\{V(n)\}\cup\{B_p(n)\mid \ p\leq n \text{ prime}\}\bigr) \ \ \ \ (n\geq3).
\end{equation}
Combining the above results yields Theorem~\ref{mainthm2}.

\begin{theorem}\label{PropNew}
For $n\ge3$, the explicit integer $B(n)$ defined in \eqref{eq:certificate} satisfies $D_0(n)\le B(n)$. Moreover, $D_0(n)\le\max\{2,n-p\}$ for any prime $n/2<p\le n$. Consequently, $D_0(n)\le n^{21/40}$ for all sufficiently large $n$.
\end{theorem}
\begin{proof}
Let $n\geq3$. Theorem~\ref{m:nr}(iii) gives $D_0(n)\leq V(n)$. For each prime $p\leq n$ we have $D_0(n)\leq D_p(n)$ by \eqref{eq:specialize} and $D_p(n)\leq B_p(n)$ by Definition~\ref{def:Bp}. Taking the minimum in \eqref{eq:certificate} gives $D_0(n)\leq B(n)$, and $B(n)\leq V(n)$ holds by definition.

Now let $p$ be a prime with $n/2<p\leq n$, and let $u$ be the prime-to-$p$ part of $n$. If $p=n$, then $u=1$, and the entry (a) of Definition~\ref{def:Bp} gives $B_p(n)\leq\lfloor M_p(1)/2\rfloor=1$, as $M_p(1)=2$. If $p<n$, then $p<n<2p$, so $p\nmid n$ and $u=n=1\cdot p+(n-p)$ with $0<n-p<p$; thus the entry (b) with $e=1$ applies, and since $M_p(1)=2$ and $M_p(n-p)=n-p+1$, as $n-p$ is a single base-$p$ digit, it gives $B_p(n)\leq\max\{2,n-p\}$. In either case
\[
 D_0(n)\leq B(n)\leq B_p(n)\leq\max\{2,n-p\}.
\]
Finally, by Baker, Harman and Pintz \cite[Theorem 1]{BHP}, for all sufficiently large $n$ there is a prime $p$ with $n-n^{21/40}\leq p\leq n$. Such $p$ satisfies $p>n/2$, whence $D_0(n)\leq\max\{2,n-p\}\leq n^{21/40}$ for all sufficiently large $n$.
\end{proof}

We now describe exactly the degrees in which $B(n)$ is small, and what Theorem~\ref{PropNew} gives in these degrees. Note that $D_0(n)\geq1$ for all $n\geq2$: for $\beta\in\mathbb{K}$, the polynomial $\rho_n(\beta,\dots, \beta)=(X+\beta)^n$ lies in $\mathfrak{X}^i_n(\mathbb{K})$ for all $1\leq i\leq n-1$, so the line $\mathbb{K}\cdot(1,\dots, 1)$ is contained in $\bigcap_{i=1}^{n-1}X^i_n(\mathbb{K})$ by Lemma~\ref{Lem4}. In the following statement, $p$ denotes a prime, $b\geq0$ and $e>i\geq1$ denote integers, and $u$ denotes the prime-to-$p$ part of $n$.

\begin{proposition}\label{maincor1}\label{prop:Bsmall}
Let $X_n$ be the arithmetic Casas-Alvero scheme of degree $n\geq3$, and let $\mathbb{K}$ be a field of characteristic zero.
\begin{enumerate}[(i)]
\item $B(n)=1$ if and only if $n=p^b$, $n=2p^b$, or $n=3p^b$ with $p\geq5$. In these degrees $D_0(n)=1$ and $X_n(\mathbb{K})=\emptyset$. In particular, $B(n)=1$ for all $3\leq n\leq11$.
\item Let $n\geq7$. Then $B(n)\leq2$ if and only if $n=p^bu$ with $p\nmid u$ and
\begin{equation}\label{eq:B2forms}
 u\in\{1,2,3,4,5,6\}\cup\{p^e+1,\ p^e+2,\ 2p^e+1\}.
\end{equation}
If $B(n)\leq2$, then $X_n(\mathbb{K})$ is finite. Consequently, $X_n(\mathbb{K})$ is finite in the degrees
\[
 n=p^b,\ 2p^b,\ 3p^b,\ 4p^b,\ 5p^b,\ 6p^b,\ p^b(p^e+1),\ p^b(p^e+2),\ p^b(2p^e+1),
\]
and for $n\geq7$ these are the only degrees in which Theorem~\ref{PropNew} yields finiteness.
\item Let $n\geq9$. Then $B(n)=3$ if and only if $n$ is not of the form described in (ii) and $n=p^bu$ with $p\nmid u$ and
\begin{equation}\label{eq:B3forms}
 u\in\{7,8\} \ \text{ with } u<p, \ \ \ \text{ or } \ \ \ u\in\{p^e+3,\ 3p^e+1\} \ \text{ with } p\geq5, \ \ \ \text{ or } \ \ \ u=p^e+p^i+1.
\end{equation}
If $B(n)=3$, then $\dim X_n(\overline{\mathbb{K}})\leq1$ for an algebraic closure $\overline{\mathbb{K}}$ of $\mathbb{K}$.
\end{enumerate}
\end{proposition}
\begin{proof}
Let $\overline{\mathbb{K}}$ be an algebraic closure of $\mathbb{K}$. By Lemma~\ref{lem:dictionary} applied over $\overline{\mathbb{K}}$ and Theorem~\ref{PropNew}, $\dim X_n(\overline{\mathbb{K}})=D_0(n)-2\leq B(n)-2$, and $X_n(\mathbb{K})\subseteq X_n(\overline{\mathbb{K}})$. Since $D_0(n)\geq1$, this gives $D_0(n)=1$ and $X_n(\overline{\mathbb{K}})=\emptyset$ if $B(n)=1$, the finiteness of $X_n(\overline{\mathbb{K}})$ if $B(n)\leq2$, and $\dim X_n(\overline{\mathbb{K}})\leq1$ if $B(n)=3$. It remains to determine the degrees with $B(n)=1$, $B(n)\leq2$ and $B(n)=3$. Throughout, $p\leq n$ is a prime, $u$ is the prime-to-$p$ part of $n$, whose last base-$p$ digit is nonzero, and the entries (a)--(d) are those of Definition~\ref{def:Bp} for the prime $p$. We shall use the following facts, which are immediate from the definition of $M_p$: for an integer $a\geq1$, $M_p(a)=2$ if and only if $a$ is a power of $p$; $M_p(a)=3$ if and only if $a=2p^j$ with $p\geq3$; and $M_p(a)=4$ if and only if $a=3p^j$ with $p\geq5$ or $a=p^j+p^{j'}$ with $j>j'\geq0$. For an integer $s\geq1$ prime to $p$, whose last digit is therefore nonzero, $M_p(s)=2,3,4$ holds if and only if $s=1$, $s=2$ with $p\geq3$, and $s=3$ with $p\geq5$ or $s=p^{i}+1$ with $i\geq1$, respectively.

(i) Since $V(n)\geq2$, we have $B(n)=1$ if and only if $B_p(n)=1$ for some $p$. The entries (b) and (c) are at least $2$; the entry (a) equals $1$ if and only if $M_p(u)\leq3$, i.e., if and only if $u=1$, or $u=2$ and $p\geq3$; and the entry (d) is present if and only if $u=3<p$. Hence $B(n)=1$ if and only if $n=p^b$, or $n=2p^b$ with $p\geq3$, or $n=3p^b$ with $p\geq5$; since $2\cdot2^b$ is a power of $2$, this is the stated condition. The integers $3\leq n\leq11$ are of this form.

(ii) Let $n\geq7$. Then $V(n)\geq3$, so $B(n)\leq2$ if and only if $B_p(n)\leq2$ for some $p$. We first determine when each entry is at most $2$.
\begin{enumerate}[(1)]
\item The entry (a) is at most $2$ if and only if $M_p(u)\leq5$, i.e., if and only if either $u$ has a single nonzero digit $c\leq4$, in which case $u=c$, or $u$ has exactly two nonzero digits, both equal to $1$, in which case $u=p^e+1$ with $e\geq1$.
\item The entries (c) and (d) are at most $2$ if and only if $3\leq u<p$ and $u\leq6$.
\item The entry (b) with exponent $e$ is at most $2$ if and only if $(M_p(a_e),M_p(s_e))$ is one of $(2,2)$, $(2,3)$ and $(3,2)$, i.e., if and only if $u=a_ep^e+s_e$ is of one of the forms $p^{e'}+1$, $p^{e'}+2$ and $2p^{e'}+1$ with $e'\geq e$.
\end{enumerate}
This shows that $B_p(n)\leq2$ forces $u$ to be of the form \eqref{eq:B2forms}. Conversely, let $u$ be of the form \eqref{eq:B2forms}. If $u\in\{1,2\}$, then $M_p(u)\leq3$ and the entry (a) equals $1$. If $u\in\{3,4,5,6\}$ and $u<p$, then the entry (c) equals $V(u)=2$. The remaining cases with $u\leq6$ are $u=3=2+1$ for $p=2$, $u=4=3+1$ for $p=3$, $u=5=2^2+1$ for $p=2$, $u=5=3+2$ for $p=3$ and $u=6=5+1$ for $p=5$, which are of the form $p^e+1$ or $p^e+2$. Finally, for $u=p^e+1$, $p^e+2$ and $2p^e+1$, the entry (b) with exponent $e$ equals $\max\{2,1\cdot1\}$, $\max\{2,1\cdot2\}$ and $\max\{2,(M_p(2)-1)\cdot1\}$ respectively, as $M_p(1)=2$ and $M_p(2)\leq3$. In all cases $B_p(n)\leq2$. The list of degrees follows, since $n=p^bu$ with $u\in\{1,\dots, 6\}$ and $p\nmid u$ means $n\in\{p^b,2p^b,\dots, 6p^b\}$, and conversely every $n=cp^b$ with $c\leq6$ can be written in this way after absorbing the powers of $p$ dividing $c$ into $p^b$.

(iii) Let $n\geq9$. Then $V(n)\geq4$, so $B(n)\leq3$ if and only if $B_p(n)\leq3$ for some $p$. We determine when each entry is at most $3$.
\begin{enumerate}[(1)]
\item The entry (a) is at most $3$ if and only if $M_p(u)\leq7$. The nonzero digits $c$ of $u$ contribute the factors $c+1\geq2$, so either $u$ has a single nonzero digit $c\leq6$, in which case $u=c\leq6$, or $u$ has exactly two nonzero digits $\{1,1\}$ or $\{1,2\}$, in which case $u\in\{p^e+1,p^e+2,2p^e+1\}$. In both cases $u$ is of the form \eqref{eq:B2forms}.
\item The entries (c) and (d) are at most $3$ if and only if $3\leq u\leq8$ and $u<p$.
\item The entry (b) with exponent $e$ is at most $3$ if and only if $(M_p(a_e)-1)(M_p(s_e)-1)\leq3$, i.e., if and only if $(M_p(a_e),M_p(s_e))$ is one of $(2,2)$, $(2,3)$, $(3,2)$, $(2,4)$ and $(4,2)$. The first three cases give the forms $p^{e'}+1$, $p^{e'}+2$ and $2p^{e'}+1$ of \eqref{eq:B2forms}. In the case $(2,4)$ we have $a_e=p^j$ and $s_e\in\{3\}$ with $p\geq5$ or $s_e=p^{i}+1$ with $1\leq i<e$, so $u=p^{e'}+3$ with $p\geq5$ or $u=p^{e'}+p^{i}+1$ with $e'=e+j>i\geq1$. In the case $(4,2)$ we have $s_e=1$ and $a_e=3p^j$ with $p\geq5$ or $a_e=p^{j}+p^{j'}$ with $j>j'\geq0$, so $u=3p^{e'}+1$ with $p\geq5$ or $u=p^{e'}+p^{i}+1$ with $e'=e+j>i=e+j'\geq1$.
\end{enumerate}
Hence, if $B_p(n)\leq3$, then $u$ is of the form \eqref{eq:B2forms} or of the form \eqref{eq:B3forms}, where in \eqref{eq:B3forms} the forms $u\in\{3,\dots, 6\}$ with $u<p$ coming from (2) have been omitted because they belong to \eqref{eq:B2forms}. Conversely, let $u$ be of the form \eqref{eq:B3forms}. If $u\in\{7,8\}$ and $u<p$, then the entry (c) equals $V(u)=3$. If $u=p^e+3$ or $u=3p^e+1$ with $p\geq5$, then the entry (b) with exponent $e$ equals $\max\{2,1\cdot(M_p(3)-1)\}=3$ or $\max\{2,(M_p(3)-1)\cdot1\}=3$, as $M_p(3)=4$. If $u=p^e+p^i+1$ with $e>i\geq1$, then $p^i<u$, $a_i=p^{e-i}+1$ and $s_i=1$, so the entry (b) with exponent $i$ equals $\max\{2,(M_p(a_i)-1)(M_p(1)-1)\}=\max\{2,3\cdot1\}=3$. In all cases $B_p(n)\leq3$. Therefore, for $n\geq9$, $B(n)\leq3$ holds if and only if $n=p^bu$ with $u$ of the form \eqref{eq:B2forms} or \eqref{eq:B3forms} for some prime $p$, and $B(n)=3$ holds if and only if, in addition, $B(n)\leq2$ fails, which by (ii) means that $n$ is not of the form described there. This proves (iii).
\end{proof}

We note that the finiteness of $X_n(\mathbb{K})$ for the degrees $n=p^b(p^e+1)$ and $p^b(p^e+2)$ also appears in \cite{Massri}, and that Theorem~\ref{PropNew} recovers Conjecture~\ref{con1} over fields of characteristic zero in the degrees $p^b$ and $2p^b$, proved in \cite{GVB}, and $3p^b$ with $p\geq5$, proved in \cite{DJ}. By Proposition~\ref{prop:Bsmall}, $B(n)\leq2$ for all $3\leq n\leq100$ except $n=70=2\cdot35$, $77=7\cdot11$ and $88=8\cdot11$, where $B(n)=3$: indeed $35=2^5+2+1$, while $7<11$ and $8<11$. The next proposition shows that, on the other hand, the degrees for which Theorem~\ref{PropNew} gives a bound below any fixed constant are rare, although they are as numerous as the primes. We write $\pi(x)$ for the number of primes $p\leq x$ and use Chebyshev's bounds $\frac{1}{6}\,x/\log x<\pi(x)<6\,x/\log x$ for $x\geq2$, see \cite[Theorem~4.6]{Apostol}.

\begin{proposition}\label{prop:density}
For every integer $C\geq1$ there is a constant $c_C>0$ such that, for all $N\geq2$,
\begin{equation}\label{eq:density}
 \pi(N)-1\leq\#\{3\leq n\leq N\mid \ B(n)\leq C\}\leq c_C\,\frac{N}{\log N}.
\end{equation}
In particular, the number of degrees $n\leq N$ with $B(n)\leq C$ lies between $N/(6\log N)-1$ and $c_CN/\log N$, the set of these degrees has natural density zero, and $B(n)$ is unbounded.
\end{proposition}
\begin{proof}
The lower bound holds because $B(p)=1\leq C$ for every prime $3\leq p\leq N$, by Proposition~\ref{prop:Bsmall}(i). For the upper bound, put $K=(C+1)^2$. We divide the proof into the following two steps.

\noindent\textbf{(i) A digit condition.} We claim that every $n>2C+2$ with $B(n)\leq C$ can be written as $n=p^bu$ with a prime $p\leq n$ and integers $b\geq0$ and $u\geq1$ such that $p\nmid u$ and $M_p(u)\leq K$. Indeed, $V(n)>C$ for $n>2C+2$, so by \eqref{eq:certificate} there is a prime $p\leq n$ with $B_p(n)\leq C$. Let $u$ be the prime-to-$p$ part of $n$ and $b=v_p(n)$, so that $n=p^bu$ and $p\nmid u$. By Definition~\ref{def:Bp}, one of the entries (a)--(d) is at most $C$. If it is (a), then $M_p(u)\leq2C+1\leq K$. If it is (c) or (d), then $u<p$ and $u\leq2C+2$, so $u$ is a single base-$p$ digit and $M_p(u)=u+1\leq K$. If it is (b) for some exponent $e$, then $(M_p(a_e)-1)(M_p(s_e)-1)\leq C$ with both factors at least $1$, so $M_p(a_e)\leq C+1$ and $M_p(s_e)\leq C+1$; since the base-$p$ digits of $a_ep^e$ and of $s_e$ occupy disjoint positions, $M_p(u)=M_p(a_e)M_p(s_e)\leq K$. This proves the claim.

\noindent\textbf{(ii) Counting.} Fix a prime $p\leq N$, and let $\ell_p=\lfloor\log_pN\rfloor+1$ be the number of base-$p$ digits of $N$. Put $J=\lfloor\log_2K\rfloor$. If $u\leq N$ satisfies $M_p(u)\leq K$, then $u$ has at most $J$ nonzero base-$p$ digits, since each nonzero digit contributes a factor at least $2$ to $M_p(u)$, and every digit $c$ of $u$ satisfies $c+1\leq K$. Choosing first the positions of the nonzero digits among the $\ell_p$ digit positions of $N$ and then their values, we see that there are at most
\[
 \sum_{j=0}^{J}\binom{\ell_p}{j}(K-1)^j\leq(J+1)(K\ell_p)^{J}
\]
such integers $u$, and for each of them there are at most $\ell_p$ integers $b\geq0$ with $p^bu\leq N$. Hence at most $(J+1)K^J\ell_p^{J+1}$ integers $n\leq N$ arise from the prime $p$ in Step 1. If $p>\sqrt{N}$, then $\ell_p\leq2$, so at most $2^{J+1}(J+1)K^J$ integers arise from $p$, and summing over the primes $p\leq N$ contributes at most $2^{J+1}(J+1)K^J\pi(N)\leq6\cdot2^{J+1}(J+1)K^J\,N/\log N$. If $p\leq\sqrt{N}$, then $\ell_p\leq\log_2N+1\leq2\log_2N$, so at most $(J+1)K^J(2\log_2N)^{J+1}$ integers arise from $p$, and summing over the at most $\sqrt{N}$ such primes contributes at most $(J+1)K^J(2\log_2N)^{J+1}\sqrt{N}$, which is at most $N/\log N$ for all $N\geq N_C$, where $N_C$ depends only on $C$. Together with the at most $2C$ integers $3\leq n\leq2C+2$, this proves the upper bound in \eqref{eq:density} for $N\geq N_C$, and then for all $N\geq2$ after enlarging $c_C$, since $\#\{3\leq n\leq N\mid \ B(n)\leq C\}\leq N\leq(\log N_C)\,N/\log N$ for $2\leq N<N_C$.

The explicit bounds follow from \eqref{eq:density} and Chebyshev's lower bound for $\pi(N)$, and the density statement follows from $N/\log N=o(N)$. Finally, suppose that $B(n)\leq C$ for all $n\geq3$. Then the middle term of \eqref{eq:density} equals $N-2$ for all $N\geq3$, so $N-2\leq c_CN/\log N$ for all $N\geq3$; but $N-2\geq N/2$ for $N\geq4$, and $N/2>c_CN/\log N$ as soon as $\log N>2c_C$, a contradiction. Hence $B(n)$ is unbounded.
\end{proof}

\subsection{Witness coordinates and regular subsequences}\label{subsec:splitting}
In this subsection we bound $\dim A_{\bfa}$ in terms of the number of distinct witness labels of the word $\bfa$, and we show that an initial segment of length about $n/2$ of the sequence $\Gw{a_1}{1},\dots, \Gw{a_m}{m}$ is always a regular sequence. We keep the notation of Section~\ref{subsubsec:component}. Throughout, $\bfa\in[n]^m$ is a word with $a_1=n$, which is no restriction by Lemma~\ref{lem:relabel}, and we write $b_d=x_{a_d}$ for $1\leq d\leq m$, so that $b_1=x_n=0$. For $1\leq r\leq m$, let
\[
 w_r=|\{a_1,\dots, a_r\}|
\]
be the number of distinct labels among $a_1,\dots, a_r$, and put $w=w_m$.

Let $l_1<\dots<l_s$ be the elements of $\{a_1,\dots, a_m\}\setminus\{n\}$, so that $s=w-1$, and put $y_a=x_{l_a}$ for $1\leq a\leq s$; we call $y_1,\dots, y_s$ the \textit{witness coordinates} of $\bfa$. Let $z_1,\dots, z_q$ be the remaining variables among $x_1,\dots, x_m$, so that $q=n-w$, $s+q=m$ and
\[
 S=\mathbb{K}[y_1,\dots, y_s,z_1,\dots, z_q], \ \ \ \ \ b_d\in\{0,y_1,\dots, y_s\} \ \text{ for all } 1\leq d\leq m.
\]
For $h\geq0$ write $e_h(y)=e_h(y_1,\dots, y_s)$ and $e_h(z)=e_h(z_1,\dots, z_q)$, with $e_h(y)=0$ for $h>s$ and $e_h(z)=0$ for $h>q$. Comparing the coefficients of $v^d$ in $\prod_{l=1}^{m}(1+x_lv)=\prod_{a=1}^{s}(1+y_av)\prod_{i=1}^{q}(1+z_iv)$ gives
\begin{equation}\label{eq:split-e}
 e_d=\sum_{h=0}^{d}e_h(y)\,e_{d-h}(z) \ \ \ (0\leq d\leq m).
\end{equation}
Finally, define $F_d,U_d\in\mathbb{K}[y_1,\dots, y_s]$ recursively by $F_0=U_0=1$ and
\begin{equation}\label{eq:recursion}
 F_d=-\sum_{h=0}^{d-1}\binom{n-h}{d-h}(-b_d)^{d-h}F_h, \ \ \ \ \ U_d=F_d-\sum_{h=1}^{d}e_h(y)\,U_{d-h} \ \ \ \ (1\leq d\leq m).
\end{equation}
Then $F_1=0$, since $b_1=0$, and $F_d$, $U_d$ are homogeneous of degree $d$ whenever they are nonzero. The following lemma explains the role of these polynomials: modulo the equations $\Gw{a_h}{h}=0$ with $h<d$, the equation $\Gw{a_d}{d}=0$ reads $e_d=F_d(y)$, and this in turn prescribes the value $e_d(z)=U_d(y)$ of the $d^{th}$ elementary symmetric polynomial of the remaining variables.

\begin{lemma}\label{lem:witness}
Let $\bfa\in[n]^m$ be a word with $a_1=n$, and let $1\leq r\leq m$.
\begin{enumerate}[(i)]
\item $(\Gw{a_1}{1},\dots, \Gw{a_r}{r})=(e_1-F_1,\dots, e_r-F_r)$ as ideals of $S$.
\item A point $x=(y,z)\in\mathbb{A}^{s}_{\mathbb{K}}(\mathbb{K})\times\mathbb{A}^{q}_{\mathbb{K}}(\mathbb{K})=\mathbb{A}^m_{\mathbb{K}}(\mathbb{K})$ lies in $V(\Gw{a_1}{1},\dots, \Gw{a_r}{r})$ if and only if
\begin{equation}\label{eq:fiber}
 e_d(z)=U_d(y) \ \text{ for } 1\leq d\leq\min\{r,q\} \ \ \text{ and } \ \ U_d(y)=0 \ \text{ for } q<d\leq r.
\end{equation}
\item Let $\pi:V(\Gw{a_1}{1},\dots, \Gw{a_r}{r})\rightarrow\mathbb{A}^s_{\mathbb{K}}(\mathbb{K})$ be the projection $(y,z)\mapsto y$. If $r\leq q$, then every fiber of $\pi$ has dimension at most $q-r$. If $r\geq q$, then every fiber of $\pi$ is finite. If $r>q$, then the image of $\pi$ is contained in $V(U_{q+1},\dots, U_r)$.
\end{enumerate}
\end{lemma}
\begin{proof}
(i) By \eqref{eq:shift-identity} with $j=a_d$ and by \eqref{eq:recursion}, we have, for $1\leq d\leq m$,
\[
 \Gw{a_d}{d}=\sum_{h=0}^{d}\binom{n-h}{d-h}(-b_d)^{d-h}e_h \ \ \text{ and } \ \ 0=\sum_{h=0}^{d}\binom{n-h}{d-h}(-b_d)^{d-h}F_h.
\]
Subtracting and using $e_0=F_0$, we obtain
\begin{equation}\label{eq:unitriangular}
 \Gw{a_d}{d}=(e_d-F_d)+\sum_{h=1}^{d-1}\binom{n-h}{d-h}(-b_d)^{d-h}(e_h-F_h) \ \ \ (1\leq d\leq m).
\end{equation}
Thus $\Gw{a_1}{1},\dots, \Gw{a_r}{r}$ are obtained from $e_1-F_1,\dots, e_r-F_r$ by a unitriangular matrix with entries in $S$, which proves (i).

(ii) By (i), $x\in V(\Gw{a_1}{1},\dots, \Gw{a_r}{r})$ if and only if $e_d(x)=F_d(y)$ for all $1\leq d\leq r$. We show by induction on $d$ that, for $0\leq d\leq r$,
\[
 e_h(x)=F_h(y) \ \text{ for all } h\leq d \ \ \ \iff \ \ \ e_h(z)=U_h(y) \ \text{ for all } h\leq d.
\]
Since $e_h(z)=0$ for $h>q$, the right-hand side reads $U_h(y)=0$ for $h>q$, so the case $d=r$ of this claim is (ii). The case $d=0$ is trivial. Let $d\geq1$ and assume that $e_h(z)=U_h(y)$ for all $h<d$. By \eqref{eq:split-e},
\[
 e_d(x)=e_d(z)+\sum_{h=1}^{d}e_h(y)e_{d-h}(z)=e_d(z)+\sum_{h=1}^{d}e_h(y)U_{d-h}(y),
\]
so by \eqref{eq:recursion} the equality $e_d(x)=F_d(y)$ holds if and only if $e_d(z)=U_d(y)$. This completes the induction.

(iii) Let $\epsilon:\mathbb{A}^q_{\mathbb{K}}\rightarrow\mathbb{A}^q_{\mathbb{K}}$ be the morphism $z\mapsto(e_1(z),\dots, e_q(z))$. It is finite, since $\mathbb{K}[z_1,\dots, z_q]$ is a finitely generated module over its subring $\mathbb{K}[e_1(z),\dots, e_q(z)]$, each $z_i$ being a root of the monic polynomial $\prod_{l=1}^{q}(T-z_l)=\sum_{h=0}^{q}(-1)^he_h(z)T^{q-h}$ with coefficients in this subring (cf.\ the Vieta map of Section~\ref{subsec3.3}). Fix $y\in\mathbb{A}^s_{\mathbb{K}}(\mathbb{K})$. If $r\leq q$, then by \eqref{eq:fiber} the fiber $\pi^{-1}(y)$ is identified with $\epsilon^{-1}(L_y)$, where $L_y=\{(U_1(y),\dots, U_r(y))\}\times\mathbb{A}^{q-r}_{\mathbb{K}}(\mathbb{K})$ is a linear subvariety of dimension $q-r$; since $\epsilon$ has finite fibers, the fiber-dimension inequality \cite[Lemma~33.20.4]{stacks-project}, applied to the restriction $\epsilon^{-1}(L_y)\rightarrow L_y$ of $\epsilon$, gives $\dim\epsilon^{-1}(L_y)\leq q-r$. If $r\geq q$, then by \eqref{eq:fiber} the fiber $\pi^{-1}(y)$ is contained in the finite set $\epsilon^{-1}(U_1(y),\dots, U_q(y))$. The last assertion is part of \eqref{eq:fiber}.
\end{proof}

By Lemma~\ref{lem:witness}(iii), the polynomials $U_{q+1},\dots, U_m$ control the dimension of $V(I_{\bfa})$. We now compute their pure-power coefficients, which turn out to be signed numbers of permutations with a prescribed descent set. For $T\subseteq[n-1]$, let $\beta_n(T)$ be the number of permutations $\sigma\in\mathfrak{S}_n$ whose descent set $\Des(\sigma)=\{i\in[n-1]\mid \ \sigma(i)>\sigma(i+1)\}$ equals $T$, and let $\alpha_n(T)$ be the number of those with $\Des(\sigma)\subseteq T$. If $T=\{t_1<\dots<t_k\}$, put $t_0=0$ and $t_{k+1}=n$. Then $\Des(\sigma)\subseteq T$ if and only if $\sigma$ is increasing on each of the intervals $[t_{i-1}+1,t_i]$, $1\leq i\leq k+1$, and $\sigma\mapsto\bigl(\sigma([t_0+1,t_1]),\dots, \sigma([t_k+1,t_{k+1}])\bigr)$ is a bijection from the set of such $\sigma$ onto the set of ordered partitions of $[n]$ into subsets of sizes $t_1-t_0,\dots, t_{k+1}-t_k$. Hence
\begin{equation}\label{eq:alpha}
 \alpha_n(T)=\binom{n}{t_1-t_0,\ t_2-t_1,\ \dots,\ t_{k+1}-t_k}=\prod_{i=1}^{k}\binom{n-t_{i-1}}{t_i-t_{i-1}}, \ \ \ \ \ \alpha_n(T)=\sum_{T'\subseteq T}\beta_n(T'),
\end{equation}
and inclusion--exclusion gives
\begin{equation}\label{eq:beta}
 \beta_n(T)=\sum_{T'\subseteq T}(-1)^{|T\setminus T'|}\alpha_n(T').
\end{equation}
Moreover $\beta_n(T)\geq1$: the permutation listing the $t_1$ largest elements of $[n]$ in increasing order, then the next $t_2-t_1$ largest elements in increasing order, and so on, has descent set exactly $T$.

\begin{theorem}\label{thm:descents}
Let $\bfa\in[n]^m$ be a word with $a_1=n$. For $1\leq a\leq s$ and $1\leq d\leq m$, let $T_a(d)=\{h\in[d]\mid \ a_h=l_a\}$ be the set of positions $h\leq d$ at which $b_h=y_a$. Then
\begin{equation}\label{eq:descent-coefficient}
 U_d(0,\dots, 0,y_a,0,\dots, 0)=(-1)^{d+|T_a(d)|}\beta_n(T_a(d))\,y_a^{d},
\end{equation}
i.e., the coefficient of $y_a^d$ in $U_d$ equals $(-1)^{d+|T_a(d)|}\beta_n(T_a(d))$. In particular, if $\mathbb{K}$ has characteristic zero and $s\geq1$, then $U_d\neq0$ for all $1\leq d\leq m$; and if $\mathbb{K}$ has characteristic zero, then
\begin{equation}\label{eq:witness-bound}
 \dim A_{\bfa}\leq\max\{0,w-2\}.
\end{equation}
\end{theorem}
\begin{proof}
Let $\pi_a:\mathbb{K}[y_1,\dots, y_s]\rightarrow\mathbb{K}[y_a]$ be the $\mathbb{K}$-algebra homomorphism with $\pi_a(y_a)=y_a$ and $\pi_a(y_{a'})=0$ for $a'\neq a$. It maps a form of degree $d$ to its coefficient of $y_a^d$ times $y_a^d$, so $\pi_a(U_d)$ is the left-hand side of \eqref{eq:descent-coefficient}. Put $T_a=T_a(m)$; then $\pi_a(b_h)=y_a$ for $h\in T_a$ and $\pi_a(b_h)=0$ for $h\notin T_a$.

We first analyze the polynomials $\pi_a(F_h)$. We claim that $\pi_a(F_h)=0$ for $h\notin T_a\cup\{0\}$ and that
\begin{equation}\label{eq:piF}
 \pi_a(F_h)=(-y_a)^h\sum_{T\subseteq T_a(h),\ h\in T}(-1)^{|T|}\alpha_n(T) \ \ \ \text{ for } h\in T_a.
\end{equation}
We argue by induction on $h\geq1$, using $\pi_a(F_0)=1$. Applying $\pi_a$ to \eqref{eq:recursion} gives
\[
 \pi_a(F_h)=-\sum_{h'=0}^{h-1}\binom{n-h'}{h-h'}(-\pi_a(b_h))^{h-h'}\pi_a(F_{h'}),
\]
which vanishes if $h\notin T_a$. Let $h\in T_a$. By the induction hypothesis, only the indices $h'\in(T_a\cap[h-1])\cup\{0\}$ contribute, and for these $\pi_a(F_{h'})=(-y_a)^{h'}\sum_{T'}(-1)^{|T'|}\alpha_n(T')$, where $T'$ ranges over the subsets of $T_a(h-1)$ with $\max(T'\cup\{0\})=h'$. For such $T'$, the product formula in \eqref{eq:alpha} gives $\binom{n-h'}{h-h'}\alpha_n(T')=\alpha_n(T'\cup\{h\})$. Therefore
\begin{align*}
 \pi_a(F_h)&=-\sum_{h'}\binom{n-h'}{h-h'}(-y_a)^{h-h'}(-y_a)^{h'}\sum_{T'}(-1)^{|T'|}\alpha_n(T')\\
 &=(-y_a)^h\sum_{T'\subseteq T_a(h-1)}(-1)^{|T'|+1}\alpha_n(T'\cup\{h\}),
\end{align*}
and the substitution $T=T'\cup\{h\}$ proves \eqref{eq:piF}.

Now, we compute the polynomials $\pi_a(U_d)$. Since $\pi_a(e_1(y))=y_a$ and $\pi_a(e_h(y))=0$ for $h\geq2$, applying $\pi_a$ to the second recursion in \eqref{eq:recursion} gives $\pi_a(U_d)=\pi_a(F_d)-y_a\pi_a(U_{d-1})$ for $d\geq1$, whence $\pi_a(U_d)=\sum_{h=0}^{d}\pi_a(F_h)(-y_a)^{d-h}$ by induction on $d$. Inserting \eqref{eq:piF} and noting that every nonempty subset $T$ of $T_a(d)$ has a maximum $h\in T_a(d)$ and satisfies $T\subseteq T_a(h)$, we obtain
\begin{align*}
 \pi_a(U_d)&=(-y_a)^d\sum_{T\subseteq T_a(d)}(-1)^{|T|}\alpha_n(T)=(-1)^{d+|T_a(d)|}y_a^d\sum_{T\subseteq T_a(d)}(-1)^{|T_a(d)\setminus T|}\alpha_n(T)\\
 &=(-1)^{d+|T_a(d)|}\beta_n(T_a(d))\,y_a^d,
\end{align*}
where the last equality is \eqref{eq:beta}. This proves \eqref{eq:descent-coefficient}.

We now prove the bound \eqref{eq:witness-bound}. Let $\mathbb{K}$ have characteristic zero. If $s\geq1$, then for every $1\leq d\leq m$ the coefficient of $y_1^d$ in $U_d$ is the nonzero integer $(-1)^{d+|T_1(d)|}\beta_n(T_1(d))$, so $U_d\neq0$; in particular $U_{q+1}\neq0$, where $q+1\leq m$ because $s\geq1$. By Lemma~\ref{lem:witness}(iii) with $r=m$, the projection $\pi:V(I_{\bfa})\rightarrow\mathbb{A}^s_{\mathbb{K}}(\mathbb{K})$ has finite fibers and its image is contained in the hypersurface $V(U_{q+1})$, which has dimension $s-1$. Applying the fiber-dimension inequality \cite[Lemma~33.20.4]{stacks-project} to the restriction of $\pi$ to each irreducible component of $V(I_{\bfa})$ gives $\dim V(I_{\bfa})\leq s-1=w-2$, and $\dim A_{\bfa}=\dim V(I_{\bfa})$ by the Nullstellensatz. If $s=0$, then $\bfa=(n,\dots, n)$ and $q=m$, so $b_d=0$ and $F_d=U_d=0$ for all $d\geq1$, and Lemma~\ref{lem:witness}(ii) with $r=m$ shows that $V(I_{\bfa})$ consists of the points $z$ with $e_1(z)=\dots=e_m(z)=0$; for such $z$ we have $\prod_{i=1}^{m}(1+z_iv)=1$ in $\mathbb{K}[v]$, so $z=0$. Hence $\dim A_{\bfa}=0$ in this case, which proves \eqref{eq:witness-bound}.
\end{proof}

\begin{theorem}\label{thm:prefix}
Assume that $\mathbb{K}$ has characteristic zero, and let $\bfa\in[n]^m$ be a word and $1\leq r\leq m$. If $r+w_r\leq n+1$, then $\Gw{a_1}{1},\dots, \Gw{a_r}{r}$ is a regular sequence in $S$. In particular,
\begin{equation}\label{eq:regular-prefix}
 r_0=\left\lfloor\frac{n+1}{2}\right\rfloor
\end{equation}
is a regular prefix length for every word: for every choice of $1\leq j_1,\dots, j_{n-1}\leq n$ in Proposition~\ref{Lem5}, the $r_0$ polynomials $\Phi^{\#}_{j_{n-1}}(HD^{n-2}_{n-1}\mathbf{x}_{n-1}),\dots, \Phi^{\#}_{j_{n-r_0}}(HD^{n-1-r_0}_{n-1}\mathbf{x}_{n-1})$ of lowest degree form a regular sequence.
\end{theorem}
\begin{proof}
By Lemma~\ref{lem:relabel}, we may assume $a_1=n$. Put $J=(\Gw{a_1}{1},\dots, \Gw{a_r}{r})\subseteq S$. We first prove that $\htop J=r$. 
Let $\bfa'=(a_1,\dots, a_r,n,\dots, n)\in[n]^m$. Then $\Gw{a'_d}{d}=\Gw{a_d}{d}$ for $d\leq r$, and $\bfa'$ has $w_r$ distinct labels, hence $s'=w_r-1$ witness coordinates and $q'=n-w_r$ remaining variables; the hypothesis $r+w_r\leq n+1$ reads $r\leq q'+1$. Let $U'_d$ be the polynomials \eqref{eq:recursion} attached to $\bfa'$, and let $\pi:V(J)\rightarrow\mathbb{A}^{s'}_{\mathbb{K}}(\mathbb{K})$ be the projection to the witness coordinates of $\bfa'$. If $r\leq q'$, then every fiber of $\pi$ has dimension at most $q'-r$ by Lemma~\ref{lem:witness}(iii) applied to $\bfa'$, and the fiber-dimension inequality \cite[Lemma~33.20.4]{stacks-project}, applied to the restriction of $\pi$ to each irreducible component of $V(J)$, gives $\dim V(J)\leq s'+(q'-r)=m-r$. If $r=q'+1$, then $w_r=n+1-r\geq2$, as $r\leq m$, so $s'\geq1$ and $U'_{q'+1}\neq0$ by Theorem~\ref{thm:descents} applied to $\bfa'$; by Lemma~\ref{lem:witness}(iii), $\pi$ has finite fibers and image contained in the hypersurface $V(U'_{q'+1})$ of dimension $s'-1$, so the fiber-dimension inequality gives $\dim V(J)\leq s'-1=m-r$. In both cases, the Nullstellensatz gives $\htop J=m-\dim V(J)\geq r$, and hence $\htop J=r$ by Krull's height theorem.

Now let $\mathfrak{m}=(x_1,\dots, x_m)\subseteq S$. Since $J$ is homogeneous, its minimal primes are homogeneous and hence contained in $\mathfrak{m}$, so that $\dim S_{\mathfrak{m}}/JS_{\mathfrak{m}}=\dim S/J=m-r$. As $S_{\mathfrak{m}}$ is a regular, hence Cohen--Macaulay, local ring of dimension $m$, the images of $\Gw{a_1}{1},\dots, \Gw{a_r}{r}$ in $S_{\mathfrak{m}}$ form a regular sequence by \cite[Lemma~10.104.2]{stacks-project}. If $\Gw{a_d}{d}$ were a zero divisor on $S/(\Gw{a_1}{1},\dots, \Gw{a_{d-1}}{d-1})$ for some $d\leq r$, it would lie in an associated prime $P$ of this homogeneous ideal; then $P$ is homogeneous, so $P\subseteq\mathfrak{m}$, and $PS_{\mathfrak{m}}$ is an associated prime of $S_{\mathfrak{m}}/(\Gw{a_1}{1},\dots, \Gw{a_{d-1}}{d-1})S_{\mathfrak{m}}$ containing $\Gw{a_d}{d}$, a contradiction. Hence $\Gw{a_1}{1},\dots, \Gw{a_r}{r}$ is a regular sequence in $S$.

Finally, since $w_r\leq r$, the hypothesis $r+w_r\leq n+1$ holds whenever $2r\leq n+1$, i.e., for all $r\leq r_0$, and $\Gw{a_d}{d}=\Phi^{\#}_{j_{n-d}}(HD^{n-1-d}_{n-1}\mathbf{x}_{n-1})$ by Notation~\ref{not:Phi}. This proves \eqref{eq:regular-prefix}.

We note, in addition, that the theorem holds for every field $\mathbb{K}$ of characteristic zero, not only for algebraically closed ones. Indeed, let $\overline{\mathbb{K}}$ be an algebraic closure of $\mathbb{K}$ and $\overline{S}=\overline{\mathbb{K}}[x_1,\dots, x_m]=S\otimes_{\mathbb{K}}\overline{\mathbb{K}}$. Since $S\rightarrow\overline{S}$ is faithfully flat, for each $d\leq r$ multiplication by $\Gw{a_d}{d}$ on $S/(\Gw{a_1}{1},\dots, \Gw{a_{d-1}}{d-1})$ is injective if and only if it is injective on $\overline{S}/(\Gw{a_1}{1},\dots, \Gw{a_{d-1}}{d-1})\overline{S}$, and $S/J\neq0$ because $J$ is generated by homogeneous polynomials of positive degree. Hence $\Gw{a_1}{1},\dots, \Gw{a_r}{r}$ is a regular sequence in $S$ as soon as it is one in $\overline{S}$, which is the case treated above.
\end{proof}
\subsection{Height bounds and the degree of finite projective fibers}\label{subsec:height}\label{subsec4.2}
We say that the word $\bfa=(a_1,\dots, a_m)$ is \textit{nonconstant} if $w\geq2$; its \textit{initial} and \textit{terminal constant blocks} are the longest initial and terminal segments $a_1=\dots=a_\ell$ and $a_{m-t+1}=\dots=a_m$ of the word, of lengths $\ell$ and $t$, so that $\ell+t\leq m$.

\begin{corollary}\label{cor:endblocks}
Assume that $\mathbb{K}$ has characteristic zero, and let $\ell$ and $t$ be the lengths of the initial and terminal constant blocks of a nonconstant word $\bfa$. Then
\begin{equation}\label{eq:endblocks}
 \dim A_{\bfa}\leq\max\left\{1,\left\lfloor\frac{n-t-\ell-1}{2}\right\rfloor\right\}.
\end{equation}
In particular, if $\ell+t\geq n-4$, i.e., if at most three entries of the word lie outside the two constant end blocks, then $\dim A_{\bfa}\leq1$, or equivalently $\htop I_{\bfa}\geq n-2$.
\end{corollary}
\begin{proof}
By Lemma~\ref{lem:relabel}, we may assume $a_1=n$, so that $a_d=n$ for $d\leq\ell$; put $a=a_m$, so that $a_d=a$ for $d>m-t$. Let $Z$ be an irreducible component of $V(I_{\bfa})$ of dimension $D\geq2$, with $r$ coordinate classes. We bound the index $\kappa$ of Theorem~\ref{m:nr}(ii) from below and $r$ from above.

By Lemma~\ref{lem:shift}, $\Gw{a_d}{d}=\Gw{n}{d}=e_d$ for $d\leq\ell$, so $e_1,\dots, e_\ell\in I(Z)$ and $\kappa\geq\ell+1$.

Next, put $\theta_l=x_l-x_a$ for $l\in[n]\setminus\{a\}$; by the definition of $\Phi^{\#}_a$ in Remark~\ref{Rem5new} and the convention $x_n=0$, we have
\[
 \Gw{a_{m-j}}{m-j}=\Phi^{\#}_a(e_{m-j})=e_{m-j}\bigl((\theta_l)_{l\in[n]\setminus\{a\}}\bigr) \ \ \ (0\leq j<t).
\]
We claim that there are distinct $l_1,\dots, l_t\in[n]\setminus\{a\}$ with $\theta_{l_1},\dots, \theta_{l_t}\in I(Z)$. Suppose that $\theta_{l_1},\dots, \theta_{l_j}\in I(Z)$ have been found for some $0\leq j<t$. Each monomial of $e_{m-j}((\theta_l)_{l\neq a})$ is a product of $m-j$ of the $m$ linear forms $\theta_l$, and the only such product not divisible by any of $\theta_{l_1},\dots, \theta_{l_j}$ is $\prod_{l\notin\{a,l_1,\dots, l_j\}}\theta_l$. Hence this product is congruent to $\Gw{a_{m-j}}{m-j}\in I(Z)$ modulo $I(Z)$, and since $I(Z)$ is prime, one of its factors $\theta_{l_{j+1}}$, with $l_{j+1}\notin\{a,l_1,\dots, l_j\}$, lies in $I(Z)$. This proves the claim, which says that $l_1,\dots, l_t\sim_Z a$; hence the coordinate class of $a$ has at least $t+1$ elements, and $r\leq n-t$.

Condition \eqref{eq:multiplicities} holds in characteristic zero, so \eqref{eq:component} gives $2D\leq r-\kappa\leq n-t-\ell-1$. Together with the components of dimension at most one, this proves \eqref{eq:endblocks}. If $\ell+t\geq n-4$, the right-hand side of \eqref{eq:endblocks} is $1$, and $\htop I_{\bfa}=m-\dim A_{\bfa}\geq n-2$.
\end{proof}

For example, in degree $n=70$ the word
\[
 (\underbrace{70,\dots,70}_{32},1,2,3,\underbrace{4,\dots,4}_{34})
\]
has five witness labels and satisfies $\dim A_{\bfa}\leq1$ by \eqref{eq:endblocks}, whereas $B(70)=3$ only gives the uniform bound $\dim A_{\bfa}\leq2$ for words of that degree.

\begin{theorem}\label{thm:degree}
Let $1\leq b\leq m-1$ and $0\leq r\leq m-b$, and suppose that $\dim A_{\bfa}\leq b$ and that $\Gw{a_1}{1},\dots,\Gw{a_r}{r}$ is a regular sequence. Then $I_{\bfa}$ contains a homogeneous regular sequence $Q_1,\dots,Q_{m-b}$ with
\[
 Q_j=\Gw{a_j}{j} \ (j\leq r), \ \ \ \deg Q_j=j+b \ (j>r), \ \ \ \deg S/(Q_1,\dots,Q_{m-b})=\frac{r!\,m!}{(r+b)!}.
\]
In particular, if $\dim A_{\bfa}\leq1$ and $\Gw{a_1}{1},\dots, \Gw{a_r}{r}$ is a regular sequence with $r\leq m-1$, then
\begin{equation}\label{eq:length-bound}
 \length\Proj A_{\bfa}\leq\frac{m!}{r+1}.
\end{equation}
\end{theorem}
\begin{proof}
Write $I_k=(\Gw{a_1}{1},\dots,\Gw{a_k}{k})$ for $0\leq k\leq m$, so that $I_m=I_{\bfa}$. We first show that $\htop I_k\geq k-b$.

Let $Z$ be an irreducible component of $V(I_k)\subseteq\mathbb{A}^m_{\mathbb{K}}(\mathbb{K})$. Intersecting the cone $Z$ with the $m-k$ hypersurfaces $V(\Gw{a_{k+1}}{k+1}),\dots, V(\Gw{a_m}{m})$ yields a nonempty cone, of dimension at least $\dim Z-(m-k)$ by Krull's height theorem, which is contained in $V(I_{\bfa})$. Since $\dim V(I_{\bfa})=\dim A_{\bfa}\leq b$, we get $\dim Z\leq m-k+b$, whence $\htop I_k=m-\dim V(I_k)\geq k-b$.

We now construct the polynomials $Q_{r+1},\dots, Q_{m-b}$. Suppose that $r<j\leq m-b$ and that $Q_1,\dots, Q_{j-1}$ have been chosen as homogeneous elements of $I_{\bfa}$ forming a regular sequence, with $Q_i=\Gw{a_i}{i}$ for $i\leq r$. Since $S/(Q_1,\dots, Q_{j-1})$ is a complete intersection, its associated primes $P$ are homogeneous of height $j-1$. As $\htop I_{j+b}\geq j>j-1$ by Step 1, no such $P$ contains $I_{j+b}$, so there is some $d\leq j+b$ with $\Gw{a_d}{d}\notin P$; also there is a linear form $L\notin P$, since $P$ is not the homogeneous maximal ideal. Hence $L^{j+b-d}\Gw{a_d}{d}$ lies in the degree-$(j+b)$ part $(I_{j+b})_{j+b}$ of $I_{j+b}$ but not in $P$. Thus the $\mathbb{K}$-vector space $(I_{j+b})_{j+b}$ is not contained in any of the finitely many associated primes $P$, and since $\mathbb{K}$ is infinite, it is not contained in their union. Any element $Q_j\in(I_{j+b})_{j+b}$ outside this union is a nonzerodivisor on $S/(Q_1,\dots, Q_{j-1})$, and $Q_j\in I_{j+b}\subseteq I_{\bfa}$ is homogeneous of degree $j+b$. This completes the induction, and the degree of the complete intersection $S/(Q_1,\dots, Q_{m-b})$ is the product of the degrees of the $Q_j$, namely $r!\prod_{j=r+1}^{m-b}(j+b)=r!\,m!/(r+b)!$.

Finally, if $\dim A_{\bfa}\leq1$, apply Steps 1 and 2 with $b=1$: the projective complete intersection $V_+(Q_1,\dots,Q_{m-1})\subseteq\mathbb{P}^{m-1}_{\mathbb{K}}$ is zero-dimensional and has length $m!/(r+1)$ by B\'ezout's theorem, and it contains $\Proj A_{\bfa}$ as a closed subscheme, since $(Q_1,\dots, Q_{m-1})\subseteq I_{\bfa}$. This proves \eqref{eq:length-bound}, including all multiplicities.
\end{proof}

As a corollary, we obtain an upper bound on $|X_n(\mathbb{K})|$, when it is finite.

\begin{corollary}\label{corexplicit}
If $D_0(n)\leq2$ and $n\geq4$, then for every field $\mathbb{K}$ of characteristic zero,
\[
 |X_n(\mathbb{K})|\leq\frac{n^{n-1}(n-1)!}{\lfloor(n+1)/2\rfloor+1}.
\]
\end{corollary}
\begin{proof}
Since $X_n(\mathbb{K})\subseteq X_n(\overline{\mathbb{K}})$ for an algebraic closure $\overline{\mathbb{K}}$ of $\mathbb{K}$, we may assume that $\mathbb{K}$ is algebraically closed. By \eqref{eq:dimCn}, every word $\bfa\in[n]^m$ satisfies $\dim A_{\bfa}\leq D_0(n)-1\leq1$, so $\Proj A_{\bfa}$ is finite or empty. By Theorem~\ref{thm:prefix}, $\Gw{a_1}{1},\dots, \Gw{a_{r_0}}{r_0}$ is a regular sequence with $r_0=\lfloor(n+1)/2\rfloor\leq m-1$, the last inequality because $n\geq4$. Hence \eqref{eq:length-bound} gives $\length\Proj A_{\bfa}\leq(n-1)!/(r_0+1)$ for each of the $n^m$ words; if a full regular sequence is known for a word, its $\Proj A_{\bfa}$ is even empty. By \eqref{Eqn8} and \eqref{eq:Gdef}, the projectivization $\mathcal{V}_n(\mathbb{K})$ of the normalized root locus is the union of the sets $\Proj A_{\bfa}(\mathbb{K})$ over all words, and $X_n(\mathbb{K})=\overline{\nu}_n(\mathcal{V}_n(\mathbb{K}))$ by Proposition~\ref{Prop:Vieta}. Summing the lengths over all words bounds the number of points of $\mathcal{V}_n(\mathbb{K})$, and hence of $X_n(\mathbb{K})$, by $n^{n-1}(n-1)!/(r_0+1)$. This proves Corollary~\ref{maincorE}.
\end{proof}

\printbibliography
\end{document}